\documentclass[journal,twoside,web]{ieeecolor}

\usepackage{generic}
\usepackage{cite}
\usepackage{amsmath,amssymb,amsfonts}
\usepackage[scr = dutchcal]{mathalpha}
\usepackage{graphicx}
\usepackage{textcomp}
\usepackage{stmaryrd}
\usepackage{tabularx}
\usepackage{multirow}

\usepackage{color}

\newcommand{\f}[2]{\frac{#1}{#2}}

\newcommand{\mbf}[1]{\mathbf{ #1}}
\newcommand{\tbf}[1]{\textbf{#1}}

\newcommand{\mcl}[1]{\mathcal{#1}}
\newcommand{\mscr}[1]{\mathscr{#1}}

\newcommand{\R}{\mathbb{R}}
\newcommand{\N}{\mathbb{N}}
\newcommand{\norm}[1]{\left\lVert{#1}\right\rVert}
\newcommand{\bmat}[1]{\begin{bmatrix}#1\end{bmatrix}}

\newcommand{\smallbmat}[1]{\left[\scriptsize\begin{smallmatrix}
		#1\end{smallmatrix} \right]}
\newcommand{\mat}[1]{\begin{matrix}#1\end{matrix}}

\newcommand{\ip}[2]{\left\langle #1, #2 \right\rangle}

\newcommand{\enn}[1]{\text{n}_{\text{#1}}}

\newtheorem{thm}{Theorem}
\newtheorem{defn}[thm]{Definition}
\newtheorem{lem}[thm]{Lemma}
\newtheorem{prop}[thm]{Proposition}
\newtheorem{cor}[thm]{Corollary}

\newtheorem{example}[thm]{Example}

\let\bl\bigl
\let\bbl\Bigl
\let\bbbl\biggl
\let\bbbbl\Biggl
\let\br\bigr
\let\bbr\Bigr
\let\bbbr\biggr
\let\bbbbr\Biggr

\title{\LARGE \bf
	$L_2$-Gain Analysis of Coupled Linear 2D PDEs using Linear PI Inequalities
}

\author{Declan S. Jagt, Matthew M. Peet %
\thanks{\tbf{Acknowledgement:} This work was supported by National Science Foundation grant CMMI-1935453.} %
}

\begin{document}

	\maketitle
	\thispagestyle{empty}
	\pagestyle{empty}

	\begin{abstract}
		
		In this paper, we present a new method for estimating the $L_2$-gain of systems governed by 2nd order linear Partial Differential Equations (PDEs) in two spatial variables, using semidefinite programming. It has previously been shown that, for any such PDE, an equivalent Partial Integral Equation (PIE) can be derived. These PIEs are expressed in terms of Partial Integral (PI) operators mapping states in $L_2[\Omega]$, and are free of the boundary and continuity constraints appearing in PDEs. In this paper, we extend the 2D PIE representation to include input and output signals in $\mathbb{R}^n$, deriving a bijective map between solutions of the PDE and the PIE, along with the necessary formulae to convert between the two representations. Next, using the algebraic properties of PI operators, we prove that an upper bound on the $L_2$-gain of PIEs can be verified by testing feasibility of a Linear PI Inequality (LPI), defined by a positivity constraint on a PI operator mapping $\mathbb{R}^n\times L_2[\Omega]$. Finally, we use positive matrices to parameterize a cone of positive PI operators on $\mathbb{R}^n\times L_2[\Omega]$, allowing feasibility of the $L_2$-gain LPI to be tested using semidefinite programming. We implement this test in the MATLAB toolbox PIETOOLS, and demonstrate that this approach allows an upper bound on the $L_2$-gain of PDEs to be estimated with little conservatism.
		
	\end{abstract}
	
	
	\section{INTRODUCTION}
	
	Physical systems are often modeled using Partial Differential Equations (PDEs), relating e.g. the temporal evolution of state variables $\mbf{u}$ to their spatial derivatives. For example, for given parameters $D$ and $\lambda$, the 2D PDE defined as
	\begin{align}\label{eq:Population_Dynamics_Intro}
	\dot{\mbf{u}}(t)&=D\bbl[\partial_x^2 \mbf{u}(t)+\partial_y^2 \mbf{u}(t)\bbr]+\lambda \mbf{u}(t)+w(t),			\nonumber\\
	z(t)&=\int_{\Omega}\mbf{u}(t,x,y)dx dy,
	\end{align}
	can be used to model the evolution of a population density $\mbf{u}(t,x,y)$ in some domain $(x,y)\in\Omega$~\cite{holmes1994partial}, where $w(t)$ is some external forcing, $z(t)$ corresponds to the total population size, and $\mbf{u}(t)$ is further constrained by boundary conditions (BCs)
	\begin{align}\label{eq:Population_BCs_Intro}
	\mbf{u}(t,x,y)&\equiv 0,		&	\forall (x,y)&\in\partial\Omega.
	\end{align}
	In analysis and control of systems such as~\eqref{eq:Population_Dynamics_Intro}, a problem that frequently arises is that of bounding the effect of the disturbances $w$ on the output $z$ of the model. 
	For example, we may wish to measure the effect of environmental conditions $w(t)$ on the growth of the population size $z(t)$.
	This effect can be quantified by the $L_2$-gain, defined as the ratio $\gamma:=\frac{\|z\|_{L_2}}{\|w\|_{L_2}}$ of the magnitude of the regulated output $z$ over that of the disturbances $w$. The $L_2$-gain provides a worst-case energy-amplification from input to output signals,  and is often used as a metric for optimilaty in control and estimation, e.g. designing controllers to minimize the effect of disturbances on the system output.

	Unfortunately, the spatial variation of the PDE state complicates efforts to compute the $L_2$-gain of systems governed by PDEs. For comparison, consider estimating the $L_2$-gain of a system governed by an Ordinary Differential Equation (ODE), written in state space representation as
	\begin{align}\label{eq:ODE_intro}
	\dot{u}(t)&=Au(t) + Bw(t),	&	u(0)&=0,	\nonumber\\
	z(t)&=Cu(t) + Dw(t).
	\end{align}
	It can be shown that the $L_2$-gain of a system of this form is bounded by a value $\gamma>0$, if there exists some positive definite \textit{storage function} $V(u)>0$ which satisfies $\dot{V}(u(t))\leq \gamma\|w(t)\|^2-\frac{1}{\gamma}\|z(t)\|^2$ along solutions $u(t)$ of the system. Parameterizing storage functions $V(u)=\ip{u}{Pu}$ using positive matrices $P>0$, this problem can be posed as the Linear Matrix Inequality (LMI) $\smallbmat{-\gamma I & D & C\\ D^T & -\gamma I & B^T P\\ C^T & PB & A^T P + PA}\leq 0$,
	which can be efficiently solved using semidefinite programming (SDP)~\cite{boyd1994linear}.
	
	However, two major issues arise when deriving a similar test for computing the $L_2$-gain of e.g. System~\eqref{eq:Population_Dynamics_Intro}. Firstly, the PDE state $\mbf{u}(t)$ at each time $t\geq 0$ exists in the space $L_2[\Omega]$ of square integrable functions on $\Omega\subseteq\R^2$, raising the question of how to parameterize the set of positive storage functions on this infinite-dimensional space. Secondly, solutions $\mbf{u}(t)$ to the system must satisfy not only the actual PDE~\eqref{eq:Population_Dynamics_Intro}, but also the BCs~\eqref{eq:Population_BCs_Intro} -- raising the challenge of enforcing the condition $\dot{V}(\mbf{u}(t))\leq \gamma\|w(t)\|^2-\frac{1}{\gamma}\|z(t)\|^2$ only along solutions $\mbf{u}(t)$ satisfying both constraints.

	To circumvent these issues associated with parameterizing storage functions for PDEs, a common approach is to approximate the PDE by a finite dimensional system -- an ODE -- using e.g. a basis function expansion~\cite{el2003analysis}. 
	However, properties such as $L_2$-gain bounds estimated for the resulting ODE may not accurately reflect those of the original system -- necessitating a posteriori error bounding methods to obtain provably valid gains. Moreover, a large number of ODE state variables may be required to obtain accurate results,
	growing exponentially with the number of spatial variables in the PDE. As a result, although ODE-based input-output analysis can be efficiently performed for certain 2D systems~\cite{lieu2013L2_gain_Couette_flow,jovanovic2021io_flow_control}, it is computationally intractable for more general 2D PDEs.

	Other methods for testing input-output properties of 2D PDEs without relying on finite-dimensional approximations are generally limited in their application. For example, in~\cite{selivanov2018Hinfty_filtering,selivanov2019Hinfty_control}, LMIs for $H_{\infty}$ filtering and control of diffusive systems are derived, using a storage function of the form $V(\mbf{u})=\|\mbf{u}\|_{L_2}^2+\ip{\nabla\mbf{u} }{P\nabla\mbf{u}}_{L_2}$, parameterized by a positive matrix $P>0$. 
	Similarly, in~\cite{ahmadi2018framework}, polynomial constraints $N(x,y)\leq 0$ are proposed for testing input-output properties of wall-bounded shear flows, also parameterizing a storage function $V(\mbf{u})=\f{1}{2}\ip{\mbf{u}}{Q\mbf{u}}_{L_2}$ by a positive matrix $Q>0$. However, the $L_2$-gain test obtained in each study is valid only for a particular type of PDE with a particular set of BCs.
	Moreover, by parameterizing storage functions merely by matrices, the proposed methods introduce significant conservatism.

	As an alternative to the aforementioned approaches, in this paper, we propose an SDP-based method for computing an upper bound on the $L_2$-gain for a general class of 2nd order, linear, 2D PDEs. Specifically, we focus on PDEs of the form,
	\begin{align}\label{eq:PDE_intro}
	&\begin{array}{lr}
	\dot{\mbf{u}}(t)=\sum_{i,j=0}^{2}A_{i,j}\partial_x^i\partial_y^j \mbf{u}(t)+Bw(t),	&	\hspace*{0.25cm}\mat{\mbf{u}(0)=\mbf{0},\\\mbf{u}(t)\in X,}
	\end{array}	\nonumber\\
	&\begin{array}{lr}
	z(t)=\int_{\Omega}\bbl(\sum_{i,j=0}^{2}C_{i,j}\partial_{x}^{i}\partial_{y}^{j}\mbf{u}(t)\bbr)dxdy + Dw(t),	& 
	\end{array}
	\end{align}
	where $X\subseteq L_2[\Omega]$ is defined by a set of well-posed (non-periodic) BCs. To derive an $L_2$-gain test for systems of this form, we adopt the approach presented in~\cite{shivakumar2019PIE_io_PDE},
	wherein an alternative representation of 1D PDEs as Partial Integral Equations (PIEs) is used. In particular, the authors prove that for any linear, 1D PDE, with sufficiently well-posed BCs $\mbf{u}(t)\in X$, there exists an equivalent PIE representation,
	\begin{align}\label{eq:PIE_intro}
	\mcl{T}\dot{\mbf{v}}(t)&=\mcl{A}\mbf{v}(t) + \mcl{B}w(t),	&	\mbf{v}(0)=\mbf{0},	\nonumber\\
	z(t)&=\mcl{C}\mbf{v}(t) + \mcl{D} w(t),
	\end{align}
	such that a function $\mbf{v}\in L_2[\Omega]$ is a solution to the PIE 	if and only if $\mcl{T}\mbf{v}\in X$ is a solution to the PDE. 
	In this representation, the operators $\{\mcl{T},\mcl{A},\mcl{B},\mcl{C},\mcl{D}\}$ are all Partial Integral (PI) operators: a class of operators that form a *-algebra, with analytic expressions for addition, multiplication, etc.. Quadratic storage functions $V(\mbf{v})=\ip{\mcl{T}\mbf{v}}{\mcl{P}\mcl{T}\mbf{v}}$ can then be parameterized by PI operators $\mcl{P}>0$, offering substantially more freedom than parameterizing by matrices. Moreover, the \textit{fundamental state} $\mbf{v}\in L_2[\Omega]$ in the PIE representation is free of the BCs imposed upon the the PDE state $\mbf{u}\in X$, allowing negativity conditions on the derivative $\dot{V}(\mbf{v}(t))$ to be readily enforced. In this manner, the authors are able to derive a Linear PI Inequality (LPI), $\mcl{Q}(\gamma)=\smallbmat{-\gamma I & \mcl{D} & \mcl{C}\\ \mcl{D}^T & -\gamma I & \mcl{B}^T \mcl{P}\\ \mcl{C}^T & \mcl{P}\mcl{B} & \mcl{A}^T \mcl{P}\mcl{T} + \mcl{T}^T\mcl{P}\mcl{A}}\leq 0$,
	for verifying an upper bound $\gamma$ on the $L_2$-gain of the PIE. Parameterizing a cone of positive PI operators by positive matrices, the authors then pose this LPI as an SDP, allowing problems of $L_2$-gain analysis of 1D PDEs to be efficiently solved~\cite{shivakumar2020PIE_duality,das2020PIE_robust_ODE_PDE,das2019PIE_estimation}.
	
	However, despite a PIE framework having recently been introduced for 2D PDEs~\cite{jagt2021PIEArxiv}, deriving an SDP test for bounding the $L_2$-gain of general systems of the form~\eqref{eq:PDE_intro} still offers several challenges.
	In particular, although a map $\mcl{T}:L_2[\Omega]\rightarrow X$ from the fundamental state space to the PDE domain has been derived for atonomous systems, this map may not be valid when disturbances $w$ are included -- presenting the problem of incorporating these disturbances in the PIE to PDE state conversion.
	In addition, a framework for converting 2D PDEs with inputs and outputs to PIEs is not yet available, still requiring formulae for computing the appropriate operators $\{\mcl{B},\mcl{C},\mcl{D}\}$ to be derived.
	Finally, posing the LPI~$\mcl{Q}(\gamma)\leq 0$ for testing the $L_2$-gain as an SDP requires parameterizing PI operators on a coupled space $\R^{n_1}\times\R^{n_2}\times L_2^{n_3}[\Omega]$, raising the challenge of performing such a parameterization for PI operators in 2D.

	In the remainder of this paper, we carefully detail how we have overcome each of these challenges in deriving and implementing an SDP test for $L_2$-gain analysis of 2D PDEs. In particular, in Section~\ref{sec:LPI_L2_Gain}, we first present an LPI for testing the $L_2$-gain of 2D PIEs, proving that this gain is bounded by $\gamma$ if there exists some positive definite 2D-PI operator $\mcl{P}:L_2^{n_2}\rightarrow L_2^{n_2}$ such that an associated operator $\mcl{Q}(\gamma,\mcl{P}):\R^{n_1}\times L_2^{n_2}\rightarrow \R^{n_1}\times L_2^{n_2}$ is negative semidefinite. In Section~\ref{sec:Representation_of_ioSystems}, we then show that a PIE representation can be derived for any linear, 2nd order 2D PDE, defining operators $\mcl{T}_0:L_2^{n_v}\rightarrow L_2^{n_v}$ and $\mcl{T}_1:\R^{n_w}\rightarrow L_2^{n_v}$ such that for a disturbance $w\in\R^{n_w}$, a function $\mbf{v}\in L_2^{n_v}$ solves the PIE if and only if $\mcl{T}_0\mbf{v}+\mcl{T}_1 w$ solves the PDE. Finally, in Section~\ref{sec:Positive_PI_Params}, we parameterize a cone of positive PI operators $\Pi:\R^{n_1}\times L_2^{n_2}\rightarrow \R^{n_1}\times L_2^{n_2}$ by positive matrices, allowing feasibility of the $L_2$-gain LPI to be posed as an SDP. This result is formulated in Section~\ref{sec:L2_gain_PDE}, and numerical tests are presented in Section~\ref{sec:Numerical_Examples}.

	\section{Preliminaries}

	\subsection{Notation}
	
	For a given domain $\Omega\subset\R^d$, let $L_2^n[\Omega]$ denote the set of $\R^n$-valued square-integrable functions on $\Omega$, where we omit the domain when clear from context. Define intervals $\Omega_{a}^{b}:=[a,b]$ and $\Omega_{c}^{d}:=[c,d]$ for spatial variables $x,y$, and let $\Omega_{ac}^{bd}:=\Omega_{a}^{b}\times\Omega_{c}^{d}$ be the corresponding 2D domain. 
	For $\text{n}\!=\!\{n_0,n_1\}\!\in\N^2$, define $\text{Z}_{1}^{\text{n}}[\Omega_{ac}^{bd}]\!:=\R^{n_0}\!\times\! L_2^{n_1}[\Omega_{a}^{b}]\!\times\! L_2^{n_1}[\Omega_{c}^{d}]$,\\ and for $\text{n}\!=\!\{n_0,n_1,n_2\}\!\in\!\N^3$, define $\text{Z}_{2}^{\text{n}}[\Omega_{ac}^{bd}]:=\R^{n_0}\times L_2^{n_1}[\Omega_{a}^{b}]\!\times L_2^{n_1}[\Omega_{c}^{d}]\times L_2^{n_2}[\Omega_{ac}^{bd}]$, where we also omit the domain when clear from context. For given $\enn{}\!\in\!\N^3$ and any\\ $\mbf{u}\!=\!\smallbmat{u_0\\\mbf{u}_x\\\mbf{u}_y\\\mbf{u}_{2}}\in\text{Z}_{2}^{\enn{}}$ and $\mbf{v}=\smallbmat{v_0\\\mbf{v}_x\\\mbf{v}_y\\\mbf{v}_{2}}\!\in\!\text{Z}_{2}^{\enn{}}$, define the inner product
	\begin{align*}
		\ip{\mbf{u}}{\mbf{v}}_{\text{Z}_{2}^{\enn{}}}\! =\! \ip{u_0}{v_0}\! +\! \ip{\mbf{u}_x}{\mbf{v}_x}_{L_2}\! +\! \ip{\mbf{u}_y}{\mbf{v}_y}_{L_2}\! +\! \ip{\mbf{u}_2}{\mbf{v}_2}_{L_2},
	\end{align*}  	
	where $\ip{.}{.}$ denotes the Euclidean inner product, and $\ip{.}{.}_{L_2}$ the standard inner product on $L_2$. For any $\alpha\in\N^{2}$, we denote $\|\alpha\|_{\infty}:=\max\{\alpha_1,\alpha_2\}$.
	Then, we define $W_{k}^n[\Omega_{ac}^{bd}]$ as a Sobolev subspace of $L_2^n[\Omega_{ac}^{bd}]$, where
	\begin{align*}
		W_{k}^n[\Omega_{ac}^{bd}]\!=\!\bl\{\mbf{v}\mid \partial_x^{\alpha_1}\partial_y^{\alpha_2}\mbf{v}\!\in\! L_2^n[\Omega_{ac}^{bd}],\ \forall \alpha_j\!\in\!\N:\! \|\alpha\|_{\infty}\!\leq k\br\}.
	\end{align*}
	As for $L_2$, we occasionally omit the domain when clear from context. For $\mbf{v}\in W_{k}^n[\Omega_{ac}^{bd}]$, we use the norm
	\begin{align*}
		\norm{\mbf{v}}_{W_k}=\sum_{\|\alpha\|_{\infty}\leq k}\norm{\partial_x^{\alpha_1}\partial_y^{\alpha_2}\mbf{v}}_{L_2}
	\end{align*}
	For $\mbf{v}\in W_k^n[\Omega_{ac}^{bd}]$, we denote the Dirac delta operators
	\[
	[\Delta_{x}^{a} \mbf{v}](y):=\mbf{v}(a,y)\quad \text{and}\quad [\Delta_{y}^{c} \mbf{v}](x):=\mbf{v}(x,c).
	\]
	For a function $N\in L_2^{n\times m}[\Omega_{ac}^{bd}]$, and any $\mbf{v}\in L_2^{m}[\Omega_{ac}^{bd}]$,  
	we define the multiplier operator $\text{M}$ and integral operator $\smallint$ as
	\begin{align*}
		(\text{M}[N]\mbf{v})(x,y)&:=N(x,y)\mbf{v}(x,y), \\
		\left(\smallint_{\Omega_{ac}^{bd}}[N]\mbf{v}\right)&\ :=\int_{a}^{b}\int_{c}^{d}N(x,y)\mbf{v}(x,y)dy dx.
	\end{align*}

	\subsection{Algebras of PI Operators on 2D}\label{sec:subsec:2D_PI}

	Partial integral (PI) operators are bounded, linear operators, parameterized by square integrable functions. In 2D, we distinguish PI operators defined by parameters in the spaces $\mcl{N}_{011}$, $\mcl{N}_{2D}$ and $\mcl{N}_{0112}$, mapping different function spaces as presented in Table~\ref{tab:PI_Nspaces}. We outline the definition of the associated PI operators in this subsection, referring to~\cite{jagt2021PIEArxiv} for more details.
	\smallskip

	\begin{defn}[011-PI Operators, $\Pi_{011}$]
		For any\\
		$\text{m}:=\{m_0,m_1\}\in\N^2$ and $\text{n}:=\{n_0,n_1\}\in\N^2$, let
		{\small
			\begin{align*}
				\mcl{N}_{011}^{\text{n}\times\text{m}}[\Omega_{ac}^{bd}]:=\bmat{\R^{n_0\times m_0}   &L_2^{n_0\times m_1}[\Omega_{a}^{b}] &L_2^{n_0\times m_1}[\Omega_{c}^{d}]\\
					L_2^{n_1\times m_0}[\Omega_{a}^{b}] &\mcl{N}_{1D}^{n_1\times m_1}[\Omega_{a}^{b}]   &L_2^{n_1\times m_1}[\Omega_{ac}^{bd}]\\
					L_2^{n_1\times m_0}[\Omega_{c}^{d}] &L_2^{n_1\times m_1}[\Omega_{ac}^{bd}]   &\mcl{N}_{1D}^{n_1\times m_1}[\Omega_{c}^{d}]}, \\[-1.8em]
			\end{align*}
		}
		where
		{\small
			\begin{align*}
				\mcl{N}_{1D}^{n\times m}[\Omega_{a}^{b}]=L_2^{n\times m}[\Omega_{a}^{b}] \times L_2^{n\times m}[\Omega_{a}^{b}\!\times\!\Omega_{a}^{b}] \times L_2^{n\times m}[\Omega_{a}^{b}\!\times\!\Omega_{a}^{b}].
			\end{align*}
		}
		Then, for given parameters $B:=\smallbmat{B_{00}&B_{01}&B_{02}\\B_{10}&B_{11}&B_{12}\\B_{20}&B_{21}&B_{22}}\in \mcl{N}_{011}^{\text{n}\times\text{m}}$, we define the associated 011-PI operator $\mcl{P}[B]:\text{Z}_{1}^{\text{m}}\rightarrow \text{Z}_{1}^{\text{n}}$ as \vspace*{-0.2cm}
		{\small
			\begin{align*}
				\mcl P[B]:=\bmat{
					B_{00} &\smallint_{\Omega_{a}^{b}} [B_{01}]  &\smallint_{\Omega_{c}^{d}} [B_{02}]  \\
					\text{M}[B_{10}] &\mcl{P}[B_{11}]        &\smallint_{\Omega_{c}^{d}}[B_{12}]\\
					\text{M}[B_{20}] &\smallint_{\Omega_{a}^{b}}[B_{21}]   &\mcl{P}[B_{22}]}.
			\end{align*}
		}
		where for $N:=\{N_0,N_1,N_2\}\in\mcl{N}_{1D}^{n\times m}[\Omega_{a}^{b}]$ and any $\mbf{v}\in L_2^{m}[\Omega_{a}^{b}]$, we define \vspace*{-0.2cm}
		{\small
			\begin{align*}
				\bl(\mcl{P}[N]\mbf{v}\br)(x)&= N_0(x)\mbf{v}(x) + \int_{a}^{x}\! N_1(x,\theta)\mbf{v}(\theta)d\theta 	\\[-0.2em]
				&\qquad+ \int_{x}^{b}\! N_2(x,\theta)\mbf{v}(\theta)d\theta.\\[-1.8em]
			\end{align*}
		}	
		We denote the set of 011-PI operators as $\Pi_{011}^{\text{n}\times\text{m}}$, so that $\mcl{P}\in \Pi_{011}^{\text{n}\times\text{m}}$ if and only if $\mcl{P}=\mcl{P}[B]$ for some $B\in\mcl{N}_{011}^{\text{n}\times\text{m}}$.
	\end{defn}
	
	\medskip
	\begin{defn}[2D-PI Operators, $\Pi_{2D}$]
		For any $m,n\in\N$, let \vspace*{-0.2cm}
		{\small
			\begin{align*}
				&\quad\mcl{N}_{2D}^{n\times m}[\Omega_{ac}^{bd}]\ :=	\\
				&\quad \!\!
				\left[\!\!\begin{array}{lll}
					L_{2}^{n \times m}[\Omega_{ac}^{bd}]           
					\!&\! L_{2}^{n \times m}[\Omega_{ac}^{bd}\!\times\!\Omega_{c}^{d}]    
					\!&\! L_{2}^{n \times m}[\Omega_{ac}^{bd}\!\times\!\Omega_{c}^{d}]\\
					L_{2}^{n \times m}[\Omega_{ac}^{bd}\!\times\!\Omega_{a}^{b}]    
					\!&\! L_2^{n\times m}[\Omega_{ac}^{bd}\!\times\!\Omega_{ac}^{bd}]
					\!&\! L_2^{n\times m}[\Omega_{ac}^{bd}\!\times\!\Omega_{ac}^{bd}]\\
					L_{2}^{n \times m}[\Omega_{ac}^{bd}\!\times\!\Omega_{a}^{b}]    
					\!&\! L_2^{n\times m}[\Omega_{ac}^{bd}\!\times\!\Omega_{ac}^{bd}]
					\!&\! L_2^{n\times m}[\Omega_{ac}^{bd}\!\times\!\Omega_{ac}^{bd}]
				\end{array}\!\right].
			\end{align*}
		}
		Then, for given parameters $N:=\smallbmat{N_{00}&N_{01}&N_{02}\\N_{10}&N_{11}&N_{12}\\N_{20}&N_{21}&N_{22}}\in\mcl{N}_{2D}^{n\times m}$, we define the associated 2D-PI operator $\mcl{P}[N]:L_2^{m}[\Omega_{ac}^{bd}]\rightarrow L_2^{n}[\Omega_{ac}^{bd}]$ such that, for any $\mbf{v}\in L_2^{m}[\Omega_{ac}^{bd}]$,
		{\small
			\begin{align*}
				&(\mcl P[N]\mbf{v})(x,y):=N_{00}(x,y)\mbf{v}(x,y)\\
				& +\int_{a}^{x}N_{10}(x,y,\theta)\mbf{v}(\theta,y)d\theta +\int_{x}^{b}N_{20}(x,y,\theta)\mbf{v}(\theta,y)d\theta \nonumber\\ &\quad+\int_{c}^{y}N_{01}(x,y,\nu)\mbf{v}(x,\nu)d\nu +\int_{y}^{d}N_{02}(x,y,\nu)\mbf{v}(x,\nu)d\nu \nonumber\\
				&\qquad+\int_{a}^{x}\int_{c}^{y}N_{11}(x,y,\theta,\nu)\mbf{v}(\theta,\nu)d\nu d\theta \nonumber\\
				&\qquad\qquad+\int_{x}^{b}\int_{c}^{y}N_{21}(x,y,\theta,\nu)\mbf{v}(\theta,\nu)d\nu d\theta \nonumber\\
				&\qquad\qquad\qquad+\int_{a}^{x}\int_{y}^{d}N_{12}(x,y,\theta,\nu)\mbf{v}(\theta,\nu)d\nu d\theta \nonumber\\
				&\qquad\qquad\qquad\qquad+\int_{x}^{b}\int_{y}^{d}N_{22}(x,y,\theta,\nu)\mbf{v}(\theta,\nu)d\nu d\theta.	\nonumber
			\end{align*}
		}
		We denote the set of 2D-PI operators as $\Pi_{2D}^{n\times m}$, so that $\mcl{P}\in \Pi_{2D}^{n\times m}$ if and only if $\mcl{P}=\mcl{P}[N]$ for some $N\in\mcl{N}_{2D}^{n\times m}$.
	\end{defn}

	\begin{table}[!t]
		\renewcommand{\arraystretch}{1.45}
		\begin{tabular}{p{1.5cm}p{1.5cm}||p{1.00cm} p{0.25cm} p{2.0cm}}
			\multicolumn{2}{l||}{PI operator parameter}
			& \multicolumn{3}{l}{Function spaces corresponding to}   \\[-0.3em]
			\multicolumn{2}{l||}{space $\mcl{N}$}
			& \multicolumn{3}{l}{PI operator $\mcl{P}[N]$, for $N\in\mcl{N}$}
			\\\hline\hline
			$\mcl{N}_{2D}^{n\times m}$ 
			&\hspace*{-0.5cm}  $n,m\in\N$
			&\hspace*{-0.0cm}
			$L_2^m[\Omega_{ac}^{bd}]$ & $\rightarrow$ & $L_2^n[\Omega_{ac}^{bd}]$
			\\
			$\mcl{N}_{011}^{\text{n}\times\text{m}}$ 
			&\hspace*{-0.5cm} $\text{n},\text{m}\in\N^{2}$
			&\hspace*{-0.0cm}
			$\text{Z}_{1}^{\text{m}}[\Omega_{ac}^{bd}]$ & $\rightarrow$ & $\text{Z}_{1}^{\text{n}}[\Omega_{ac}^{bd}]$
			\\
			$\mcl{N}_{2D\leftarrow 011}^{n\times\text{m}}$ 
			&\hspace*{-0.5cm} $n\in\N,\ \text{m}\in\N^{2}$
			&\hspace*{-0.0cm}
			$\text{Z}_{1}^{\text{m}}[\Omega_{ac}^{bd}]$ & $\rightarrow$ & $L_2^n[\Omega_{ac}^{bd}]$
			\\
			$\mcl{N}_{011\leftarrow 2D}^{\text{n}\times m}$ 
			&\hspace*{-0.5cm} $\text{n}\in\N^2,\ m\in\N$
			&\hspace*{-0.0cm}
			$L_2^m[\Omega_{ac}^{bd}]$ & $\rightarrow$ & $\text{Z}_{1}^{\text{n}}[\Omega_{ac}^{bd}]$
			\\
			$\mcl{N}_{0112}^{\text{n}\times\text{m}}$ 
			&\hspace*{-0.5cm} $\text{n},\text{m}\in\N^{3}$
			&\hspace*{-0.cm}
			$\text{Z}_{2}^{\text{m}}[\Omega_{ac}^{bd}]$ & $\rightarrow$ & $\text{Z}_{2}^{\text{n}}[\Omega_{ac}^{bd}]$
		\end{tabular}
		\caption{\vspace*{-0.1cm} Function spaces associated to PI operator parameter spaces introduced in Subsection~\ref{sec:subsec:2D_PI}
		}
		\label{tab:PI_Nspaces}
	\end{table}

	\begin{defn}[0112-PI Operators, $\Pi_{0112}$]
		For any \\
		$\text{m}:=\{m_0,m_1,m_2\}\in\N^3$ and $\text{n}:=\{n_0,n_1,n_2\}\in\N^3$, let
		\vspace*{-0.05cm}
		{
			\begin{align*}
				\mcl{N}_{0112}^{\text{n}\times\text{m}}&:=
				\left[\!\begin{array}{ll}
					\mcl{N}_{011}^{\tilde{\text{n}}\times\tilde{\text{m}}}[\Omega_{ac}^{bd}] 
					&
					\mcl{N}_{011\leftarrow 2D}^{\tilde{\text{n}}\times m_2}[\Omega_{ac}^{bd}]
					\\[0.25em]
					\mcl{N}_{2D\leftarrow 011}^{n_2\times\tilde{\text{m}}}[\Omega_{ac}^{bd}]
					&\mcl{N}_{2D}^{n_2\times n_2}[\Omega_{ac}^{bd}]
				\end{array}\!\right],
			\end{align*}
		}
		where $\tilde{\text{n}}:=\{n_0,n_1\}$, $\tilde{\text{m}}:=\{m_0,m_1\}$, and
		{\small
			\begin{align*}
				\mcl{N}_{2D\leftarrow 011}^{n_2\times\tilde{\text{m}}}&\!:=\!\!
				\left[\!\!\begin{array}{lll}
					L_2^{n_2\times m_0}[\Omega_{ac}^{bd}] \\
					\mcl{N}^{n_2\times m_1}_{1D}[\Omega_{ac}^{bd}] \\
					\mcl{N}^{n_2\times m_1}_{1D}[\Omega_{ca}^{db}]
				\end{array}\!\!\right],	
				&
				\mcl{N}_{011\leftarrow 2D}^{\tilde{\text{n}}\times m_2}&\!:=\!\!
				\left[\!\!\begin{array}{l}
					L_2^{n_0\times m_2}[\Omega_{ac}^{bd}]\\ 
					\mcl{N}^{n_1\times m_2}_{1D}[\Omega_{ac}^{bd}]\\
					\mcl{N}^{n_1\times m_2}_{1D}[\Omega_{ca}^{db}]
				\end{array}\!\!\right] \\[-1.7em]
			\end{align*}
		}
		with
		{\small
			\begin{align*}
				\mcl{N}_{1D}^{n\times m}[\Omega_{ac}^{bd}]\!:=L_2^{n\times m}[\Omega_{ac}^{bd}] \!\times \!
				L_2^{n\times m}[\Omega_{ac}^{bd}\!\times\!\Omega_{a}^{b}] \!\times\!
				L_2^{n\times m}[\Omega_{ac}^{bd}\!\times\!\Omega_{a}^{b}].
			\end{align*}
		}
		Then, for given parameters $G={\scriptsize\bmat{B&C_1\\C_2&N}}\in
		\mcl{N}_{0112}^{\text{n}\times\text{m}}$,
		where
		\begin{align*}
			C_2&:=\smallbmat{C_{30}\\ C_{31}\\ C_{32}}\in\mcl{N}_{2D\leftarrow 011}^{n_2\times\tilde{\text{m}}},	&
			&\text{and,} 	&
			C_1&:=\smallbmat{C_{03}\\C_{13}\\C_{23}}\in\mcl{N}_{011\leftarrow 2D}^{\tilde{\text{n}}\times m_2} 
		\end{align*}
		we define the associated 0112-PI operator $\mcl{P}[G]:\text{Z}_{2}^{\text{m}}\rightarrow \text{Z}_{2}^{\text{n}}$ as \vspace*{-0.20cm}
		{\small
			\begin{align*}
				\mcl{P}[G]=
				\left[\!\!\!\!\begin{array}{cc}
					\mcl{P}[B] &
					\mcl{P}[C_1]		\\
					\mcl{P}[C_2]	&	
					\mcl{P}[N]
				\end{array}\!\!\!\right],
			\end{align*}
		}
		where for $E=\smallbmat{E_0\\E_1\\E_2}\in\mcl{N}_{2D\leftarrow 011}^{n\times\text{m}}$ and $D=\smallbmat{D_0\\D_1\\D_2}\in\mcl{N}_{011\leftarrow 2D}^{\text{n}\times m}$ with $\text{m},\text{n}\in\N^2$ we define \\[-1.6em]
		{\small
			\begin{align*}
				\mcl{P}[E]&\!=\!\left[\!\!\begin{array}{lll}
					\text{M}[E_{0}]  \!\!&\! \mcl{P}[E_{1}] \!&\!\! \mcl{P}[E_{2}]
				\end{array}\!\!\right],	&
				\mcl{P}[D]&\!=\!\left[\!\!
				\begin{array}{r}
					\smallint_{\Omega_{ac}^{bd}}[D_{0}]\\
					\smallint_{\Omega_{c}^{d}}[I] \circ \mcl{P}[D_{1}] \\
					\smallint_{\Omega_{a}^{b}}[I] \circ \mcl{P}[D_{2}]
				\end{array}\!\!	\right]\!,
			\end{align*}
		}
		where for $R:=\{R_0,R_1,R_2\}\in\mcl{N}^{n\times m}_{1D}[\Omega_{ac}^{bd}]$, we define
		{\small
			\begin{align*}
				&(\mcl{P}[R]\mbf{v})(x,y):=R_{0}(x,y)\mbf{v}(x,y) +\int_{a}^{x}R_{1}(x,y,\theta)\mbf{v}(\theta,y)d\theta\\[-0.2em]
				&\hspace*{2.75cm}
				+\int_{x}^{b}R_{2}(x,y,\theta)\mbf{v}(\theta,y)d\theta, \\[-2.0em]
			\end{align*}	
		}
		for arbitrary $\mbf{v}\in L_2^{m}[\Omega_{ac}^{bd}]$. We denote the set of 0112-PI operators as $\Pi_{0112}^{\text{n}\times \text{m}}$, so that $\mcl{P}\in \Pi_{0112}^{\text{n}\times \text{m}}$ if and only if $\mcl{P}=\mcl{P}[G]$ for some $G\in\mcl{N}_{0112}^{\text{n}\times\text{m}}$.
	\end{defn}

	\subsection{Properties of PI Operators}\label{sec:subsec:PI_properties}
	
	In~\cite{jagt2021PIEArxiv}, it was shown that the set of 0112-PI operators $\Pi_{0112}^{\text{n}\times\text{m}}$ forms a *-algebra, with several useful properties. 
	We summarize a few of these properties below, referring to~\cite{jagt2021PIEArxiv} for more details and a proof of each result.

	\subsubsection{The sum of 0112-PI operators is a 0112-PI operator}\label{enum:PIprop_sum}
	
	\begin{prop}
		For any $\mcl{Q},\mcl{R}\in\Pi_{0112}^{\text{n}\times\text{m}}$ with $\text{n},\text{m}\in\N^3$, there exists a unique $\mcl{P}\in\Pi_{0112}^{\text{n}\times\text{m}}$ such that $\mcl{P}=\mcl{Q}+\mcl{R}$. 
	\end{prop}
	We denote the associated parameter map as $\mcl{L}_{+}:\mcl{N}^{\text{n}\times\text{m}}_{0112}\times \mcl{N}^{\text{n}\times\text{m}}_{0112}\rightarrow \mcl{N}^{\text{n}\times\text{m}}_{0112}$, so that, for any $Q,R\in\mcl{N}^{\text{n}\times\text{m}}_{0112}$,
	\begin{align*}
		\mcl{P}[P]&=\mcl{P}[Q]+\mcl{P}[R],	&
		&\text{if and only if }\	P=\mcl{L}_{+}(Q,R).
	\end{align*}
	
	\subsubsection{The product of 0112-PI operators is a 0112-PI \\ operator}\label{enum:PIprop_prod}
	
	\begin{prop}
		For any $\mcl{Q}\in\Pi_{0112}^{\text{n}\times\text{p}}$ and $\mcl{R}\in\Pi_{0112}^{\text{p}\times\text{m}}$ with $\text{n},\text{p},\text{m}\in\N^3$, there exists a unique $\mcl{P}\in\Pi_{0112}^{\text{n}\times\text{m}}$ such that $\mcl{P}=\mcl{Q}\mcl{R}$. 
	\end{prop}
	We denote the associated parameter map as $\mcl{L}_{\times}:\mcl{N}^{\text{n}\times\text{p}}_{0112}\times \mcl{N}^{\text{p}\times\text{m}}_{0112}\rightarrow \mcl{N}^{\text{n}\times\text{m}}_{0112}$, so that, for any $Q\in\mcl{N}^{\text{n}\times\text{p}}_{0112}$ and $R\in\mcl{N}^{\text{p}\times\text{m}}_{0112}$,
	\begin{align*}
		\mcl{P}[P]&=\mcl{P}[Q]\mcl{P}[R],	&
		&\text{if and only if }\	P=\mcl{L}_{\times}(Q,R).
	\end{align*}
	
	\subsubsection{The inverse of a suitable 011-PI operator is a 011-PI operator}\label{enum:PIprop_inv}
	
	\begin{prop}
		For any $\mcl{R}\in\Pi_{011}^{\text{n}\times\text{n}}$ with $\text{n}\in\N^2$, satisfying the conditions of Lemma~5 in~\cite{jagt2021PIEArxiv}, there exists a unique $\hat{\mcl{R}}\in\Pi_{011}^{\text{n}\times\text{n}}$ such that $\hat{\mcl{R}}\mcl{R}=\mcl{R}\hat{\mcl{R}}=I$.
	\end{prop}
	We denote the associated parameter map as $\mcl{L}_{\text{inv}}:\mcl{N}^{\text{n}\times\text{n}}_{011}\rightarrow \mcl{N}^{\text{n}\times\text{n}}_{011}$, so that, for any $R\in\mcl{N}^{\text{n}\times\text{n}}_{011}$ as in Lemma~5 in~\cite{jagt2021PIEArxiv},
	\begin{align*}
		\mcl{P}[\hat{R}]\mcl{P}[R]&=I,	&
		&\text{if and only if }\	\hat{R}=\mcl{L}_{\text{inv}}(R).
	\end{align*}
	
	\subsubsection{The composition of a differential operator with a suitable 2D-PI operator is a 2D-PI operator}\label{enum:PIprop_dif}
	
	We refer to Lemmas~6 and~7 in~\cite{jagt2021PIEArxiv} for more information.
	
	\subsubsection{The adjoint of a 2D-PI operator is a 2D-PI operator}\label{enum:PIprop_adj}
	
	Here we define the adjoint of a PI operator $\mcl{P}\in\Pi_{0112}^{\text{n}\times\text{m}}$, as the unique operator $\mcl{P}^*\in\Pi_{0112}^{\text{m}\times\text{n}}$ that satisfies
	\begin{align*}
		\ip{\mbf{v}}{\mcl{P}\mbf{u}}_{\text{Z}_{2}^{\text{n}}}=\ip{\mcl{P}^*\mbf{v}}{\mbf{u}}_{\text{Z}_{2}^{\text{m}}}
	\end{align*}
	for any $\mbf{u}\in\text{Z}_{2}^{\text{m}}$ and $\mbf{v}\in\text{Z}_{2}^{\text{n}}$, where $\text{n},\text{m}\in\N^3$.\\[-0.6em]

	\subsubsection{A cone of positive semidefinite 2D-PI operators can be parameterized by positive semidefinite matrices}\label{enum:PIprop_PSD}
	
	Here we say that an operator $\mcl{P}\in\Pi_{0112}^{\enn{}\times\enn{}}$ is positive semidefinite or (strictly) positive definite, denoted as $\mcl{P}\geq0$ and $\mcl{P}>0$, if for any $\mbf{v}\in\text{Z}_{2}^{\enn{}}$ with $\mbf{v}\neq\mbf{0}$ and some $\epsilon>0$,
	\begin{align*}
		\ip{\mbf{v}}{\mcl{P}\mbf{v}}_{\text{Z}_{2}^{\enn{}}}&\geq0,	&	&\text{or respectively,}
		&
		\ip{\mbf{v}}{\mbf{v}}_{\text{Z}_{2}^{\enn{}}}\geq \epsilon\ip{\mbf{v}}{\mbf{v}}_{\text{Z}_{2}^{\text{n}}}.
	\end{align*}
	
	\medskip
	
	Using Properties~\ref{enum:PIprop_sum} through~\ref{enum:PIprop_dif}, we will derive an equivalent PIE representation of linear 2D PDEs with inputs and outputs in Section~\ref{sec:Representation_of_ioSystems}. For this, we note that Property~\ref{enum:PIprop_dif} holds for PI operators mapping $\text{Z}_{2}^{\{n_0,0,n_2\}}$
	as well, as shown in Appendix~\ref{sec:appx_div_operator}. In Section~\ref{sec:Positive_PI_Params}, we prove that Properties~\ref{enum:PIprop_adj} and~\ref{enum:PIprop_PSD} also hold for PI operators on $\text{Z}_{2}^{\{n_0,0,n_2\}}$, allowing us to numerically solve the $L_2$-gain LPI presented in Section~\ref{sec:LPI_L2_Gain} using semidefinite programming.

	\subsection{Partial Integral Equations}
	
	A Partial Integral Equation (PIE) is a linear differential equation, parameterized by PI operators, describing the evolution of a fundamental state $\mbf{v}(t)\in L_2[\Omega_{ac}^{bd}]$. 
	For any linear, 2nd order, autonomous, 2D PDE, there exists an equivalent PIE representation, as well as a differential operator $\mscr{D}$ and PI operator $\mcl{T}$ such that any solution $\mbf{v}(t)$ to the PIE satisfies $\mbf{v}(t)=\mscr{D}\bar{\mbf{v}}(t)$, where $\bar{\mbf{v}}(t)=\mcl{T}\mbf{v}(t)$ is a solution to the PDE.
	
	\begin{example}
		Consider a simple 2D PDE on $(x,y)\in[0,1]\times[0,1]$, with Dirichlet boundary conditions,
		\begin{align}\label{eq:PDE_example}
			\dot{\bar{\mbf{v}}}(t)&\!=\!c\bbl[\partial_x\bar{\mbf{v}}(t) + \partial_y\bar{\mbf{v}}(t)\bbr],	&
			0&\!=\!\bar{\mbf{v}}(t,0,y)\!=\!\bar{\mbf{v}}(t,x,0).
		\end{align}
		Defining the fundamental state $\mbf{v}(t)=\partial_x\partial_y\bar{\mbf{v}}(t)\in L_2$, this system may be equivalently represented by the PIE
		\begin{align*}
			\int_{0}^{x}\!\!\int_{0}^{y}\!\!\dot{\mbf{v}}(t,\theta,\nu)d\nu d\theta&\!=c\bbbl[\int_{0}^{y}\!\!\mbf{v}(t,x,\nu)d\nu +\! \int_{0}^{x}\!\!\mbf{v}(t,\theta,y)d\theta\bbbr],	
		\end{align*}
		where $\mbf{v}(t)$ solves this PIE 
		if and only if $\bar{\mbf{v}}(t):=\int_{0}^{x}\!\!\int_{0}^{y}\!\mbf{v}(t,\theta,\nu)d\nu d\theta$ solves the PDE~\eqref{eq:PDE_example}. Defining $R:=\smallbmat{0&0&0\\0&R_{11}&0\\0&0&0}\in\mcl{N}_{2D}$ and $Q:=\smallbmat{0&Q_{01}&0\\Q_{10}&0&0\\0&0&0}\in\mcl{N}_{2D}$, where $R_{11}(x,y,\theta,\nu)=1$ and $Q_{01}(x,y,\nu)=Q_{10}(x,y,\theta)=c$, we may equivalently express the PIE as
		\begin{align*}
			\mcl{T}\dot{\mbf{v}}(t)&=\mcl{A}\mbf{v}(t),
		\end{align*}
		where $\mcl{T}:=\mcl{P}[R]\in\Pi_{2D}^{1\times 1}$ and $\mcl{A}:=\mcl{P}[Q]\in\Pi_{2D}^{1\times 1}$. 
		
		Consider now including a disturbance $w(t)\in\R$ and regulated output $z(t)\in\R$ in the PDE, as
		\begin{align}\label{eq:PDE_example_io}
			\dot{\bar{\mbf{v}}}(t)&=c\bbl[\partial_x\bar{\mbf{v}}(t) + \partial_y\bar{\mbf{v}}(t)\bbr]+kw(t),	\nonumber\\
			z(t)&=\int_{0}^{1}\int_{0}^{1}\bar{\mbf{v}}(t,x,y)dy dx.
		\end{align}
		Then, assuming the same boundary conditions, the system may be equivalently represented by the PIE
		\begin{align}\label{eq:PIE_example_io}
			\mcl{T}\dot{\mbf{v}}(t)&=\mcl{A}\mbf{v}(t)+kw(t)	&	&\hspace*{-0.3cm}=\mcl{A}\mbf{v}(t)+\mcl{B}w(t),	\nonumber\\[-0.2em]
			z(t)&=\bbbl(\smallint_{\Omega_{00}^{11}}[I]\circ\mcl{T}\bbbr)\mbf{v}(t)	&	&\hspace*{-0.3cm}=\mcl{C}\mbf{v}(t), \\[-1.8em] \nonumber
		\end{align}
		where $\mcl{T},\mcl{A}\in\Pi_{2D}^{1\times 1}=\Pi_{0112}^{\{0,0,1\}\times\{0,0,1\}}$ are as before, and we define $\mcl{B}:=\text{M}[k]\in\Pi_{0112}^{\{0,0,1\}\times\{1,0,0\}}$ and
		\begin{align*}
			\mcl{C}&:=\smallint_{\Omega_{00}^{11}}[(1-x)(1-y)]	&	&\in\Pi_{0112}^{\{1,0,0\}\times\{0,0,1\}}.
		\end{align*}
		Then, for any input $w$, the pair $(\mbf{v},z)$ is a solution to the PIE~\eqref{eq:PIE_example_io} if and only if $(\mcl{T}\mbf{v},z)$ is a solution to the PDE~\eqref{eq:PDE_example_io}.
		
	\end{example}

	\section{An LPI for $L_2$-Gain Analysis}\label{sec:LPI_L2_Gain}
	
	In this section, we present the main technical result of this paper.
	In particular, we provide an LPI for verifying a bound $\gamma$ on the $L_2$ gain of a PIE of the form
	\begin{align}\label{eq:primal_PIE}
	\mcl{T} \dot{\mbf{v}}(t)&= \mcl{A}\mbf{v}(t)+\mcl{B}w(t),	
	&	\mbf{v}(0)&=\mbf{0}, \nonumber\\
	z(t)&=\mcl{C}\mbf{v}(t)+\mcl{D}w(t),	
	\end{align}
	where $w(t)\in\R^{n_w}$, $z(t)\in\R^{n_z}$, and $\mbf{v}(t)\in L_2^{n_v}[\Omega_{ac}^{bd}]$ represent respectively the value of the input, output, and (fundamental) state at any time $t\geq 0$, and where
	\begin{align*}
	\mcl{T},\mcl{A}&\in\Pi_{0112}^{\{0,0,n_v\}\times\{0,0,n_v\}},	&
	\mcl{B}&\in\Pi_{0112}^{\{0,0,n_v\}\times\{n_w,0,0\}},	\\
	\mcl{C}&\in\Pi_{0112}^{\{n_z,0,0\}\times\{0,0,n_v\}},	&
	\mcl{D}&\in\Pi_{0112}^{\{n_z,0,0\}\times\{n_w,0,0\}}.
	\end{align*}
	\begin{lem}\label{lem:KYP}
		Let $\gamma>0$, and suppose there exists a 2D-PI operator $\mcl{P}\in\Pi_{2D}^{n_v\times n_v}$ such that $\mcl{P}=\mcl{P}^*>0$ and
		\begin{align}\label{eq:KYP_inequality}
		&\bmat{-\gamma I &\! \mcl{D} &\! \mcl{C}\\
			(\cdot)^* &\! -\gamma I &\! \mcl{B}^*\mcl{P}\mcl{T}	\\
			(\cdot)^* &\! (\cdot)^* &\! (\cdot)^* + \mcl{T}^*\mcl{P}\mcl{A}}	
		\leq 0
		\end{align}
		Then, for any $w\in L_2^{n_w}[0,\infty)$, if $(w,z)$ satisfies the PIE~\eqref{eq:primal_PIE}, then $z\in L_2^{n_z}[0,\infty)$ and $\|z\|_{L_2}\leq\gamma \|w\|_{L_2}$.
	\end{lem}
	
	\begin{proof}
		Define a storage function $V:L_2^{n_v}\rightarrow\R$ as $V(\mbf{v}):=\ip{\mcl{T}\mbf{v}}{\mcl{P}\mcl{T}\mbf{v}}_{L_2}$.
		Since $\mcl{P}> 0$, we have $V(\mbf{v})>0$ for any $\mbf{v}\neq \mbf{0}$.
		In addition, for any $w\in L_2[0,\infty)$, the derivative $\dot{V}(\mbf{v}(t))$ for $\mbf{v}(t)$ satisfying PIE~\eqref{eq:primal_PIE} is given by
		\begin{align*}
		\dot{V}(\mbf{v}(t))\! &=\ip{\mcl{T}\mbf{v}(t)}{\mcl{P}\mcl{T}\dot{\mbf{v}}(t)}_{L_2}	
		+\ip{\mcl{T}\dot{\mbf{v}}(t)}{\mcl{P}\mcl{T}\mbf{v}(t)}_{L_2}	\\
		&=\ip{\mcl{T}\mbf{v}(t)}{\mcl{P}\bl[\mcl{A}\mbf{v}(t)+\mcl{B}w(t)\br]}_{L_2}	\\
		&\qquad +\ip{\bl[\mcl{A}\mbf{v}(t)+\mcl{B}w(t)\br]}{\mcl{P}\mcl{T}\mbf{v}(t)}_{L_2}	\\
		&=\! \ip{\bmat{w(t)\\\mbf{v}(t)}\! }{\! 
			\bmat{0&\! \mcl{B}^*\mcl{P}\mcl{T}\\
				\mcl{T}^*\mcl{P}\mcl{B}&\!  \mcl{A}^*\mcl{P}\mcl{T}\! +\! \mcl{T}^*\mcl{P}\mcl{A}}\! \! 
			\bmat{w(t)\\\mbf{v}(t)}}_{\text{Z}_{2}^{\enn{1}}}
		\end{align*}
		where $\enn{1}:=\{n_w,0,n_v\}$ so that $\text{Z}_{2}^{\enn{1}}=\R^{n_w}\times L_2^{n_v}[\Omega_{ac}^{bd}]$. Define $\enn{2}:=\{n_z+n_w,0,n_v\}$. 
		Then, for any $w(t)\in\R^{n_w}$, and for any $\mbf{v}(t)\in L_2^{n_v}[\Omega_{ac}^{bd}]$ and $z(t)\in\R^{n_z}$ satisfying the PIE~\eqref{eq:primal_PIE} with input $w$, we have
		\begin{align*}
		&\ip{\bmat{\frac{z(t)}{\gamma}\\w(t)\\\mbf{v}(t)}\!}{\!\bmat{-\gamma I & \mcl{D} & \mcl{C}\\
				\mcl{D}^* & -\gamma I & \mcl{B}^*\mcl{P}\mcl{T}	\\
				\mcl{C}^* & \mcl{T}^*\mcl{P}\mcl{B} & \mcl{A}^*\mcl{P}\mcl{T} + \mcl{T}^*\mcl{P}\mcl{A}}\!
			\bmat{\frac{z(t)}{\gamma}\\w(t)\\\mbf{v}(t)}}_{\text{Z}_{2}^{\enn{2}}}	\\
		&=\ip{\bmat{w(t)\\\mbf{v}(t)}}{
			\bmat{0&\mcl{B}^*\mcl{P}\mcl{T}\\
				\mcl{T}^*\mcl{P}\mcl{B}& \mcl{A}^*\mcl{P}\mcl{T}+\mcl{T}^*\mcl{P}\mcl{A}}
			\bmat{w(t)\\\mbf{v}(t)}}_{\text{Z}_{2}^{\enn{1}}}	\\
		&\qquad -\gamma\|w(t)\|^2 +\gamma^{-1}\bbl[\ip{\mbf{v}(t)}{\mcl{C}^*z(t)}_{L_2} + \ip{w(t)}{\mcl{D}^* z(t)}\bbr]	\\
		&\qquad\quad+\gamma^{-1}\bbl[\ip{z(t)}{\mcl{C}\mbf{v}(t)}_{L_2} + \ip{z(t)}{\mcl{D}w(t)}\bbr] -\gamma^{-1}\|z(t)\|^2	\\
		&=\dot{V}\bl(\mbf{v}(t)\br) - \gamma\|w(t)\|^2 +\gamma^{-1}\|z(t)\|^2
		\end{align*}
		Invoking Eqn.~\eqref{eq:KYP_inequality}, it follows that
		\begin{align*}
		&\dot{V}\bl(\mbf{v}(t)\br)\leq  \gamma\|w(t)\|^2-\gamma^{-1}\|z(t)\|^2.
		\end{align*}
		Integrating both sides of this inequality from $0$ up to $\infty$, noting that $V(\mbf{v}(0))=V(0)=0$, we find
		\begin{align*}
		0&\leq
		\lim_{t\rightarrow\infty} V\bl(\mbf{v}(t)\br)
		\leq \gamma\|w\|^2_{L_2} -\gamma^{-1}\|z\|^2_{L_2},
		\end{align*}
		and therefore $\|z\|_{L_2}\leq \gamma\|w\|_{L_2}$.
	\end{proof}
	
	Lemma~\ref{lem:KYP} proves that, if the LPI~\eqref{eq:KYP_inequality} is feasible for some $\gamma>0$, then the $L_2$-gain $\frac{\|z\|_{L_2}}{\|w\|_{L_2}}$ of the 2D PIE~\eqref{eq:primal_PIE} is bounded by $\gamma$. In Section~\ref{sec:Representation_of_ioSystems}, we will show that any well-posed, linear, 2nd order 2D PDE can be equivalently represented as a PIE of the form~\eqref{eq:primal_PIE} -- thus allowing the $L_2$-gain to be tested as an LPI. In Section~\ref{sec:Positive_PI_Params}, we then show that feasibility of an LPI can be tested as an LMI, allowing the $L_2$-gain of 2D PDEs to be verified using semidefinite programming -- a result we show in Section~\ref{sec:L2_gain_PDE}.

	\section{A PIE Representation of 2D Partial Differential Input-Output Systems}\label{sec:Representation_of_ioSystems}
	
	Having shown that the $L_2$-gain of a 2D PIE can be tested by solving an LPI, we now show that equivalent PIE representations can be derived for systems belonging to a large class of 2D PDEs. In particular, in Subsection~\ref{sec:subsec:ioPDE_Representation}, we present a standardized format for representing linear, 2nd order, 2D PDEs with finite-dimensional input and output signals. In Subsection~\ref{sec:subsec:Fundamental_state}, we then derive a bijective map between the PDE state space $X_{w}\subset L_2[\Omega_{ac}^{bd}]$, constrained by boundary and continuity conditions, and the \textit{fundamental} state space $L_2[\Omega_{ac}^{bd}]$. Finally, in Subsection~\ref{sec:subsec:ioPIE_Representation}, we prove that for any solution to the PDE, an equivalent solution exists to an associated PIE, presenting the PI operators $\{\mcl{T}_0,\mcl{T}_1,\mcl{A},\mcl{B},\mcl{C},\mcl{D}\}$ defining this representation. 		

	\subsection{A Standardized PDE Format in 2D}\label{sec:subsec:ioPDE_Representation}
	
	We consider a coupled linear PDE of the form
	\begin{align}\label{eq:standard_PDE}
	\dot{\bar{\mbf{v}}}(t)
	&=\bar{\mcl{A}}\bar{\mbf{v}}(t)+\bar{\mcl{B}}w(t),	\nonumber\\
	z(t)&=\bar{\mcl{C}}\bar{\mbf{v}}(t) + \bar{\mcl{D}}w(t),
	\end{align}
	where at any time $t\geq0$, $w(t)\in\R^{n_w}$, $z(t)\in\R^{n_z}$, and $\bar{\mbf{v}}(t)\in X_{w(t)}$, where $X_{w(t)}\subseteq L_2^{n_0+n_1+n_2}[\Omega_{ac}^{bd}]$ includes the boundary conditions and continuity constraints, defined as
	\begin{align}\label{eq:Xset1}
	X_{w} :=\! \bbbbl\{
	\bar{\mbf{v}}=\bbbbl[\ \mat{\bar{\mbf{v}}_0\\[-0.2ex]\bar{\mbf{v}}_1\\[-0.2ex]\bar{\mbf{v}}_2}\ \bbbbr]\! \in\! \bbbbl[\ \mat{L_2^{n_0}\\[-0.2ex]W_1^{n_1}\\[-0.2ex]W_2^{n_2}}\ \bbbbr]
	 \bbbbl\lvert\ 
	\bar{\mcl{E}}_0\bar{\mbf{v}}+\mcl{E}_1 w=0 
	\bbbbr\},
	\end{align}
	and where the operators $\{\bar{\mcl{A}},\bar{\mcl{B}},\bar{\mcl{C}},\bar{\mcl{D}},\bar{\mcl{E}}_0,\mcl{E}_1\}$ are all linear. In particular, the PDE dynamics are defined by the operators \\[-3.5ex]
	\begin{align}\label{eq:PDE_operators}
	\bar{\mcl{A}}&\! :=\! \sum_{i,j=0}^{2}\text{M}[A_{ij}]\ \partial_x^i \partial_y^j\  \text{M}[S_{i,j}],	&
	\bar{\mcl{B}}&:=\text{M}[B],		\nonumber\\[-0.6ex]
	\bar{\mcl{C}}&\! :=\!  \sum_{i,j=0}^{2} \smallint_{x=a}^{b}\! [I]\smallint_{y=c}^{d}\! [C_{ij}]\ \partial_x^i \partial_y^j\  \text{M}[S_{i,j}], &
	\bar{\mcl{D}}&\! :=\! \text{M}[D], \\[-4.0ex]
	\ \nonumber
	\end{align}
	parameterized by matrix-valued functions
	\begin{align*}
	\bmat{A_{ij}&B\\ C_{ij}&D}\in
	\bmat{L_2^{n_v\times m_{ij}}[\Omega_{ac}^{bd}]&L_2^{n_v\times n_w}[\Omega_{ac}^{bd}]\\
		L_2^{n_z\times m_{ij}}[\Omega_{ac}^{bd}]&\R^{n_z\times n_w}}, 
	\ \nonumber
	\end{align*}
	where $n_v=n_0+n_1+n_2$ and $m_{ij}:=\sum_{k=\max\{i,j\}}^{2}n_k$,
	and where the matrix
	\begin{align*}
	S_{i,j} := \begin{cases}
	I_{n_0+n_1+n_2},	&	\text{if } i=j=0,	\\
	\bmat{0_{(n_1+n_2)\times n_0} & I_{n_1+n_2}},	&	\text{if } \max\{i,j\}=1,	\\
	\bmat{0_{n_2\times n_0} & 0_{n_2\times n_1} & I_{n_2}},	&	\text{if } \max\{i,j\}=2,	\\
	\end{cases}
	\end{align*}
	extracts all elements $\mbf{u}(t)=S_{i,j}\bar{\mbf{v}}(t)\in W^{m_{ij}}_{\max\{i,j\}}[\Omega_{ac}^{bd}]$ of the state $\bar{\mbf{v}}(t)$ which are differentiable up to at least order $i$ in $x$ and $j$ in $y$, for any $t\geq0$. In addition, the state $\bar{\mbf{v}}(t)$ at each time is constrained by the boundary conditions $\bar{\mcl{E}}_0\bar{\mbf{v}}(t)+\mcl{E}_1 w(t)=0$,
	where
	\begin{align}\label{eq:BC_Operators}
	 \bar{\mcl{E}}_0&=\mcl{P}[E_0]\ \Lambda_{\text{bf}},
	 &	&\text{and}	&
	 \mcl{E}_1&=\text{M}[E_{1}],
	\end{align}
	for a matrix-valued function $E_1\in\text{Z}_{1}^{\enn{b}\times n_w}[\Omega_{ac}^{bd}]$ and parameters $E_0\in\mcl{N}_{011}^{\enn{b}\times \enn{f}}[\Omega_{ac}^{bd}]$, where $\enn{b}:=\{n_1+4n_2,n_1+2n_2\}$ corresponds to the number of boundary conditions, and $\enn{f}:=\{4n_1+16n_2,2n_1+4n_2\}$,
	and where the operator $\Lambda_{\text{bf}}:L_2^{n_0}\!\times\! W_1^{n_1}\!\times\! W_2^{n_2}\rightarrow\text{Z}_{1}^{\enn{f}}$ 
	extracts all the possible boundary values for the state components $\bar{\mbf{v}}_1$ and $\bar{\mbf{v}}_2$, as limited by differentiability.
	In particular,
	\begin{align}\label{eq:Lambda_bf}
	\Lambda_{\text{bf}}\! =\! \bmat{\Lambda_1\\\Lambda_2\\\Lambda_3}\!: L_2^{n_0}\!\times W_1^{n_1}\!\times W_2^{n_2}\!\rightarrow\! \bmat{\R^{4n_1+16n_2}\\ L_2^{2n_1+4n_2}[\Omega_{a}^{b}]\\ L_2^{2n_1+4n_2}[\Omega_{c}^{d}]},
	\end{align}
	where \\[-4.25ex]
	{\small
	\begin{align*}
	\Lambda_1&\! :=\!
	\bmat{\begin{array}{l}
		0~~\Delta_1~\thinspace 0\\
		0~~0~~~\Delta_1 \\
		0~~0~~~\Delta_1\partial_x \\
		0~~0~~~\Delta_1\partial_y  \\
		0~~0~~~\Delta_1\partial_{xy}
		\end{array}}
	,\enspace
	\bmat{\Lambda_2\\ \Lambda_3}\! :=\!
	\bmat{\begin{array}{l}
		0~~\Delta_2 \partial_x~\thinspace 0   \\
		0~~0~~~~~~\Delta_2 \partial_x^2  \\
		0~~0~~~~~~\Delta_2\partial_x^2\partial_y \\
		0~~\Delta_3 \partial_y~\thinspace 0\\
		0~~0~~~~~~\Delta_3\partial_y^2   \\
		0~~0~~~~~~\Delta_3\partial_x\partial_y^2
		\end{array}},
	\end{align*}
	}
	and where we use the Dirac operators $\Delta_k$ defined as
	\[
	\Delta_1=\bmat{
		\Delta_x^a\Delta_y^c\\
		\Delta_x^b\Delta_y^c\\
		\Delta_x^a\Delta_y^d\\
		\Delta_x^b\Delta_y^d
	},\quad \Delta_2= \bmat{
		\Delta_y^c    \\
		\Delta_y^d
	},\quad \Delta_3= \bmat{
		\Delta_x^a    \\
		\Delta_x^b
	}.
	\]	
	
	\begin{defn}[Solution to the PDE]
	For a given input signal $w$ and given initial conditions $\bar{\mbf{v}}_{\text{I}}\in X_{w(0)}$, we say that $(\bar{\mbf{v}},z)$ is a solution to the PDE defined by $\{A_{ij},B,C_{ij},D,E_0,E_1\}$ if $\bar{\mbf{v}}$ is Frech\'et differentiable, $\bar{\mbf{v}}(0)=\bar{\mbf{v}}_{\text{I}}$, and for all $t\geq0$, $\bar{\mbf{v}}(t)\in X_{w(t)}$, and $(\bar{\mbf{v}}(t),z(t))$ satisfies Eqn.~\eqref{eq:standard_PDE} with the operators $\{\bar{\mcl{A}},\bar{\mcl{B}},\bar{\mcl{C}},\bar{\mcl{D}},\bar{\mcl{E}}_0,\mcl{E}_1\}$ defined as in~\eqref{eq:PDE_operators} and~\eqref{eq:BC_Operators}.
	\end{defn}

	\subsection{A Bijection Between the Fundamental and PDE State}\label{sec:subsec:Fundamental_state}
	
	In the PDE~\eqref{eq:standard_PDE} defined by $\{A_{ij},B,C_{ij},D,E_0,E_1\}$, the state $\bar{\mbf{v}}(t)\in X_{w(t)}$ at each time $t\geq 0$ is constrained to satisfy continuity constraints and boundary conditions, defined by $E_0$ and $E_1$. For any such $\bar{\mbf{v}}\in X_{w}$, we define an associated fundamental state $\mbf{v}\in L_2^{n_v}[\Omega_{ac}^{bd}]$, free of boundary and continuity constraints, using a
	 differential operator $\mscr{D}$:
	\begin{align*}
	\mbf{v}:=\bmat{\mbf{v}_0\\\mbf{v}_1\\\mbf{v}_2}= \underbrace{\bmat{I_{n_0}& & \\ &\partial_{x}\partial_{y} & \\ & &\partial_{x}^2\partial_{y}^2}}_{\mscr{D}}\bmat{\bar{\mbf{v}}_0\\\bar{\mbf{v}}_1\\\bar{\mbf{v}}_2}=\mscr{D}\bar{\mbf{v}}.\\[-4.5ex]
	\end{align*}
	In this subsection, we show that if the parameters $E_0,E_1$ define well-posed boundary conditions, then there exist associated PI operators $\mcl{T}_0,\mcl{T}_1$ such that
	\begin{align*}
	\bar{\mbf{v}} &= \mcl{T}_0\mscr{D}\bar{\mbf{v}} + \mcl{T}_1 w,
	&	&\text{and}	&
	\mbf{v} &= \mscr{D}\bl[\mcl{T}_0\mbf{v} + \mcl{T}_1 w\br],
	\end{align*}
	for any $\bar{\mbf{v}}\in X_{w}$ and $\mbf{v}\in L_2$. 
	To prove this result, we recall the following lemma from~\cite{jagt2021PIEArxiv}, expressing the PDE state in terms of the fundamental state and a set of boundary values.
	
	\begin{figure*}[!ht]
		\hrulefill
		\footnotesize
		\begin{flalign}
		&K_1=\bmat{K_{30} \\ \{0,K_{31},0\} \\ \{0,K_{32},0\}}\in \bmat{L_2 \\ \mcl{N}_{1D\rightarrow 2D} \\ \mcl{N}_{1D\rightarrow 2D}},		&
		&K_2= \bmat{T_{00}&0&0\\0&K_{33}&0\\0&0&0}\in\mcl{N}_{2D},	 \label{eq:K_mat} \\
		H_1 &= \bmat{H_{00}&H_{01}&H_{02}\\0&\{H_{11},0,0\}&0\\0&0&\{H_{22},0,0\}}\in\mcl{N}_{011},	&
		H_2 &= \bmat{H_{03}\\\{H_{13},0,0\}\\\{H_{23},0,0\}}\in\bmat{L_2\\\mcl{N}_{2D\rightarrow 1D}\\\mcl{N}_{2D\rightarrow 1D}},		&	& \label{eq:H_mat}
		\end{flalign}
		where $T_{00}$ is as defined in~\eqref{eq:Tmat} in Fig.~\ref{fig:Tmap_matrices}, and
		{\tiny
			\begin{align}
			&\mat{K_{31}(x,y,\theta)=
				\bmat{
					0&0&0\\
					I_{n_1}&0&0\\
					0&(x-\theta)&(y-c)(x-\theta)
				},		&
				&\hspace*{1.65cm}K_{32}(x,y,\nu)=
				\bmat{
					0&0&0\\
					I_{n_1}&0&0\\
					0&(y-\nu)&(x-a)(y-\nu)
				}, 	&
				&T_{00}=\bmat{I_{n_0}&0&0\\0&0&0\\0&0&0}},	\nonumber\\
			&\mat{K_{30}(x,y)=
				\bmat{
					0&0&0&0&0\\
					I_{n_1}&0&0&0&0\\
					0&I_{n_2}&(x-a)&(y-c)&(y-c)(x-a)
				},	&	
				\hspace{2.0cm}
				&K_{33}(x,y,\theta,\nu)=
				\bmat{
					0&0&0\\
					0&I_{n_1}&0\\
					0&0&(x-\theta)(y-\nu)
			}}.  \nonumber
			\end{align}}	
		and where,
		{\tiny
			\begin{align*}
			&H_{00} =
			\bmat{I_{n_1}&0&0&0&0\\
				I_{n_1}&0&0&0&0\\
				I_{n_1}&0&0&0&0\\
				I_{n_1}&0&0&0&0\\
				0&I_{n_2}&0&0&0\\
				0&I_{n_2}&(b-a)&0&0\\
				0&I_{n_2}&0&(d-c)&0\\
				0&I_{n_2}&(b-a)&(d-c)&(d-c)(b-a)\\
				0&0&I_{n_2}&0&0\\
				0&0&I_{n_2}&0&0\\
				0&0&I_{n_2}&0&(d-c)\\
				0&0&I_{n_2}&0&(d-c)\\
				0&0&0&I_{n_2}&0\\
				0&0&0&I_{n_2}&(b-a)\\
				0&0&0&I_{n_2}&0\\
				0&0&0&I_{n_2}&(b-a)\\
				0&0&0&0&I_{n_2}\\
				0&0&0&0&I_{n_2}\\
				0&0&0&0&I_{n_2}\\
				0&0&0&0&I_{n_2}},    &
			&H_{01}(x) =
			\bmat{
				0&0&0\\I_{n_1}&0&0\\0&0&0\\I_{n_1}&0&0\\
				0&0&0\\ 0&(b-x)&0\\ 0&0&0\\
				0&(b-x)&(d-c)(b-x)\\
				0&0&0\\ 0&I_{n_2}&0 \\0&0&0\\ 0&I_{n_2}&(d-c)\\
				0&0&0\\ 0&0&(b-x) \\0&0&0\\ 0&0&(b-x)\\
				0&0&0\\ 0&0&I_{n_2} \\0&0&0\\ 0&0&I_{n_2}
			},   &
			&H_{02}(y) =
			\bmat{
				0&0&0\\0&0&0\\I_{n_1}&0&0\\I_{n_1}&0&0\\
				0&0&0\\ 0&0&0\\ 0&(d-y)&0\\
				0&(d-y)&(b-a)(d-y)\\
				0&0&0\\ 0&0&0 \\0&0&(d-y)\\ 0&0&(d-y)\\
				0&0&0\\ 0&0&0 \\0&I_{n_2}&0\\ 0&I_{n_2}&(b-a)\\
				0&0&0\\ 0&0&0 \\0&0&I_{n_2}\\ 0&0&I_{n_2}
			},   \nonumber\\
			&\begin{array}{l}
			H_{11}=
			\bmat{
				I_{n_1}&0&0\\I_{n_1}&0&0\\
				0&I_{n_2}&0\\0&I_{n_2}&(d-c)\\0&0&I_{n_2}\\0&0&I_{n_2}
			},\\
			\\
			H_{22}=
			\bmat{
				I_{n_1}&0&0\\I_{n_1}&0&0\\
				0&I_{n_2}&0\\0&I_{n_2}&(b-a)\\0&0&I_{n_2}\\0&0&I_{n_2}
			}
			,
			\end{array}
			&
			&\begin{array}{l}
			H_{13}(y)=
			\bmat{
				0&0&0\\0&I_{n_1}&0\\
				0&0&0\\ 0&0&(d-y)\\0&0&0\\0&0&I_{n_2}
			},\\
			\\
			H_{23}(x)=
			\bmat{
				0&0&0\\0&I_{n_1}&0\\
				0&0&0\\ 0&0&(b-x)\\0&0&0\\0&0&I_{n_2}
			},
			\end{array}
			&
			&H_{03}(x,y)=
			\bmat{
				0&0&0\\0&0&0\\0&0&0\\0&I_{n_1}&0\\
				0&0&0\\0&0&0\\0&0&0\\ 0&0&(d-y)(b-x)\\0&0&0\\0&0&0\\0&0&0\\ 0&0&(d-y)\\0&0&0\\0&0&0\\0&0&0\\ 0&0&(b-x)\\
				0&0&0\\0&0&0\\0&0&0\\0&0&I_{n_2}
			}.
			\end{align*}
		}
		\vspace*{-0.1cm}
		
		\hrulefill
		\vspace*{-0.25cm}
		\caption{Parameters $K_1$, $K_2$, $H_1$ and $H_2$ defining the mappings in Lemma~\ref{lem:vhat_to_v} and in Corollary~\ref{cor:vhat_to_BC}}
		\label{fig:H_K_matrices}
		\vspace*{-0.50cm}
	\end{figure*}
	
	\begin{lem}\label{lem:vhat_to_v}
		Let $\bar{\mbf{v}}\in L_2^{n_0}\!\times\! W_1^{n_1}\!\times\! W_2^{n_2}$ and define $\Lambda_{\text{bc}}:L_2^{n_0}\!\times\! W_1^{n_1}\!\times\! W_2^{n_2}\rightarrow\text{Z}_{1}^{\enn{b}}$ with $\enn{b}=\{n_1+4n_2,n_1+2n_2\}$ as
		{\small
			\begin{align}\label{eq:Lambda_bc}
			\Lambda_{\text{bc}} :=\bmat{\begin{array}{lll}
				0&\Delta_x^a \Delta_y^c&0\\
				0&0&\Delta_x^a\Delta_y^c \\
				0&0&\Delta_x^a\Delta_y^c \  \partial_x   \\
				0&0&\Delta_x^a\Delta_y^c \  \partial_y   \\
				0&0&\Delta_x^a\Delta_y^c \  \partial_x\partial_y  \\
				0& \Delta_y^c\ \partial_x & 0\\
				0&0&\Delta_y^c\  \partial_x^2  \\
				0&0&\Delta_y^c\  \partial_x^2\partial_y  \\
				0& \Delta_x^a\ \partial_y & 0\\
				0&0& \Delta_x^a\   \partial_y^2 \\
				0&0& \Delta_x^a\ \partial_x \partial_y^2
				\end{array}}.
			\end{align}
		}
		Then, if PI operators $\mcl{K}_1=\mcl{P}[K_1]\in\Pi_{2D\leftarrow 011}^{n_v\times\enn{b}}$ and $\mcl{K}_2=\mcl{P}[K_2]\in\Pi_{2D}^{n_v\times n_v}$, where $K_1$ and $K_2$ are as defined in Eqn.~\eqref{eq:K_mat} in Figure~\ref{fig:H_K_matrices}, then
		\begin{align*}
		\bar{\mbf{v}}=\mcl{K}_1 \Lambda_{\text{bc}}\bar{\mbf{v}}+\mcl{K}_2 {\mbf {v}}
		\end{align*}
		where $\mbf{v} = \mscr{D}\bar{\mbf{v}}$.
	\end{lem}
	
	\begin{proof}
		A proof can be found in~\cite{jagt2021PIEArxiv}.
	\end{proof}

	\begin{cor}\label{cor:vhat_to_BC}
		Let $\mbf{v}\in L_2^{n_0}\!\times\! W_1^{n_1}\!\times\! W_2^{n_2}$ and let $\Lambda_{\text{bf}}$ be as defined in Eqn.~\eqref{eq:Lambda_bf}. Then if PI operators  $\mcl{H}_1=\mcl{P}[H_1]\in\Pi_{011}^{\enn{f}\times\enn{b}}$ and $\mcl{H}_2=\mcl{P}[H_2]\in\Pi_{011\leftarrow 2D}^{\enn{f}\times n_v}$ with $\enn{f}=\{4n_1+16n_2,2n_1+4n_2\}$, where $H_1$ and $H_2$ are as defined in Eqn.~\eqref{eq:H_mat} in Figure~\ref{fig:H_K_matrices}, then
		\[
		\Lambda_{\text{bf}} \bar{\mbf{v}} = \mcl{H}_1 \Lambda_{\text{bc}} \bar{\mbf{v}}+\mcl{H}_2 \mbf{v}
		\]
		where $\mbf{v} = \mscr{D} \bar{\mbf{v}}$, and where $\Lambda_{\text{bc}}$ is as defined in Eqn.~\eqref{eq:Lambda_bc}.
	\end{cor}
	
	Using these results, we can express $\bar{\mbf{v}}\in X_{w}$ directly in terms of $\mscr{D}\bar{\mbf{v}}\in L_2^{n_v}$ and the input signal $w$, as shown in the following theorem. We provide an outline of the proof of this result here, referring to Appendix~\ref{sec:appx_Tmap} for a full proof.
	
	\begin{figure*}[!ht]
		\hrulefill\\
		\footnotesize
		Define \\[-5.0ex]
		\begin{flalign}\label{eq:Tmat}
		&\enspace T:=\bmat{T_{00}&0&0\\0&T_{11}&T_{12}\\0&T_{21}&T_{22}}\in\mcl{N}^{n_v\times n_v}_{2D}
		\hspace*{1.0cm}\text{and}\hspace*{1.0cm}
		\begin{array}{l}
		T_{1}(x,y):=-K_{30}(x,y)Q_{0} - \int_{a}^{x}K_{31}^{1}(x,y,\theta)Q_{1}(\theta)d\theta\\ 
		\hspace*{2.0cm}- \int_{c}^{y}K_{32}^{1}(x,y,\nu)Q_{2}(\nu)d\nu,
		\end{array}	\\[-2.75ex]
		\intertext{where}\nonumber\\[-6.2ex]
		&\mat{\begin{array}{l}
			T_{11}(x,y,\theta,\nu) = K_{33}(x,y,\theta,\nu)
			+ T_{21}(x,y,\theta,\nu) + T_{12}(x,y,\theta,\nu)
			-T_{22}(x,y,\theta,\nu), \nonumber\\
			T_{21}(x,y,\theta,\nu) = - K_{32}(x,y,\nu)G_{23}^{0}(\theta,\nu) + T_{22}(x,y,\theta,\nu), \nonumber\\
			T_{12}(x,y,\theta,\nu) = - K_{31}^{1}(x,y,\theta)G_{13}^{0}(\theta,\nu) + T_{22}(x,y,\theta,\nu), \nonumber\\
			T_{22}(x,y,\theta,\nu) = -K_{30}(x,y)G_{03}(\theta,\nu)
			-\int_{a}^{x}K_{31}(x,y,\eta)G_{13}^{1}(\eta,\nu,\theta)d\eta
			-\int_{c}^{y}K_{32}(x,y,\mu)G_{23}^{1}(\theta,\mu,\nu)d\mu, 
			\end{array}	&
			T_{00}=\bmat{I_{n_0}&0&0\\0&0&0\\0&0&0}},\nonumber\\[-1.75ex]
		\intertext{
			where the functions $K_{ij}$ are the parameters defining the operators $\mcl{K}_1$ and $\mcl{K}_2$ in Lem.~\ref{lem:vhat_to_v}, defined as in Eqn.~\eqref{eq:K_mat} in Fig.~\ref{fig:H_K_matrices}, and}\nonumber\\[-6.4ex]
		&\enspace G_{0}(x,y)=\hat{R}_{00}F_{0}(x,y) + \hat{R}_{01}(x)F_{1}^{0}(x,y) + \int_{a}^{b}\hat{R}_{01}(\theta)F_{1}^{1}(\theta,y,x)d\theta
		+ \hat{R}_{02}(y)F_{2}^{0}(x,y) + \int_{c}^{d}\hat{R}_{02}(\nu)F_{2}^{1}(x,\nu,y),    \nonumber\\[-0.25ex]
		&
		\begin{array}{l}
		G_{1}^{0}(x,y)=\hat{R}_{11}^{0}(x)F_{1}^{0}(x,y),
		\\
		G_{1}^{1}(x,y,\theta)=\hat{R}_{10}(x)F_{0}(\theta,y) + \hat{R}_{11}^{0}(x)F_{1}^{1}(x,y,\theta) 
		\\
		\quad+ \hat{R}_{11}^{1}(x,\theta)F_{1}^{0}(\theta,y)  
		+\int_{a}^{b}\hat{R}_{11}^{1}(x,\eta)F_{1}^{1}(\eta,y,\theta)d\eta
		\\
		\qquad+\hat{R}_{12}(x,y)F_{2}^{0}(\theta,y) + \int_{c}^{d}\hat{R}_{12}(x,\nu)F_{2}^{1}(\theta,\nu,y)d\nu,
		\end{array}
		\hspace{2.5cm}
		\begin{array}{l}
		G_{2}^{0}(x,y)=\hat{R}_{22}^{0}(y)F_{2}^{0}(x,y), 
		\\
		G_{2}^{1}(x,y,\nu)=\hat{R}_{20}(y)F_{0}(x,\nu) + \hat{R}_{22}^{0}(y)F_{2}^{1}(x,y,\nu) 
		\\
		\quad+ \hat{R}_{22}^{1}(y,\nu)F_{2}^{0}(x,\nu)  
		+\int_{c}^{d}\hat{R}_{22}^{1}(y,\mu)F_{2}^{1}(x,\mu,\nu)d\mu
		\\
		\qquad+\hat{R}_{21}(x,y)F_{1}^{0}(x,\nu) + \int_{a}^{b}\hat{R}_{21}(\theta,y)F_{1}^{1}(\theta,\nu,x)d\theta,
		\end{array} \nonumber\\[-2.75ex]
		\intertext{and,}\nonumber\\[-6.35ex]
		&\begin{array}{l}
		Q_{0}=\hat{R}_{00}E_{1,0} + \int_{a}^{b}\hat{R}_{01}(\theta)E_{1,1}(\theta)d\theta \int_{c}^{d}\hat{R}_{02}(\nu)E_{1,2}(\nu),    \\
		Q_{1}(x)=\hat{R}_{10}(x)E_{1,0} + \hat{R}_{11}^{0}(x)E_{1,1}(x) + \int_{a}^{x}\hat{R}_{11}^{1}(x,\theta)E_{1,1}(\theta)d\theta + \int_{x}^{b}\hat{R}_{11}^{2}(x,\theta)E_{1,1}(\theta)d\theta + \int_{c}^{d}\hat{R}_{12}(x,\nu)E_{1,2}(\nu)d\nu,    \\
		Q_{2}(y)=\hat{R}_{20}(y)E_{1,0} + \int_{a}^{b}\hat{R}_{21}(y,\theta)E_{1,1}(\theta)d\theta + \hat{R}_{22}^{0}(y)E_{1,2}(y) + \int_{c}^{y}\hat{R}_{22}^{1}(y,\nu)E_{1,2}(\nu)d\nu + \int_{y}^{d}\hat{R}_{22}^{2}(y,\nu)E_{1,2}(\nu)d\nu,
		\end{array} \nonumber\\[-2.75ex]
		\intertext{with}\nonumber\\[-6.35ex]
		&\begin{array}{l l}
		F_{0}(x,y)=E_{00}H_{03}(x,y) + E_{01}(x)H_{13}(x,y) + E_{02}(y)H_{23}(x,y),    \\
		F_{1}^{1}(x,y,\theta)=E_{10}(x)H_{03}(\theta,y) + E_{11}^{1}(x,\theta)H_{13}(\theta,y)
		+E_{12}(x,y)H_{23}(\theta,y),   & \quad
		F_{1}^{0}(x,y)=E_{11}^{0}(x)H_{13}(x,y), \\
		F_{2}^{1}(x,y,\nu)=E_{20}(y)H_{03}(x,\nu) + E_{22}^{1}(y,\nu)H_{23}(x,\nu)
		+E_{21}(x,y)H_{13}(x,\nu),   & \quad
		F_{2}^{0}(x,y)=E_{22}^{0}(y)H_{23}(x,y),
		\end{array} \nonumber\\[-2.0ex]
		\intertext{and where $\smallbmat{\hat{R}_{00}&\hat{R}_{01}&\hat{R}_{02}\\\hat{R}_{10}&\hat{R}_{11}&\hat{R}_{12}\\\hat{R}_{20}&\hat{R}_{21}&\hat{R}_{22}}=\mcl{L}_{\text{inv}}\left(\smallbmat{R_{00}&R_{01}&R_{02}\\R_{10}&R_{11}&R_{12}\\R_{20}&R_{21}&R_{22}}\right)\in\mcl{N}_{011}$ are the parameters associated to the inverse of $\mcl{P}\smallbmat{R_{00}&R_{01}&R_{02}\\R_{10}&R_{11}&R_{12}\\R_{20}&R_{21}&R_{22}}$, where $\mcl{L}_{\text{inv}}:\mcl{N}_{011}\rightarrow\mcl{N}_{011}$ is defined as in Equation~(37) in Appendix~II-A of~\cite{jagt2021PIEArxiv}, and}\nonumber\\[-5.1ex] &\enspace R_{11}=\{R_{11}^{0},R_{11}^{1},R_{11}^{1}\}\in\mcl{N}_{1D},\qquad \text{and}\qquad R_{22}=\{R_{22}^{0},R_{22}^{1},R_{22}^{1}\}\in\mcl{N}_{1D},\qquad \text{with}\nonumber\\[-0.25ex]
		&\begin{array}{l l}
		R_{00}=E_{00}H_{00} + \int_{a}^{b}E_{01}(x)H_{10}(x)dx + \int_{c}^{d}E_{02}(y)H_{20}(y)dy, \\
		R_{01}(x)=E_{00}H_{01}(x) + E_{01}(x)H_{11}(x),  &
		R_{02}(y)=E_{00}H_{02}(y) + E_{02}(y)H_{22}(y),  \\
		R_{10}(x)=E_{10}(x)H_{00}, &
		R_{20}(y)=E_{20}(y)H_{00}, \\
		R_{11}^{0}(x)=E_{11}^{0}H_{11},   &
		R_{22}^{0}(y)=E_{22}^{0}H_{22},   \\
		R_{11}^{1}(x,\theta)= E_{10}(x)H_{01}(\theta) + E_{11}^{1}(x,\theta)H_{11}, &
		R_{22}^{1}(y,\nu)= E_{20}(y)H_{02}(\nu) + E_{22}^{1}(y,\nu)H_{22}, \\
		R_{12}(x,y)=E_{10}(x)H_{02}(y) + E_{12}(x,y)H_{22}(y), &
		R_{21}(x,y)=E_{20}(y)H_{01}(x) + E_{21}(x,y)H_{11}(x),
		\end{array} \nonumber\\[-4.0ex]
		\ \nonumber
		\end{flalign}
		where the functions $H_{ij}$ are the parameters defining the operators $\mcl{H}_1$ and $\mcl{H}_2$ in Cor.~\ref{cor:vhat_to_BC}, defined as in Eqn.~\eqref{eq:H_mat} in Fig.~\ref{fig:H_K_matrices}.
		\vspace*{-0.1cm}
		
		\hrulefill
		\vspace*{-0.25cm}
		\caption{Parameters $T_0$ and $T_1$ defining the PI operators $\mcl{T}_0=\mcl{P}[T_0]$ and $\mcl{T}_1=\text{M}[T_1]$ mapping the fundamental state back to the PDE state in Theorem~\ref{thm:Tmap}}
		\vspace*{-0.50cm}
		\label{fig:Tmap_matrices}
	\end{figure*}
	
	\begin{thm}\label{thm:Tmap}
		Let
 		$E_0=\smallbmat{E_{00}&E_{01}&E_{02}\\E_{10}&E_{11}&E_{12}\\E_{20}&E_{21}&E_{22}}\in\mcl{N}_{011}^{\enn{b}\times\enn{f}}$ and $E_1=\smallbmat{E_{1,0}\\E_{1,1}\\E_{1,2}}\in\mcl{N}_{011\leftarrow 2D}^{\enn{b}\times n_w}$ 
 		with $E_{jj}:=\{E_{jj}^{0},E_{jj}^{1},E_{jj}^{1}\}\in\mcl{N}_{1D}$ for $j\in\{1,2\}$
		be given, and $\mcl{E}_0\mcl{H}_1$ be invertible, where $\mcl{E}_0:=\mcl{P}[E_0]\in\Pi_{011}^{\enn{b}\times\enn{f}}$ and $\mcl{H}_1\in\Pi_{011}^{\enn{f}\times\enn{b}}$ is as in Cor.~\ref{cor:vhat_to_BC}. Let $w$ be a given input signal, with associated set $X_{w}$ as defined in Eqn.~\eqref{eq:Xset1}. Then, if $\mcl{T}_0=\mcl{P}[T_0]\in\Pi_{2D}^{n_v\times n_v}$ and $\mcl{T}_1=\text{M}[T_1]\in\Pi_{2D\leftarrow 011}^{n_v\times \{n_w,0\}}$, where $T_0\in\mcl{N}_{2D}^{n_v\times n_v}$ and $T_1\in L_2^{n_v\times n_w}[\Omega_{ac}^{bd}]$ are as defined in Eqn.~\eqref{eq:Tmat} in Fig.~\ref{fig:Tmap_matrices}, then for any $\bar{\mbf{v}}\in X_{w}$ and $\mbf{v} \in L^{n_v}_2$, \\[-2.75ex]
		\begin{align*}
		\bar{\mbf{v}} &= \mcl{T}_0\mscr{D}\bar{\mbf{v}} + \mcl{T}_1 w
		&	&\text{and}	&
		 \mbf{v} &= \mscr{D}\bbl[\mcl{T}_0\mbf{v}+\mcl{T}_1 w\bbr], \\[-4.25ex]
		\end{align*}
		where $\mscr{D}=\smallbmat{I_{n_0}&&\\&\partial_x\partial_y&\\&&\partial_x^2\partial_y^2}$.
	\end{thm}

	\paragraph*{Outline of proof}
	The result follows from application of Lemma~\ref{lem:vhat_to_v} and Corollary~\ref{cor:vhat_to_BC}, allowing us to express
	\begin{align*}
	 \bar{\mbf{v}}&=\mcl{K}_1\Lambda_{\text{bc}}\bar{\mbf{v}}+\mcl{K}_2\mbf{v},	 &
	 \Lambda_{\text{bf}}\bar{\mbf{v}}&=\mcl{H}_1\Lambda_{\text{bc}}\bar{\mbf{v}}+\mcl{H}_2\mbf{v}.
	\end{align*}
	Substituting the second result into the expression for the boundary conditions, $0=\mcl{E}_0\Lambda_{\text{bf}}\bar{\mbf{v}}+\mcl{E}_1 w$, it follows that
	\begin{align*}
	 \mcl{E}_0\mcl{H}_1\Lambda_{\text{bc}}\bar{\mbf{v}}
	 =-\mcl{E}_0\mcl{H}_2\mbf{v}-\mcl{E}_1 w,
	\end{align*}
	so that, since (by assumption) $\mcl{R}:=\mcl{E}_0\mcl{H}_1$ is invertible,
	\begin{align*}
	 \Lambda_{\text{bc}}\bar{\mbf{v}}=-\mcl{R}^{-1}\mcl{E}_0\mcl{H}_2\mbf{v}-\mcl{R}^{-1}\mcl{E}_1 w.
	\end{align*}
	Substituting this result into the expression for $\bar{\mbf{v}}$ from Lemma~\ref{lem:vhat_to_v}, it immediately follows that
	\begin{align*}
	\bar{\mbf{v}}&=[\mcl{K}_2-\mcl{K}_1\mcl{R}^{-1}\mcl{E}_0\mcl{H}_2]\mbf{v}-\mcl{K}_1\mcl{R}^{-1}\mcl{E}_1 w
	=\mcl{T}_0\mbf{v} + \mcl{T}_1 w.
	\end{align*}

	\subsection{PDE to PIE Conversion}\label{sec:subsec:ioPIE_Representation}
	
	Having constructed the PI operators $\mcl{T}_0,\mcl{T}_1$ mapping fundamental states $\mbf{v}\in L_2^{n_v}[\Omega_{ac}^{bd}]$ to PDE states $\bar{\mbf{v}}\in X_{w}$, we can now define an equivalent PIE representation of the standardized PDE. In particular, for given PI operators $\{\mcl{T}_0,\mcl{T}_1,\mcl{A},\mcl{B},\mcl{C},\mcl{D}\}$, we define the associated PIE as \\[-3.0ex]
	\begin{align}\label{eq:standard_PIE}
	\mcl{T}_1 \dot{w}(t)+\mcl{T}_0 \dot{\mbf{v}}(t)&= \mcl{A}\mbf{v}(t)+\mcl{B}w(t) \nonumber\\
	z(t)&=\mcl{C}\mbf{v}(t)+\mcl{D}w(t) \\[-3.5ex]	\ \nonumber
	\end{align}
	where $\mbf{v}(t)\in L_2^{n_v}$ at any time $t\geq 0$.
	
	\begin{defn}[Solution to the PIE]
		For a given input signal $w$ and given initial conditions $\mbf{v}_{\text{I}}\in L_2^{n_v}$, we say that $(\mbf{v},z)$ is a solution to the PIE defined by $\{\mcl{T}_0,\mcl{T}_1,\mcl{A},\mcl{B},\mcl{C},\mcl{D}\}$ if $\mbf{v}$ is Frech\'et differentiable, $\mbf{v}(0)=\mbf{v}_{\text{I}}$, and for all $t\geq0$, $(\mbf{v}(t),z(t))$ satisfies Eqn.~\eqref{eq:standard_PIE}.
	\end{defn}
	
	The following lemma shows that for any PDE of the form~\eqref{eq:standard_PDE} for which $\mcl{E}_0\mcl{H}_1$ in Theorem.~\ref{thm:Tmap} is invertible, there exists an equivalent PIE of the form~\eqref{eq:standard_PIE}.
	\begin{lem}\label{lem:PDE_to_PIE}
		Suppose $\mcl{T}_0,\mcl{T}_1$ are as defined in Thm.~\ref{thm:Tmap}. Let
		\begin{align*}
		 \mcl{A}&:=\bar{\mcl{A}}\circ\mcl{T}_0\in\Pi_{0112}^{\enn{v}\times \enn{v}},	&
		 \mcl{B}&:=\bar{\mcl{B}} + \bar{\mcl{A}}\circ\mcl{T}_1\in\Pi_{0112}^{\enn{v}\times \enn{w}},	\\
		 \mcl{C}&:=\bar{\mcl{C}}\circ\mcl{T}_0\in\Pi_{0112}^{\enn{z}\times \enn{v}},	&
		 \mcl{D}&:=\bar{\mcl{D}} + \bar{\mcl{C}}\circ\mcl{T}_1\in\Pi_{0112}^{\enn{z}\times \enn{w}},	\\[-3.5ex]
		\end{align*}
		where the operators $\{\bar{\mcl{A}},\bar{\mcl{B}},\bar{\mcl{C}},\bar{\mcl{D}}\}$ are parameterized by $\{A_{ij},B,C_{ij},D\}$, as in Eqn.~\eqref{eq:PDE_operators}, and where we define $\enn{v}:=\{0,0,n_v\}$, $\enn{w}:=\{n_w,0,0\}$ and $\enn{z}:=\{n_z,0,0\}$.
		Then, for a given input $w$ and initial values $\mbf{v}_{\text{I}}\in L_2^{n_v}$, $(\mbf{v},z)$ solves the PIE~\eqref{eq:standard_PIE} defined by $\{\mcl{T}_0,\mcl{T}_1,\mcl{A},\mcl{B},\mcl{C},\mcl{D}\}$ with initial conditions $\mbf{v}_{\text{I}}$ if and only if $(\bar{\mbf{v}},z)$ with $\bar{\mbf{v}}(t)=\mcl{T}_0\mbf{v}(t)+\mcl{T}_1 w(t)$ solves the PDE~\eqref{eq:standard_PDE} defined by $\{A_{ij},B,C_{ij},D,E_0,E_1\}$ with initial conditions $\bar{\mbf{v}}_{\text{I}}:=\mcl{T}_0\mbf{v}_{\text{I}}+\mcl{T}_1 w(0)$.
		
	\end{lem}
		
	\begin{proof}
		To see that the operators $\{\mcl{A},\mcl{B},\mcl{C},\mcl{D}\}$, we first note that the composition of any multiplier operator $\mbf{M}[R]$ with a PI operator is a PI operator. Noting that the sum and composition of PI operators can be expressed as PI operators as well, and studying Eqns.~\eqref{eq:PDE_operators} defining the operators $\{\bar{\mcl{A}},\bar{\mcl{B}},\bar{\mcl{C}},\bar{\mcl{D}}\}$, it suffices then to show then that the composition $\partial_x^i\partial_y^{j}\ \text{M}[S_{i,j}]\ \mcl{T}_{k}$ is a PI operator for each $k\in\{0,1\}$ and $k,\ell\in\{1,2\}$. For this, we refer to the second part of the proof of Thm.~\ref{thm:Tmap} in Appx.~\ref{sec:appx_Tmap}, in which it is shown that the composition of differential operator $\partial_x^i\partial_y^{j}$ with PI operator $\text{M}[S_{i,j}]\mcl{T}_{k}$ may indeed be expressed as a PI operator for each $k\in\{0,1\}$ and $i,j\in\{1,2\}$.
		
		Suppose $(\mbf{v},z)$ is such that $(\bar{\mbf{v}},z)$ is a solution to the PDE defined by $\{A_{ij},B,C_{ij},D,E_0,E_1\}$, where $\bar{\mbf{v}}(t)=\mcl{T}_0\mbf{v}(t)+\mcl{T}_1 w(t)$, and with initial condition $\bar{\mbf{v}}_{\text{I}}:=\mcl{T}_0\mbf{v}_{\text{I}}+\mcl{T}_1 w(0)$. Then, by Thm.~\ref{thm:Tmap},
		\begin{align*}
		\mbf{v}(0)=\mscr{D}\bar{\mbf{v}}(0)	
		=\mscr{D}\bbl[\mcl{T}_0\mbf{v}_{\text{I}}+\mcl{T}_1 w(0)\bbr]
		=\mbf{v}_{\text{I}},
		\end{align*}
		and, invoking Eqn.~\eqref{eq:standard_PDE} for the PDE defined by $\{A_{ij},B,C_{ij},D\}$,
		\begin{align*}
		\bmat{\mcl{T}_1 \dot{w}(t)+\mcl{T}_0 \dot{\mbf{v}}(t)\\z(t)}
		&=\bmat{\dot{\bar{\mbf{v}}}(t)\\z(t)}	\\
		&=\bmat{\bar{\mcl{A}}&\bar{\mcl{B}}\\ \bar{\mcl{C}}&\bar{\mcl{D}}} \bmat{\bar{\mbf{v}}(t)\\w(t)}\\
		&=\bmat{\bar{\mcl{A}}&\bar{\mcl{B}}\\ \bar{\mcl{C}}&\bar{\mcl{D}}}
		\bmat{\mcl{T}_0\mbf{v}(t)+\mcl{T}_1 w(t)\\ w(t)}	\\
		&=\bmat{\bar{\mcl{A}}\mcl{T}_0&\bar{\mcl{B}}+\bar{\mcl{A}}\mcl{T}_1 \\ \bar{\mcl{C}}\mcl{T}_0&\bar{\mcl{D}}+\bar{\mcl{C}}\mcl{T}_1}\bmat{\mbf{v}(t)\\w(t)}	\\
		&=\bmat{\mcl{A}&\mcl{B} \\ \mcl{C}&\mcl{D}}\bmat{\mbf{v}(t)\\w(t)},
		\end{align*}
		proving that $(\mbf{v},z)$ satisfies the PIE~\eqref{eq:standard_PIE}.
		
		Conversely, suppose now that $(\mbf{v},z)$ is a solution to the PIE defined by $\{\mcl{T},\mcl{A},\mcl{B},\mcl{C},\mcl{D}\}$, and with initial condition $\mbf{v}_{\text{I}}$, and let $\bar{\mbf{v}}(t)=\mcl{T}_0\mbf{v}(t)+\mcl{T}_1 w(t)$ and $\bar{\mbf{v}}_{\text{I}}=\mcl{T}_0\mbf{v}_{\text{I}}+\mcl{T}_1 w(0)$. Then, by Thm.~\ref{thm:Tmap}
		\begin{align*}
		\bar{\mbf{v}}(0)&=\mcl{T}_0\mbf{v}(0)+\mcl{T}_1 w(0)	
		=\mcl{T}_0\mbf{v}_{\text{I}}+\mcl{T}_1 w(0)
		=\bar{\mbf{v}}_{\text{I}},
		\end{align*}
		and, invoking Eqn.~\eqref{eq:standard_PIE} for the PIE defined by $\{\mcl{T},\mcl{A},\mcl{B}_j,\mcl{C}_i,\mcl{D}_{ij}\}$,
		\begin{align*}
		\bmat{\dot{\bar{\mbf{v}}}(t)\\z(t)}&=\bmat{\mcl{T}_1 \dot{w}(t)+\mcl{T}_0 \dot{\mbf{v}}(t)\\z(t)}	\\
		&=\bmat{\mcl{A}&\mcl{B} \\ \mcl{C}&\mcl{D}}\bmat{\mbf{v}(t)\\w(t)}	\\
		&=\bmat{\bar{\mcl{A}}\mcl{T}_0&\bar{\mcl{B}}+\bar{\mcl{A}}\mcl{T}_1 \\ \bar{\mcl{C}}\mcl{T}_0&\bar{\mcl{D}}+\bar{\mcl{C}}_1\mcl{T}_1}\bmat{\mbf{v}(t)\\w(t)}	\\
		&=\bmat{\bar{\mcl{A}}&\bar{\mcl{B}}\\ \bar{\mcl{C}}&\bar{\mcl{D}}}
		\bmat{\mcl{T}_0\mbf{v}(t)+\mcl{T}_1 w(t)\\ w(t)}	
		=\bmat{\bar{\mcl{A}}&\bar{\mcl{B}}\\ \bar{\mcl{C}}&\bar{\mcl{D}}}
		\bmat{\bar{\mbf{v}}(t)\\ w(t)}.
		\end{align*}
		Since also $\bar{\mbf{v}}(t)$ satisfies the boundary conditions $0=\mcl{E}_0\Lambda_{\text{bc}}\bar{\mbf{v}}(t)+\mcl{E}_1w(t)$ at each $t\geq 0$, we conclude that $(\bar{\mbf{v}},z,q)$ is a solution to the PDE~\eqref{eq:standard_PDE}.
		
	\end{proof}

	\section{Parameterizing Positive PI Operators}\label{sec:Positive_PI_Params}
	
	\begin{figure*}
		\hrulefill
		\footnotesize
		\begin{flalign}\label{eq:L_adj}
			&\quad \mcl{L}_{\text{adj}}(G)=\hat{G}=\smallbmat{\hat{B}_{00}&\hat{B}_{01}&\hat{B}_{02}&\hat{C}_{0}\\\hat{B}_{10}&\hat{B}_{11}&\hat{B}_{12}&\hat{C}_{1}\\\hat{B}_{20}&\hat{B}_{21}&\hat{B}_{22}&\hat{C}_{2}\\\hat{D}_{0}&\hat{D}_{1}&\hat{D}_{2}&\hat{N}} \in \mcl N_{0112}^{\enn{u}\times\enn{v}}
			& &\text{with} &
			\begin{array}{ll}
				\hat{B}_{jj}=\{\hat{B}_{jj,0},\hat{B}_{jj,1},\hat{B}_{jj,2}\} &\in\mcl{N}_{011}^{m_1\times n_1}\\
				\hat{C}_{j}=\{\hat{C}_{j,0},\hat{C}_{j,1},\hat{C}_{j,2}\} &\in\mcl{N}_{2D\rightarrow 1D}^{m_1\times n_2}\\
				\hat{D}_{j}=\{\hat{D}_{j,0},\hat{D}_{j,1},\hat{C}_{D,2}\} &\in\mcl{N}_{1D\rightarrow 2D}^{m_2\times n_1}\\
				\hat{N}=\smallbmat{\hat{N}_{00}&\hat{N}_{01}&\hat{N}_{02}\\\hat{N}_{10}&\hat{N}_{11}&\hat{N}_{12}\\\hat{N}_{20}&\hat{N}_{21}&\hat{N}_{22}} &\in\mcl{N}_{2D}^{m_2\times n_2}, 
			\end{array} & &
		\end{flalign}
		where $\enn{u}=\{m_0,m_1,m_2\}$ and $\enn{v}=\{n_0,n_1,n_2\}$, and where
		\begin{flalign}\label{eq:adjoint_params}
			\quad&\bmat{\hat{B}_{00}&\hat{B}_{01}(x)&\hat{B}_{02}(y)&\hat{C}_{0}(x,y)\\
				\hat{B}_{10}(x)& &\hat{B}_{12}(x,y)& \\
				\hat{B}_{20}(y)&\hat{B}_{21}(x,y)& & \\
				\hat{D}_{0}(x,y)& & & }
			=\bmat{B_{00}^T&B_{10}^T(x)&B_{20}^T(y)&D_{0}^T(x,y) \\
				B_{01}^T(x)& &B_{21}^T(x,y)&  \\
				B_{02}^T(y)&B_{12}^T(x,y)& &  \\
				C_{0}^T(x,y)& & & },	\nonumber\\
			&\{\hat{B}_{11,0}(x),\hat{B}_{11,1}(x,\theta),\hat{B}_{11,2}(x,\theta)\}
			=\{B_{11,0}^T(x),B_{11,2}^T(\theta,x),B_{11,1}^T(\theta,x)\},	\nonumber\\
			&\{\hat{B}_{22,0}(y),\hat{B}_{22,1}(y,\nu),\hat{B}_{22,2}(y,\nu)\}
			=\{B_{22,0}^T(y),B_{22,2}^T(\nu,y),B_{22,1}^T(\nu,y)\},	\nonumber\\
			&\{\hat{C}_{1,0}(x,y),\hat{C}_{1,1}(x,y,\theta),\hat{C}_{1,2}(x,y,\theta)\}
			=\{D_{1,0}^T(x,y),D_{1,2}^T(\theta,y,x),D_{1,1}^T(\theta,y,x)\},	\nonumber\\
			&\{\hat{C}_{2,0}(x,y),\hat{C}_{2,1}(x,y,\nu),\hat{C}_{2,2}(x,y,\nu)\}
			=\{D_{2,0}^T(x,y),D_{2,2}^T(x,\nu,y),D_{2,1}^T(x,\nu,y)\},	\nonumber\\
			&\{\hat{D}_{1,0}(x,y),\hat{D}_{1,1}(x,y,\theta),\hat{D}_{1,2}(x,y,\theta)\}
			=\{C_{1,0}^T(x,y),C_{1,2}^T(\theta,y,x),C_{1,1}^T(\theta,y,x)\},	\nonumber\\
			&\{\hat{D}_{2,0}(x,y),\hat{D}_{2,1}(x,y,\nu),\hat{D}_{2,2}(x,y,\nu)\}
			=\{C_{2,0}^T(x,y),C_{2,2}^T(x,\nu,y),C_{2,1}^T(x,\nu,y)\},	\nonumber\\
			&\bmat{\hat{N}_{00}(x,y)&\hat{N}_{01}(x,y,\nu)&\hat{N}_{02}(x,y,\nu)\\ \hat{N}_{10}(x,y,\theta)&\hat{N}_{11}(x,y,\theta,\nu)&\hat{N}_{12}(x,y,\theta,\nu)\\ \hat{N}_{20}(x,y,\theta)&\hat{N}_{21}(x,y,\theta,\nu)&\hat{N}_{22}(x,y,\theta,\nu)}=\bmat{N_{00}^T(x,y)&N^T_{02}(x,\nu,y)&N^T_{01}(x,\nu,y)\\ N^T_{20}(\theta,y,x)&N^T_{22}(\theta,\nu,x,y)&N^T_{21}(\theta,\nu,x,y)\\ N^T_{10}(\theta,y,x)&N^T_{12}(\theta,\nu,x,y)&N^T_{11}(\theta,\nu,x,y)} &
		\end{flalign}
		
		\vspace*{-0.1cm}
		
		\hrulefill
		\vspace*{-0.25cm}
		\caption{Parameters $\hat{G}$ defining the adjoint $\mcl{P}[\hat{G}]=\mcl{P}^*[G]$ of the PI operator $\mcl{P}[G]$ in Lem.~\ref{lem:adjoint}}\label{fig:adjoint}
		\vspace*{-0.5cm}
	\end{figure*}
	
	Using the results from Sections~\ref{sec:LPI_L2_Gain} and~\ref{sec:Representation_of_ioSystems}, a bound on the $L_2$-gain of a large class of 2D PDEs can be verified using the LPI~\eqref{eq:KYP_inequality}. In this section, we show how feasibility of such LPIs can be tested using LMIs, by parameterizing a cone of positive PI operators $\mcl{P}\in\Pi_{0112,+}^{\{n_0,0,n_2\}\times\{n_0,0,n_2\}}$ mapping $\text{Z}_{2}^{\{n_0,0,n_2\}}$
 	by positive matrices. Since positive PI operators must be self-adjoint, we first show that the adjoint of any PI operator acting on $\text{Z}_{2}^{\{n_0,0,n_2\}}$ is also a PI operator.

	\begin{lem}[Adjoint of 0112-PI operator]\label{lem:adjoint}
		Suppose $ G:=\smallbmat{B_{00}&B_{01}&B_{02}&C_{0}\\B_{10}&B_{11}&B_{12}&C_{1}\\B_{20}&B_{21}&B_{22}&C_{2}\\D_{0}&D_{1}&D_{2}&N} \in \mcl N_{0112}^{\enn{v}\times\enn{u}}$ for some $\enn{u}=\{m_0,m_1,m_2\}$ and $\enn{v}=\{n_0,n_1,n_2\}$, and where, for $j\in\{1,2\}$,
		\begin{align*}
		B_{jj}&=\{B_{jj,0},B_{jj,1},B_{jj,2}\}\in\mcl{N}_{1D}^{n_1\times m_1},	\\
		C_{j}&=\{C_{j,0},C_{j,1},C_{j,2}\}\in\mcl{N}_{2D\rightarrow 1D}^{n_1\times m_2},	\\
		D_{j}&=\{D_{j,0},D_{j,1},D_{j,2}\}\in\mcl{N}_{1D\rightarrow 2D}^{n_2\times m_1},	\\
		N&=\smallbmat{N_{00}&N_{01}&N_{02}\\N_{10}&N_{11}&N_{12}\\N_{20}&N_{21}&N_{22}}\in\mcl{N}_{2D}^{n_2\times m_2},	& &
		\end{align*}
		and let $\hat{G}:=\mcl{L}_{\text{adj}}(G)$, where $\mcl{L}_{\text{adj}}:\mcl{N}_{0112}^{\enn{v}\times\enn{u}}\rightarrow\mcl{N}_{0112}^{\enn{u}\times\enn{v}}$ is as defined in Eqn.~\eqref{eq:L_adj} in Fig.~\ref{fig:adjoint}.	Then for any $\mbf{u}\in \text{Z}_{2}^{\enn{u}}[\Omega_{ac}^{bd}]$ and $\mbf{v}\in \text{Z}_{2}^{\enn{v}}[\Omega_{ac}^{bd}]$,
		\[
		\ip{\mbf{v}}{\mcl{P}[G]\mbf{u}}_{\text{Z}_{2}^{\enn{v}}}=\ip{\mcl{P}[\hat{G}]\mbf{v}}{\mbf{u}}_{\text{Z}_{2}^{\enn{u}}}.\\
		\]
	\end{lem}
	\smallskip
	\begin{proof}
		A proof is given in Appendix~\ref{sec:appx_adjoint}.
	\end{proof}
	\medskip

	Having defined the adjoint of PI operators on $\text{Z}_{2}^{\{n_0,0,n_2\}}$, we now propose a parameterization of a cone of positive PI operators on this function space. A proof of this result can be found in Appendix.~III. 

	\begin{prop}\label{prop:positive_PI_parameterization}
		For any $Z\in L_2^{q\times n_2}[\Omega_{ac}^{bd}\times\Omega_{ac}^{bd}]$ and scalar function $g\in L_2[\Omega_{ac}^{bd}]$ satisfying $g(x,y)\geq 0$ for any $(x,y)\in\Omega_{ac}^{bd}$, let $\mcl{L}_{\text{PI}}:\R^{(9q+n_0)\times (9q+n_0)}\rightarrow\mcl{N}_{0112}^{\enn{u}\times \enn{u}}$ be as defined in Eqn.~\eqref{eq:appx_posmat_to_posPI_appendix} in Appendix~\ref{sec:appx_pos_PI}, where $\enn{u}:=\{n_0,0,n_2\}$. Then, for any $P\geq 0$, if $B=\mcl{L}_{\text{PI}}(P)$, then 
		$\mcl{P}:=\mcl{P}[B]\in\Pi_{0112}^{\enn{u}\times\enn{u}}$ satisfies $\mcl{P}^*=\mcl{P}$ and $\ip{\mbf{u}}{\mcl{P}\mbf{u}}_{\text{Z}_{2}^{\enn{u}}}\geq 0$ for any $\mbf{u}\in \text{Z}_{2}^{\enn{u}}$.
	\end{prop}
	\paragraph*{Outline of Proof}
	A PI operator $\mcl{Z}\in\Pi_{0112}^{\{0,0,n_0+n_2\}\times \{n_0,0,n_2\}}$ is defined in terms of the functions $Z\in L_2^{q\times n_2}[\Omega_{ac}^{bd}\times\Omega_{ac}^{bd}]$ and $\sqrt{g}\in L_2[\Omega_{ac}^{bd}]$. The map $\mcl{L}_{\text{PI}}:\R^{(9q+n_0)\times (9q+n_0)}\rightarrow\mcl{N}_{0112}^{\enn{u}\times\enn{u}}$ is then defined such that $\mcl{P}[P]=\mcl{Z}^*P\mcl{Z}$ for any $P\in\R^{(9q+n_0)\times (9q+n_0)}$. It follows that, for any $P\geq 0$, 
	\begin{align*}
	\ip{\mbf{u}}{\mcl{P}\mbf{u}}_{\text{Z}_{2}^{\enn{u}}}
	=\ip{\mcl{Z}\mbf{u}}{P\mcl{Z}\mbf{u}}_{L_2}
	=\ip{P^{\frac{1}{2}}\mcl{Z}\mbf{u}}{P^{\frac{1}{2}}\mcl{Z}\mbf{u}}_{L_2}
	\geq 0,
	\end{align*}
	for any $\mbf{u}\in \text{Z}_{2}^{\enn{u}}=\R^{n_0}\times L_2^{n_1}[\Omega_{ac}^{bd}]$.

	Parameterizing positive PI operators as in Prop.~\ref{prop:positive_PI_parameterization}, we use a monomial basis $Z_d$ of degree at most $d$ to define $\mcl{Z}$, yielding polynomial paramaters $B =\mcl L_{\text{PI}} (P)$ for any (positive) matrix $P$. For the scalar function $g(x,y)\geq 0$, we include the candidates\\[-3.5ex]
	\begin{align}\label{eq:psatz_funs}
	 g_0(x,y)&\!=\!1,	\hspace*{0.35cm}
	 g_1(x,y)\!=\!(x\!-\!a)(b\!-\!x)(y\!-\!c)(d\!-\!y),	
	\end{align}
	which are all nonnegative on the domain $\Omega_{ac}^{bd}:= [a,b]\times[c,d]$. We denote the resulting set of operators as $\Xi_d$, so that\\[-3.5ex]
	\begin{align*}
	\Xi_d:=\bbbl\{&\sum_{j=0}^{2}\mcl{P}[B_j]\ \bbl\lvert B_j=\mcl{L}_{\text{PI}}(P_j)
	\text{ for some }P_j\geq 0,\\[-0.6cm]
	&\hspace*{1.8cm}\text{with } Z=Z_d \text{ and } g_j(x,y) \text{ as in}~\eqref{eq:psatz_funs}\bbbr\} \\[-4.5ex]
	\end{align*}
	where now $\mcl{P}\in\Xi_d$ is an LMI constraint implying $\mcl{P}\geq 0$.

	\paragraph*{Computational complexity}
	Since the number of monomials of degree at most $d$ in 2 variables is of the order $\mcl{O}(d^2)$, the size of the matrix $P\in\mathbb{S}^{q\times q}$ parameterizing a 2D-PI operator $\mcl{P}[P]\in\Xi_{d}$ will be $q=\mcl{O}(nd^2)$, for $\mcl{P}[P]\in\Pi_{2D}^{n\times n}$. As such, the number of decision variables in the LMI $P\geq 0$ will scale with $q^2=\mcl{O}(n^2 d^4)$ -- a substantial increase compared to the $\mcl{O}(n^2 d^2)$ scaling for 1D PDEs, and the $\mcl{O}(n^2)$ scaling for ODEs.
	Nevertheless, accurate $L_2$-gain bounds for 2D PDEs can already be verified with $d=1$, as we illustrate in Section~\ref{sec:Numerical_Examples}.

	\section{An LMI For $L_2$-gain Analysis of 2D PDEs}\label{sec:L2_gain_PDE}
	
	Combining the results from the previous sections, we finally construct an LMI test for verifying an upper bound on the $L_2$-gain of a 2D PDE.
	
	\begin{thm}\label{thm:KYP_PDE}
		Let parameters $\{A_{ij},B,C_{ij},D,E_0,E_1\}$ with $E_1=0$ define a PDE of the form~\eqref{eq:standard_PDE} as in Subsection~\ref{sec:subsec:ioPDE_Representation}. Let associated operators $\{\mcl{T}_0,\mcl{T}_1,\mcl{A},\mcl{B},\mcl{C},\mcl{D}\}$ be as defined in Lemma~\ref{lem:PDE_to_PIE} in Subsection~\ref{sec:subsec:ioPIE_Representation}. Finally, let $\gamma>0$, and suppose there exists a PI operator $\mcl{P}\in\Pi_{2D}^{n_v\times n_v}$ such that $\mcl{P}-\epsilon I\in\Xi_{d_1}$ and $-\mcl{Q}\in\Xi_{d_2}$ for some $d_1,d_2\in\N$ and $\epsilon>0$, where
		{\small
		\begin{align}\label{eq:KYP_inequality_PDE}
		&\mcl{Q}:=\bmat{-\gamma I &\! \mcl{D} &\! \mcl{C}\\
			(\cdot)^* &\! -\gamma I &\! \mcl{B}^*\mcl{P}\mcl{T}_0	\\
			(\cdot)^* &\! (\cdot)^* &\! (\cdot)^* + \mcl{T}_0^*\mcl{P}\mcl{A}}.
		\end{align}
		}
		Then, for any $w\in L_2^{n_w}[0,\infty)$, if $(w(t),z(t))$ satisfies the PDE~\eqref{eq:standard_PDE} for all $t\geq 0$, then $z\in L_2^{n_z}[0,\infty)$ and $\frac{\|z\|_{L_2}}{\|w\|_{L_2}}\leq\gamma$.
	\end{thm}

	\begin{proof}
		Let the parameters $\{A_{ij},B,C_{ij},D,E_0,E_1\}$ and operators $\{\mcl{T}_0,\mcl{T}_1,\mcl{A},\mcl{B},\mcl{C},\mcl{D}\}$ be as proposed. Let $w\in L^{n_w}_2[0,\infty)$ be arbitrary, and let $(\bar{\mbf{v}},z)$ be a solution to the PDE~\eqref{eq:standard_PDE} with input $w$. Then, by Lem.~\ref{lem:PDE_to_PIE}, letting $\mbf{v}=\mscr{D}\bar{\mbf{v}}$, $(\mbf{v},z)$ is a solution to the PIE~\eqref{eq:standard_PIE} with input $w$. Since $E_1=0$, it follows by Thm.~\ref{thm:Tmap} that $\mcl{T}_1=0$, and therefore $(\mbf{v},z)$ is a solution to the PIE~\eqref{eq:primal_PIE} with $\mcl{T}=\mcl{T}_0$. Finally, by Prop.~\ref{prop:positive_PI_parameterization}, if $\mcl{P}-\epsilon I\in\Xi_{d_1}$ and $-\mcl{Q}\in\Xi_{d_2}$, we have $\mcl{P}>0$ and $\mcl{Q}\leq 0$. Then, all conditions of Lem.~\ref{lem:KYP} are satisfied, and we find that $z\in L_2^{n_z}[0,\infty)$ and $\frac{\|z\|_{L_2}}{\|w\|_{L_2}}\leq\gamma$.
	\end{proof}

	\section{Numerical Examples}\label{sec:Numerical_Examples}
	
	In this section, results of several numerical tests are presented, computing an upper bound on the $L_2$-gain of 2D PDEs using the LPI methodology proposed in the previous sections, incorporated into the MATLAB toolbox PIETOOLS~\cite{shivakumar2021PIETOOLS}. Results are shown using monomials of degree at most $d=1$ to parameterize the positive operator $\mcl{P}\in\Xi_d$ in Theorem~\ref{thm:KYP_PDE}.	
	Estimates of the $L_2$-gain computed using discretization are also shown, using a finite difference scheme on $N\times N$ uniformly distributed grid points.
		
	For each of the proposed PDEs, a regulated output $z(t)=\int_{\Omega_{a}^{b}}\int_{\Omega_{c}^{d}}\bar{\mbf{v}}(t,x,y)dydx$ is considered, corresponding to $C_{00}=I$ and $C_{ij}=0$ for all other $i,j\in\{0,1,2\}$ in Eqns.~\eqref{eq:PDE_operators} defining the parameters for the PDE~\eqref{eq:standard_PDE}. 
	
	\subsection{KISS Model}
	Consider first a particular instance of the KISS model as presented in~\cite{holmes1994partial}, with uniformly distributed disturbances on $[0,1]\times[0,1]$, and Dirichlet boundary conditions, \\[-3.25ex]
	\begin{align}\label{eq:Heat_eq}
	 &\dot{\bar{\mbf{v}}}(t)=\bbl[\partial_{x}^2\bar{\mbf{v}}(t) + \partial_{y}^2\bar{\mbf{v}}(t)\bbr] + \lambda\bar{\mbf{v}}(t) + w(t)	\nonumber\\[-0.25ex]
	 &0=\bar{\mbf{v}}(t,0,y)=\bar{\mbf{v}}(t,1,y)=\bar{\mbf{v}}(t,x,0)=\bar{\mbf{v}}(t,x,1). \\[-4.0ex]
	 \ \nonumber
	\end{align}
	Figure~\ref{fig:Heat_eq_L2_gain} presents bounds on the $L_2$-gain of this system for $\lambda\in[9,19]$, computed using the LPI approach. Gains estimated using discretization with $N=12$ grid points are also displayed. The results show that the LPI method is able to achieve (provably valid) bounds on the $L_2$-gain that are lower than the values estimated through discretization. 
	
	\begin{figure}[h!]
		\centering
		\hspace*{-0cm}\includegraphics[width=0.475\textwidth]{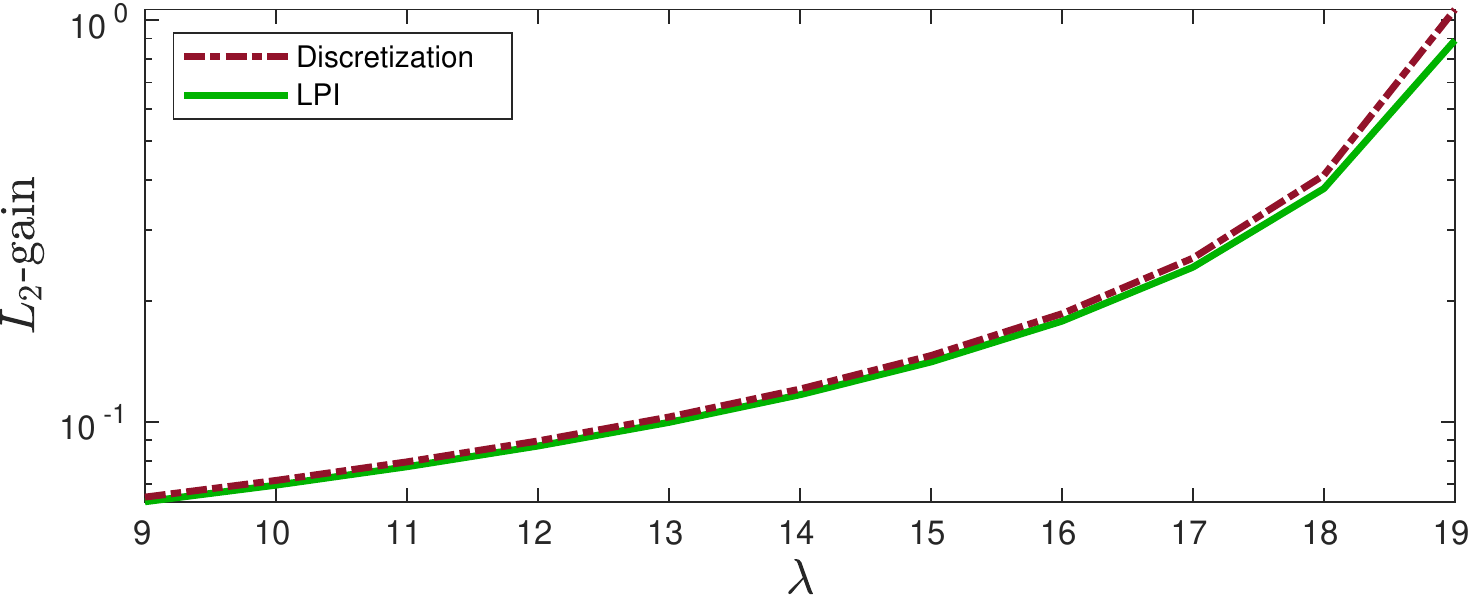}
		\vspace*{-0.3cm}
		\caption{\footnotesize 
			Bounds on the $L_2$-gain of System~\eqref{eq:Heat_eq} computed using the LPI methodology, parameterizing $\mcl{P}\in\Xi_d$ in Thm.~\ref{thm:KYP_PDE} using monomials of degree at most $d=1$. Estimates of the gain computed through discretization are also shown, using a grid of $12\times 12$ uniformly distributed points.
		}
		\label{fig:Heat_eq_L2_gain}
		\vspace*{-0.4cm}
	\end{figure}

	\subsection{Other Parabolic Systems}\label{sec:Num_example_2}
	
	Consider now a bound on the $L_2$-gain computed using the LPI approach, and an estimated gain computed using discretizaion, for each of the following variations on System~\eqref{eq:Heat_eq}, where $g(x,y):= 1-2(x-0.5)^2 + 2(y-0.5)^2$:
	
	1) Using an inhomogeneously distributed reaction term:\\[-3.5ex]
	\begin{align*}
	\dot{\bar{\mbf{v}}}(t)&=\bl[\bar{\mbf{v}}_{xx}(t) + \bar{\mbf{v}}_{yy}(t)\br] + g(x,y)\bar{\mbf{v}}(t) + w(t).	\\[-3.75ex]
	\end{align*}
	
	2) Using an inhomogeneously distributed disturbance:\\[-3.5ex]
		\begin{align*}
		\dot{\bar{\mbf{v}}}(t)&=\bl[\bar{\mbf{v}}_{xx}(t) + \bar{\mbf{v}}_{yy}(t)\br] + \bar{\mbf{v}}(t) + g(x,y)w(t).	\\[-3.75ex]
		\end{align*}
	
	3) Using $\lambda=1$ and Neumann boundary conditions:\\[-3.5ex]
		\begin{align*}
			0&=\bar{\mbf{v}}(t,0,y)=\partial_x\bar{\mbf{v}}(t,1,y)=\bar{\mbf{v}}(t,x,0)=\partial_y\bar{\mbf{v}}(t,x,1). \\[-3.75ex]	
		\end{align*}

	The results of each test are provided in Table~\ref{tab:Numerical_results}, along with the required CPU times. The results once more show that the LPI method is able to produce bounds on the $L_2$-gain which are smaller than the estimates obtained through discretization, in relatively short time.

\begin{table}[!h]
	\begin{tabular}{l||l|l l l | l}
			
		& & \multicolumn{3}{c|}{Discretization} & LPI\\ [0.25ex] \hline\hline
		& & $N=6$ & $N=9$ & $N=12$ & $d=1$  \\ [0.3ex] \hline 
	\multirow{2}{0.3em}{\!1)\!}  & \!$L_2$-Gain	& $0.0404$ & $0.0384$ & $0.0376$ &  \!$0.0367$  \\ [0.1ex] 
		& \!CPU Time (s)\!\! & $5.56$ & $6.58\cdot 10^2$ & $3.75\cdot 10^4$ & \!$1.73\cdot 10^4$ \\ [0.5ex]  
	\multirow{2}{0.3em}{\!2)\!}  & \!$L_2$-Gain & $0.0315$  & $0.0302$  & $0.0298$ & \!$0.0293$ \\ [0.1ex] 
		& \!CPU Time (s)\!\! & $3.91$ & $6.59\cdot 10^2$ & $3.76\cdot 10^4$ & \!$2.64\cdot 10^5$ \\ [0.5ex]  
	\multirow{2}{0.3em}{\!3)\!}  & \!$L_2$-Gain	& $0.1793$  & $0.1767$  & $0.1758$ & \!$0.1747$ \\ [0.1ex] 
		& \!CPU Time (s)\!\! & $3.77$ & $6.59\cdot 10^2$ & $4.09\cdot 10^{4}$ & \!$1.32\cdot 10^4$ \\ 
	\end{tabular}
	\caption{\normalfont \footnotesize 
		Bounds on the $L_2$-gain for variations 1 through 3 on System~\eqref{eq:Heat_eq} \\[-0.1ex]
		computed using the LPI approach, along with the CPU time required for \\[-0.1ex]
		each test. Estimates computed using discretization are also provided, using $N\times N$ uniformly distributed grid points. \\[-0.1ex]
	}
	\label{tab:Numerical_results}
		\vspace*{-1.05cm}
\end{table}

	\section{Conclusion}
	
	In this paper, a new method for estimating the $L_2$-gain of linear, 2nd order, 2D PDEs, using semidefinite programming was presented. To this end, it was proved that any such PDE can be equivalently represented by a PIE, and the necessary formulae to convert between the representations was derived. It was further proved that the problem of verifying an upper bound on the $L_2$-gain of a PIE can be posed as an LPI, and a method for parameterizing such LPIs as LMIs was presented. Implementing this approach in the MATLAB toolbox PIETOOLS, relatively accurate bounds on the $L_2$-gain of several PDEs could be numerically computed.
	
	\vspace*{-0.1cm}

	\bibliographystyle{IEEEtran}
	\bibliography{bibfile}
	
	\clearpage
	
	\begin{appendices}

		\section{Properties of PI Operators}
		
		\subsection{Composition of Derivative and PI Operator}\label{sec:appx_div_operator}
		
		In this section, we prove that the composition of a differential operator $\partial_x$ and a suitable $0112$-PI operator can be expressed as a $0112$-PI operator. A similar result can be derived for the composition with a differential operator $\partial_y$, as shown in Lem.~\ref{lem:appx_div_operator_y}.
		
		\begin{lem}[Composition of $x$-derivative and PI operator]\label{lem:appx_div_operator_x}
			Suppose $G:=\smallbmat{()&()&()&()\\()&()&()&()\\()&()&()&()\\B&C&D&N}\in\mcl{N}_{0112}^{\{0,0,n_2\}\times\{m_0,m_1,m_2\}}$ where
			\begin{align*}
			B& 	& &\in W_1[\Omega_{ac}^{bd}]	\nonumber\\
			C&=\{0,C_1,C_2\}	& &\in\{0,W_1,W_1\}\subset\mcl{N}_{1D\rightarrow 2D},	\\
			D&=\{D_0,D_1,D_2\}	& &\in\{W_1,W_1,W_1\}\subset\mcl{N}_{1D\rightarrow 2D},	\\
			N\! &=\! \bmat{0&\! \! 0&\! \! 0\\N_{10}&\! \! N_{11}&\! \! N_{12}\\N_{20}&\! \! N_{21}&\! \! N_{22}} & &\in \bmat{
				0               &\! \! \! 0              &\! \! \!0\\
				W_1 &\! \! \! W_1&\! \! \! W_1\\
				W_1 &\! \! \! W_1&\! \! \! W_1}\! \subset\! \mcl{N}_{2D}^{n\times m}.
			\end{align*}
			Let $H:=\smallbmat{()&()&()&()\\()&()&()&()\\()&()&()&()\\P&Q&R&M}\in\mcl{N}_{0112}^{\{n_0,n_1,n_2\}\times\{m_0,m_1,m_2\}}$ with
			\begin{align}\label{eq:appx_div_operator_x}
			P& & &\in L_2[\Omega_{ac}^{bd}]	\nonumber\\
			Q&=\{Q_0,Q_1,Q_2\}	& &\in\mcl{N}_{1D\rightarrow 2D},	\nonumber\\
			R&=\{R_0,R_1,R_2\}	& &\in\mcl{N}_{1D\rightarrow 2D},	\nonumber\\
			M&=\bmat{M_{00}&M_{01}&M_{02}\\M_{10}&M_{11}&M_{12}\\M_{20}&M_{21}&M_{22}}&	&\in\mcl{N}_{2D}^{n\times m},	&	&\hspace*{1.0cm}
			\end{align}
			where, for $i,j\in\{1,2\}$,
			\begin{align*}
			&P(x,y)=\partial_x B(x,y),	\\
			&Q_0(x,y)=C_1(x,y,x)-C_2(x,y,x),	\\
			&Q_i(x,y,\theta)=\partial_x C_i(x,y,\theta),	\\
			&R_0(x,y)=\partial_x D_0(x,y),	\\
			&R_j(x,y,\nu)=\partial_x D_j(x,y,\nu),	\\
			&M_{00}(x,y)=N_{10}(x,y,x)-N_{20}(x,y,x),   \\
			&M_{i0}(x,y,\theta)=\partial_x N_{i0}(x,y,\theta),\\
			&M_{0j}(x,y,\nu)=N_{1j}(x,y,x,\nu)-N_{2j}(x,y,x,\nu),   \\
			&M_{ij}(x,y,\theta,\nu)=\partial_x N_{ij}(x,y,\theta,\nu).	& &
			\end{align*}
			Then, for any $\mbf{v}\in\text{Z}_{2}^{\{m_0,m_1,m_2\}}[\Omega_{ac}^{bd}]$,
			\begin{align*}
			\partial_x(\mcl{P}[G]\mbf{v})(x,y)
			=(\mcl{P}[H]\mbf{v})(x,y).\\
			\end{align*}
		\end{lem}
		
		\begin{proof}
			To prove this result, we exploit the linearity of PI operators, noting that for arbitrary $\mbf{v}:=\smallbmat{v_0\\\mbf{v}_x\\\mbf{v}_y\\\mbf{v}_2}\in\text{Z}_{2}^{\{m_0,m_1,m_2\}}[\Omega_{ac}^{bd}]$
			\begin{align*}
			\bl(\mcl{P}[G]\mbf{v}\br)(x,y)&=\left(\mcl{P}\smallbmat{()&()&()&()\\()&()&()&()\\()&()&()&()\\B&C&D&N}\smallbmat{v_0\\\mbf{v}_x\\\mbf{v}_y\\\mbf{v}_2}\right)(x,y)	\\
			&=\bl(\mcl{P}[B]v_0\br)(x,y) + \bl(\mcl{P}[C]\mbf{v}_x\br)(x,y)\\
			&\qquad+ \bl(\mcl{P}[D]\mbf{v}_y\br)(x,y) + \bl(\mcl{P}[N]\mbf{v}_2\br)(x,y)
			\end{align*}
			Since $v_0$ and $\mbf{v}_y$ do not depend on $x$, it is easy to see that
			\begin{align*}
			\partial_x\bl(\mcl{P}[B]v_0\br)(x,y)&=\partial_x\bl(B(x,y)v_0\br) \\
			&=P(x,y)v_0=\bl(\mcl{P}[P]v_0\br)(x,y),
			\end{align*}
			and similarly
			\begin{align*}
			&\partial_x(\mcl{P}[D]\mbf{v}_y)(x,y)=\partial_xD_{0}(x,y)\mbf{v}_y(y)\\
			&\ +\partial_x\int_{c}^{y} D_{1}(x,y,\nu)\mbf{v}_y(\nu)d\nu
			+\partial_x\int_{y}^{d}D_{2}(x,y,\nu)\mbf{v}_y(\nu)d\nu	\\
			&=R_{0}(x,y)\mbf{v}(y)+\int_{c}^{y} R_{1}(x,y,\nu)\mbf{v}_y(\nu)d\nu \\
			&\hspace*{2.00cm}+\int_{y}^{d}R_{2}(x,y,\nu)\mbf{v}_y(\nu)d\nu
			=(\mcl{P}[R]\mbf{v}_y)(x,y)
			\end{align*}
			For the remaining terms, we recall the Leibniz integral rule, stating that, for differentiable $K\in L_2[\Omega_{a}^{b}\times\Omega_{a}^{b}]$,
			\begin{align*}
			&\frac{d}{dx}\left(\int_{a(x)}^{b(x)} K(x,\theta)d\theta\right)
			=\int_{a(x)}^{b(x)} \partial_x K(x,\theta)d\theta \\
			&\hspace*{2.0cm}+K(x,b(x))\frac{d}{dx}b(x)
			-P(x,a(x))\frac{d}{dx}a(x).
			\end{align*}
			Using this result, it follows that
			\begin{align*}
			&\partial_x(\mcl{P}[C]\mbf{v}_x)(x,y)	\\
			&=\partial_x\left(\int_{a}^{x} C_{1}(x,y,\theta)\mbf{v}_x(\theta)d\theta
			+\int_{x}^{b}C_{2}(x,y,\theta)\mbf{v}_x(\theta)d\theta\right)	\\
			&=C_{1}(x,y,x)\mbf{v}_x(x)
			+\int_{a}^{x}\partial_{x}C_{1}(x,y,\theta)\mbf{v}_x(\theta)d\theta \\
			&\qquad-C_{2}(x,y,x)\mbf{v}_x(x)
			+\int_{x}^{b}\partial_{x}C_{2}(x,y,\theta)\mbf{v}_x(\theta)d\theta \\
			&=(\mcl{P}[Q]\mbf{v}_x)(x,y),
			\end{align*}
			and similarly, as also shown in Lem.~6 in~\cite{jagt2021PIEArxiv},
			\begin{align*}
			\partial_x(\mcl{P}[N]\mbf{v}_2)(x,y)=(\mcl{P}[M]\mbf{v}_2)(x,y).
			\end{align*}
			Combining the results, we conclude that
			\begin{align*}
			&\partial_x\bl(\mcl{P}[G]\mbf{v}\br)(x,y)	\\
			&\ =\partial_x\bl(\mcl{P}[B]v_0\br)(x,y) + \partial_x\bl(\mcl{P}[C]\mbf{v}_x\br)(x,y)\\
			&\quad+ \partial_x\bl(\mcl{P}[D]\mbf{v}_y\br)(x,y) + \partial_x\bl(\mcl{P}[N]\mbf{v}_2\br)(x,y)	\\
			&\ =\bl(\mcl{P}[P]v_0\br)(x,y) + \bl(\mcl{P}[Q]\mbf{v}_x\br)(x,y)\\
			&\quad+\bl(\mcl{P}[R]\mbf{v}_y\br)(x,y) + \bl(\mcl{P}[M]\mbf{v}_2\br)(x,y)	
			= \bl(\mcl{P}[H]\mbf{v}\br)(x,y).
			\end{align*}
			
		\end{proof}
		
		\begin{lem}\label{lem:appx_div_operator_y}
			Suppose $G:=\smallbmat{()&()&()&()\\()&()&()&()\\()&()&()&()\\B&C&D&N}\in\mcl{N}_{0112}^{\{0,0,n_2\}\times\{m_0,m_1,m_2\}}$ where
			\begin{align*}
			B & & &\in W_1^{n_2\times m_0}[\Omega_{ac}^{bd}]\nonumber\\
			C&=\{C_0,C_1,C_2\}	& &\in\{W_1,W_1,W_1\}\subset\mcl{N}^{n_2\times m_1}_{1D\rightarrow 2D},	\\
			D&=\{0,D_1,D_2\}	& &\in\{W_1,W_1,W_1\}\subset\mcl{N}^{n_2\times m_1}_{1D\rightarrow 2D},	\\
			N\! &=\! \bmat{0&\! \! N_{01}&\! \! N_{02}\\0&\! \! N_{11}&\! \! N_{12}\\0&\! \! N_{21}&\! \! N_{22}} & &\in \bmat{
				0 &\! \! \! W_1 &\! \! \! W_1\\
				0 &\! \! \! W_1 &\! \! \! W_1\\
				0 &\! \! \! W_1 &\! \! \! W_1}\! \subset\! \mcl{N}_{2D}^{n_2\times m_2},
			\end{align*}
			and let $H:=\smallbmat{()&()&()&()\\()&()&()&()\\()&()&()&()\\P&Q&R&M}\in\mcl{N}_{0112}^{\{0,0,n_2\}\times\{m_0,m_1,m_2\}}$ with
			\begin{align}\label{eq:div_operator_y}
			P& & &\in L^{n_2\times m_0}_2[\Omega_{ac}^{bd}]	\nonumber\\
			Q&=\{Q_0,Q_1,Q_2\}	& &\in\mcl{N}^{n_2\times m_1}_{1D\rightarrow 2D},	\nonumber\\
			R&=\{R_0,R_1,R_2\}	& &\in\mcl{N}^{n_2\times m_1}_{1D\rightarrow 2D},	\nonumber\\
			M&=\bmat{M_{00}&M_{01}&M_{02}\\M_{10}&M_{11}&M_{12}\\M_{20}&M_{21}&M_{22}}&	&\in\mcl{N}_{2D}^{n_2\times m_2},	&	&\hspace*{1.0cm}
			\end{align}
			where, for $i,j\in\{1,2\}$,		
			\begin{align*}
			&P(x,y)=\partial_y B(x,y),	\\
			&Q_0(x,y)=\partial_y C_0(x,y),	\\
			&Q_i(x,y,\theta)=\partial_y C_i(x,y,\theta),	\\
			&R_0(x,y)=D_1(x,y,y)-D_2(x,y,y),	\\
			&R_j(x,y,\nu)=\partial_y D_j(x,y,\nu),	\\
			&M_{00}(x,y)=N_{01}(x,y,y)-N_{02}(x,y,y),    \\
			&M_{i0}(x,y,\theta)=N_{i1}(x,y,\theta,y)-N_{i2}(x,y,\theta,y),   \\
			&M_{0j}(x,y,\nu)=\partial_y N_{0j}(x,y,\nu), \\
			&M_{ij}(x,y,\theta,\nu)=\partial_y N_{ij}(x,y,\theta,\nu). &
			\end{align*}
			Then, for any $\mbf{v}\in\text{Z}_{2}^{\{m_0,m_1,m_2\}}[\Omega_{ac}^{bd}]$,
			\begin{align*}
			\partial_y(\mcl{P}[G]\mbf{v})(x,y)
			=(\mcl{P}[H]\mbf{v})(x,y).\\
			\end{align*}
		\end{lem}
		\begin{proof}
			The proof follows along the same lines as that of Lemma~\ref{lem:appx_div_operator_x}.
		\end{proof}

		\subsection{Adjoint of 0112-PI Operators}\label{sec:appx_adjoint}
		
		\begin{figure*}
			\hrulefill
			\footnotesize
			\begin{flalign}\label{eq:appx_L_adj}
			&\quad \mcl{L}_{\text{adj}}(G)=\hat{G}=\smallbmat{\hat{B}_{00}&\hat{B}_{01}&\hat{B}_{02}&\hat{C}_{0}\\\hat{B}_{10}&\hat{B}_{11}&\hat{B}_{12}&\hat{C}_{1}\\\hat{B}_{20}&\hat{B}_{21}&\hat{B}_{22}&\hat{C}_{2}\\\hat{D}_{0}&\hat{D}_{1}&\hat{D}_{2}&\hat{N}} \in \mcl N_{0112}^{\enn{u}\times\enn{v}}
			& &\text{with} &
			\begin{array}{ll}
			\hat{B}_{jj}=\{\hat{B}_{jj,0},\hat{B}_{jj,1},\hat{B}_{jj,2}\} &\in\mcl{N}_{011}^{m_1\times n_1}\\
			\hat{C}_{j}=\{\hat{C}_{j,0},\hat{C}_{j,1},\hat{C}_{j,2}\} &\in\mcl{N}_{2D\rightarrow 1D}^{m_1\times n_2}\\
			\hat{D}_{j}=\{\hat{D}_{j,0},\hat{D}_{j,1},\hat{C}_{D,2}\} &\in\mcl{N}_{1D\rightarrow 2D}^{m_2\times n_1}\\
			\hat{N}=\smallbmat{\hat{N}_{00}&\hat{N}_{01}&\hat{N}_{02}\\\hat{N}_{10}&\hat{N}_{11}&\hat{N}_{12}\\\hat{N}_{20}&\hat{N}_{21}&\hat{N}_{22}} &\in\mcl{N}_{2D}^{m_2\times n_2}, 
			\end{array} & &
			\end{flalign}
			where $\enn{u}=\{m_0,m_1,m_2\}$ and $\enn{v}=\{n_0,n_1,n_2\}$, and where
			\begin{flalign}\label{eq:appx_adjoint_params}
			\quad&\bmat{\hat{B}_{00}&\hat{B}_{01}(x)&\hat{B}_{02}(y)&\hat{C}_{0}(x,y)\\
				\hat{B}_{10}(x)& &\hat{B}_{12}(x,y)& \\
				\hat{B}_{20}(y)&\hat{B}_{21}(x,y)& & \\
				\hat{D}_{0}(x,y)& & & }
			=\bmat{B_{00}^T&B_{10}^T(x)&B_{20}^T(y)&D_{0}^T(x,y) \\
				B_{01}^T(x)& &B_{21}^T(x,y)&  \\
				B_{02}^T(y)&B_{12}^T(x,y)& &  \\
				C_{0}^T(x,y)& & & },	\nonumber\\
			&\{\hat{B}_{11,0}(x),\hat{B}_{11,1}(x,\theta),\hat{B}_{11,2}(x,\theta)\}
			=\{B_{11,0}^T(x),B_{11,2}^T(\theta,x),B_{11,1}^T(\theta,x)\},	\nonumber\\
			&\{\hat{B}_{22,0}(y),\hat{B}_{22,1}(y,\nu),\hat{B}_{22,2}(y,\nu)\}
			=\{B_{22,0}^T(y),B_{22,2}^T(\nu,y),B_{22,1}^T(\nu,y)\},	\nonumber\\
			&\{\hat{C}_{1,0}(x,y),\hat{C}_{1,1}(x,y,\theta),\hat{C}_{1,2}(x,y,\theta)\}
			=\{D_{1,0}^T(x,y),D_{1,2}^T(\theta,y,x),D_{1,1}^T(\theta,y,x)\},	\nonumber\\
			&\{\hat{C}_{2,0}(x,y),\hat{C}_{2,1}(x,y,\nu),\hat{C}_{2,2}(x,y,\nu)\}
			=\{D_{2,0}^T(x,y),D_{2,2}^T(x,\nu,y),D_{2,1}^T(x,\nu,y)\},	\nonumber\\
			&\{\hat{D}_{1,0}(x,y),\hat{D}_{1,1}(x,y,\theta),\hat{D}_{1,2}(x,y,\theta)\}
			=\{C_{1,0}^T(x,y),C_{1,2}^T(\theta,y,x),C_{1,1}^T(\theta,y,x)\},	\nonumber\\
			&\{\hat{D}_{2,0}(x,y),\hat{D}_{2,1}(x,y,\nu),\hat{D}_{2,2}(x,y,\nu)\}
			=\{C_{2,0}^T(x,y),C_{2,2}^T(x,\nu,y),C_{2,1}^T(x,\nu,y)\},	\nonumber\\
			&\bmat{\hat{N}_{00}(x,y)&\hat{N}_{01}(x,y,\nu)&\hat{N}_{02}(x,y,\nu)\\ \hat{N}_{10}(x,y,\theta)&\hat{N}_{11}(x,y,\theta,\nu)&\hat{N}_{12}(x,y,\theta,\nu)\\ \hat{N}_{20}(x,y,\theta)&\hat{N}_{21}(x,y,\theta,\nu)&\hat{N}_{22}(x,y,\theta,\nu)}=\bmat{N_{00}^T(x,y)&N^T_{02}(x,\nu,y)&N^T_{01}(x,\nu,y)\\ N^T_{20}(\theta,y,x)&N^T_{22}(\theta,\nu,x,y)&N^T_{21}(\theta,\nu,x,y)\\ N^T_{10}(\theta,y,x)&N^T_{12}(\theta,\nu,x,y)&N^T_{11}(\theta,\nu,x,y)} &
			\end{flalign}
			
			\hrulefill
			\vspace*{-0.2cm}
			\caption{Parameters $\hat{G}$ defining the adjoint $\mcl{P}[\hat{G}]=\mcl{P}^*[G]$ of the PI operator $\mcl{P}[G]$ in Lem.~\ref{lem:appx_adjoint}}\label{fig:appx_adjoint}
			\vspace*{-0.4cm}
		\end{figure*}
		
		\begin{lem}[Adjoint of 0112-PI operator]\label{lem:appx_adjoint}
			Suppose $ G:=\smallbmat{B_{00}&B_{01}&B_{02}&C_{0}\\B_{10}&B_{11}&B_{12}&C_{1}\\B_{20}&B_{21}&B_{22}&C_{2}\\D_{0}&D_{1}&D_{2}&N} \in \mcl N_{0112}^{\enn{v}\times\enn{u}}$ for some $\enn{u}=\{m_0,m_1,m_2\}$ and $\enn{v}=\{n_0,n_1,n_2\}$, and where, for $j\in\{1,2\}$,
			\begin{align*}
			B_{jj}&=\{B_{jj,0},B_{jj,1},B_{jj,2}\}\in\mcl{N}_{1D}^{n_1\times m_1},	\\
			C_{j}&=\{C_{j,0},C_{j,1},C_{j,2}\}\in\mcl{N}_{2D\rightarrow 1D}^{n_1\times m_2},	\\
			D_{j}&=\{D_{j,0},D_{j,1},D_{j,2}\}\in\mcl{N}_{1D\rightarrow 2D}^{n_2\times m_1},	\\
			N&=\smallbmat{N_{00}&N_{01}&N_{02}\\N_{10}&N_{11}&N_{12}\\N_{20}&N_{21}&N_{22}}\in\mcl{N}_{2D}^{n_2\times m_2},	& &
			\end{align*}
			and let $\hat{G}:=\mcl{L}_{\text{adj}}(G)$, where $\mcl{L}_{\text{adj}}:\mcl{N}_{0112}^{\enn{v}\times\enn{u}}\rightarrow\mcl{N}_{0112}^{\enn{u}\times\enn{v}}$ is as defined in Eqn.~\eqref{eq:appx_L_adj} in Fig.~\ref{fig:appx_adjoint}.	Then for any $\mbf{u}\in \text{Z}_{2}^{\enn{u}}[\Omega_{ac}^{bd}]$ and $\mbf{v}\in \text{Z}_{2}^{\enn{v}}[\Omega_{ac}^{bd}]$,
			\[
			\ip{\mbf{v}}{\mcl{P}[G]\mbf{u}}_{\text{Z}_{2}^{\enn{v}}}=\ip{\mcl{P}[\hat{G}]\mbf{v}}{\mbf{u}}_{\text{Z}_{2}^{\enn{u}}}.
			\]
		\end{lem}
		
		\begin{proof}
			
			Let $\mbf{u}=\smallbmat{u_0\\\mbf{u}_x\\\mbf{u}_y\\\mbf{u}_2}\in\smallbmat{\R^{m_0}\\L_2^{m_1}[\Omega_{a}^{b}]\\L_2^{m_1}[\Omega_{c}^{d}]\\L_2^{m_2}[\Omega_{ac}^{bd}]}$ and $\mbf{v}=\smallbmat{v_0\\\mbf{v}_x\\\mbf{v}_y\\\mbf{v}_2}\in\smallbmat{\R^{m_0}\\L_2^{m_1}[\Omega_{a}^{b}]\\L_2^{m_1}[\Omega_{c}^{d}]\\L_2^{m_2}[\Omega_{ac}^{bd}]}$ be arbitrary, and let $G\in\mcl{N}_{0112}^{\enn{v}\times\enn{u}}$ and $\hat{G}\in\mcl{N}_{0112}^{\enn{u}\times\enn{v}}$ for $\enn{u}:=\{m_0,m_1,m_2\}$ and $\enn{v}:=\{n_0,n_1,n_2\}$ be as defined. By linearity of PI operators, we note that
			\begin{align}\label{eq:appx_ip_expanded}
			&\ip{\mbf{v}}{\mcl{P}[G]\mbf{u}}_{Z^{\enn{v}}}=\ip{\smallbmat{v_0\\\mbf{v}_x\\\mbf{v}_y\\\mbf{v}_2}}{\mcl{P}\smallbmat{B_{00}&B_{01}&B_{02}&C_{0}\\B_{10}&B_{11}&B_{12}&C_{1}\\B_{20}&B_{21}&B_{22}&C_{2}\\D_{0}&D_{1}&D_{2}&N}\smallbmat{u_0\\\mbf{u}_x\\\mbf{u}_y\\\mbf{u}_2}}_{Z^{\enn{v}}}	\nonumber\\
			&\quad=\ip{v_0}{\mcl{P}\smallbmat{B_{00}&B_{01}&B_{02}&C_{0}\\()&()&()&()\\()&()&()&()\\()&()&()&()}\smallbmat{u_0\\\mbf{u}_x\\\mbf{u}_y\\\mbf{u}_2}}	\nonumber\\
			&\quad\qquad+\ip{\mbf{v}_x}{\mcl{P}\smallbmat{()&()&()&()\\B_{10}&B_{11}&B_{12}&C_{1}\\()&()&()&()\\()&()&()&()}\smallbmat{u_0\\\mbf{u}_x\\\mbf{u}_y\\\mbf{u}_2}}_{L_2}
			\nonumber\\
			&\quad\qquad\qquad+\ip{\mbf{v}_y}{\mcl{P}\smallbmat{()&()&()&()\\()&()&()&()\\B_{20}&B_{21}&B_{22}&C_{2}\\()&()&()&()}\smallbmat{u_0\\\mbf{u}_x\\\mbf{u}_y\\\mbf{u}_2}}_{L_2}\nonumber\\
			&\quad\qquad\qquad\qquad+\ip{\mbf{v}_2}{\mcl{P}\smallbmat{()&()&()&()\\()&()&()&()\\()&()&()&()\\D_{0}&D_{1}&D_{2}&N}\smallbmat{u_0\\\mbf{u}_x\\\mbf{u}_y\\\mbf{u}_2}}_{L_2}
			\end{align}
			We consider each of these terms separately, starting with the finite-dimensional inner product, for which it is clear that
			\begin{align*}
			&\ip{v_0}{\mcl{P}\smallbmat{B_{00}&B_{01}&B_{02}&C_{0}\\()&()&()&()\\()&()&()&()\\()&()&()&()}\smallbmat{u_0\\\mbf{u}_x\\\mbf{u}_y\\\mbf{u}_2}}	\\
			&\qquad=v_0^T B_{00}u_0 + v_0^T\int_{a}^{b}B_{01}(x)\mbf{u}_x(x)dx \\
			&\qquad\qquad+ v_0^T\int_{c}^{d}B_{02}(y)\mbf{u}_y(y)dy\nonumber\\
			&\qquad\qquad\qquad+ v_0^T\int_{a}^{b}\int_{c}^{d}C_{0}(x,y)\mbf{u}_2(x,y)dydx \nonumber\\
			&\qquad=\bbl[B_{00}^T v_0\bbr]^T u_0 + \int_{a}^{b}\bbl[B_{01}^T(x)v_0\bbr]^T\mbf{u}_x(x)dx \nonumber\\
			&\qquad\qquad+ \int_{c}^{d}\bbl[B_{02}^T(y)v_0\bbr]^T\mbf{u}_y(y)dy\\
			&\qquad\qquad\qquad+ \int_{a}^{b}\int_{c}^{d}\bbl[C_{0}^T(x,y)v_0\bbr]^T\mbf{u}_2(x,y)dydx 7\nonumber\\
			&\qquad = \ip{\mcl{P}\smallbmat{\hat{B}_{00}&()&()&()\\\hat{B}_{10}&()&()&()\\\hat{B}_{20}&()&()&()\\\hat{D}_{0}&()&()&()}v_0}{\smallbmat{u_0\\\mbf{u}_x\\\mbf{u}_y\\\mbf{u}_2}}_{\text{Z}_{2}^{\enn{u}}}.
			\end{align*}
			For the remaining terms, we use the results in~\cite{peet2019PIE_representation} (Lem.~6) and~\cite{jagt2021PIEArxiv} (Lem.~8), providing the adjoints of 1D- and 2D-PI operators. From these results, it follows that
			\begin{align*}
			\ip{\mbf{v}_x}{\mcl{P}[B_{11}]\mbf{u}_x}_{L_2}
			&=\ip{\mcl{P}[\hat{B}_{11}]\mbf{v}_x}{\mbf{u}_x}_{L_2}	\\
			\ip{\mbf{v}_y}{\mcl{P}[B_{22}]\mbf{u}_y}_{L_2},
			&=\ip{\mcl{P}[\hat{B}_{22}]\mbf{v}_y}{\mbf{u}_y}_{L_2}	\\
			\ip{\mbf{v}_2}{\mcl{P}[N]\mbf{u}_2}_{L_2},
			&=\ip{\mcl{P}[\hat{N}]\mbf{v}_2}{\mbf{u}_2}_{L_2},	\\
			\end{align*}
			and it is straightforward to show that also
			\begin{align*}
			\ip{\mbf{v}_x}{\mcl{P}[C_1]\mbf{u}_2}_{L_2}
			&=\ip{\mcl{P}[\hat{D}_1]\mbf{v}_x}{\mbf{u}_2}_{L_2},	\\
			\ip{\mbf{v}_y}{\mcl{P}[C_2]\mbf{u}_2}_{L_2}
			&=\ip{\mcl{P}[\hat{D}_2]\mbf{v}_y}{\mbf{u}_2}_{L_2},	\\
			\ip{\mbf{v}_2}{\mcl{P}[D_1]\mbf{u}_x}_{L_2}
			&=\ip{\mcl{P}[\hat{C}_1]\mbf{v}_2}{\mbf{u}_x}_{L_2},	\\
			\ip{\mbf{v}_2}{\mcl{P}[D_2]\mbf{u}_y}_{L_2}
			&=\ip{\mcl{P}[\hat{C}_2]\mbf{v}_2}{\mbf{u}_y}_{L_2}.
			\end{align*}
			Applying these relations, we find that
			\begin{align*}
			&\ip{\mbf{v}_x}{\mcl{P}\smallbmat{()&()&()&()\\B_{10}&B_{11}&B_{12}&C_1\\()&()&()&()\\()&()&()&()}\smallbmat{u_0\\\mbf{u}_x\\\mbf{u}_y\\\mbf{u}_2}}_{L_2}	\\
			&\qquad=\int_{a}^{b}\mbf{v}_x^T(x)B_{10}(x)u_0dx + \ip{\mbf{v}_x}{\mcl{P}[B_{11}]\mbf{u}_x}_{L_2}	\\
			&\qquad\qquad+\int_{a}^{b}\mbf{v}_x^T(x)\int_{c}^{d}B_{12}(x,y)\mbf{u}_y(y)dydx	\\
			&\qquad\qquad\qquad +\ip{\mbf{v}_x}{\mcl{P}[C_1]\mbf{u}_2}_{L_2}	\\
			&\qquad=\left[\int_{a}^{b}B_{10}^T(x)\mbf{v}_x  dx\right]^T u_0 + \ip{\mcl{P}[\hat{B}_{11}]\mbf{v}_x}{\mbf{u}_x}_{L_2}	\\
			&\qquad\qquad+ \int_{c}^{d}\left[\int_{a}^{b}B_{12}^T(x,y)\mbf{v}_x(x)dx\right]^T\mbf{u}_y(y)dy	\\
			&\qquad\qquad\qquad+\ip{\mcl{P}[\hat{D}_1]\mbf{v}_x}{\mbf{u}_2}_{L_2} \\
			&\qquad=\ip{\mcl{P}\smallbmat{()&\hat{B}_{01}&()&()\\()&\hat{B}_{11}&()&()\\()&\hat{B}_{21}&()&()\\()&\hat{D}_{1}&()&()}\mbf{v}_x}{\smallbmat{u_0\\\mbf{u}_x\\\mbf{u}_y\\\mbf{u}_2}}_{\text{Z}_{2}^{\enn{u}}},
			\end{align*}
			and similarly
			\begin{align*}
			&\ip{\mbf{v}_y}{\mcl{P}\smallbmat{()&()&()&()\\()&()&()&()\\B_{20}&B_{21}&B_{22}&C_2\\()&()&()&()}\smallbmat{u_0\\\mbf{u}_x\\\mbf{u}_y\\\mbf{u}_2}}_{L_2}	\\
			&\hspace*{3.0cm}=\ip{\mcl{P}\smallbmat{()&()&\hat{B}_{02}&()\\()&()&\hat{B}_{12}&()\\()&()&\hat{B}_{22}&()\\()&()&\hat{D}_{2}&()}\mbf{v}_y}{\smallbmat{u_0\\\mbf{u}_x\\\mbf{u}_y\\\mbf{u}_2}}_{\text{Z}_{2}^{\enn{u}}},	\\
			&\ip{\mbf{v}_2}{\mcl{P}\smallbmat{()&()&()&()\\()&()&()&()\\()&()&()&()\\D_{0}&D_{1}&D_{2}&N}\smallbmat{u_0\\\mbf{u}_x\\\mbf{u}_y\\\mbf{u}_2}}_{L_2}	\\
			&\hspace*{3.0cm}=\ip{\mcl{P}\smallbmat{()&()&()&\hat{C}_{0}\\()&()&()&\hat{C}_{1}\\()&()&()&\hat{C}_{2}\\()&()&()&\hat{N}}\mbf{v}_2}{\smallbmat{u_0\\\mbf{u}_x\\\mbf{u}_y\\\mbf{u}_2}}_{\text{Z}_{2}^{\enn{u}}}. \\[-1.8em]
			\end{align*}
			Substituting the obtained expressions back into Eq.~\ref{eq:appx_ip_expanded}, we find \\[-1.4em]
			\begin{align*}
			&\ip{\mbf{v}}{\mcl{P}[G]\mbf{u}}_{Z^{\enn{v}}}=\ip{\smallbmat{v_0\\\mbf{v}_x\\\mbf{v}_y\\\mbf{v}_2}}{\mcl{P}\smallbmat{B_{00}&B_{01}&B_{02}&C_{0}\\B_{10}&B_{11}&B_{12}&C_{1}\\B_{20}&B_{21}&B_{22}&C_{2}\\D_{0}&D_{1}&D_{2}&N}\smallbmat{u_0\\\mbf{u}_x\\\mbf{u}_y\\\mbf{u}_2}}_{Z^{\enn{v}}}	\nonumber\\
			&=\ip{\mcl{P}\smallbmat{\hat{B}_{00}&()&()&()\\\hat{B}_{10}&()&()&()\\\hat{B}_{20}&()&()&()\\\hat{D}_{0}&()&()&()}v_0}{\smallbmat{u_0\\\mbf{u}_x\\\mbf{u}_y\\\mbf{u}_2}}_{\text{Z}_{2}^{\enn{u}}}	\\
			&\qquad+ \ip{\mcl{P}\smallbmat{()&\hat{B}_{01}&()&()\\()&\hat{B}_{11}&()&()\\()&\hat{B}_{21}&()&()\\()&\hat{D}_{1}&()&()}\mbf{v}_x}{\smallbmat{u_0\\\mbf{u}_x\\\mbf{u}_y\\\mbf{u}_2}}_{\text{Z}_{2}^{\enn{u}}} \\
			&\qquad\qquad+\ip{\mcl{P}\smallbmat{()&()&\hat{B}_{02}&()\\()&()&\hat{B}_{12}&()\\()&()&\hat{B}_{22}&()\\()&()&\hat{D}_{2}&()}\mbf{v}_y}{\smallbmat{u_0\\\mbf{u}_x\\\mbf{u}_y\\\mbf{u}_2}}_{\text{Z}_{2}^{\enn{u}}} \\
			&\qquad\qquad\qquad+\ip{\mcl{P}\smallbmat{()&()&()&\hat{C}_{0}\\()&()&()&\hat{C}_{1}\\()&()&()&\hat{C}_{2}\\()&()&()&\hat{N}}\mbf{v}_2}{\smallbmat{u_0\\\mbf{u}_x\\\mbf{u}_y\\\mbf{u}_2}}_{\text{Z}_{2}^{\enn{u}}}	\\
			&=\ip{\mcl{P}\smallbmat{\hat{B}_{00}&\hat{B}_{01}&\hat{B}_{02}&\hat{C}_{0}\\\hat{B}_{10}&\hat{B}_{11}&\hat{B}_{12}&\hat{C}_{1}\\\hat{B}_{20}&\hat{B}_{21}&\hat{B}_{22}&\hat{C}_{2}\\\hat{D}_{0}&\hat{D}_{1}&\hat{D}_{2}&\hat{N}}\smallbmat{v_0\\\mbf{v}_x\\\mbf{v}_y\\\mbf{v}_2}}{\smallbmat{u_0\\\mbf{u}_x\\\mbf{u}_y\\\mbf{u}_2}}_{Z^{\enn{v}}}
			\!\!=\ip{\mcl{P}[\hat{G}]\mbf{v}}{\mbf{u}}_{Z^{\enn{v}}}, \\[-1.8em]
			\end{align*}
			as desired.
		\end{proof}
		
		\clearpage
		
		\section{A Map from Fundamental to PDE State}\label{sec:appx_Tmap}

		For the following theorem, we recall the following definition of the domain of solutions of a PDE,
		\begin{align}\label{eq:appx_Xset1}
		X_{w} := \left\{
		\bar{\mbf{v}}=\bmat{\bar{\mbf{v}}_0\\\bar{\mbf{v}}_1\\\bar{\mbf{v}}_2}\! \in\! \bmat{L_2^{n_0}\\W_1^{n_1}\\W_2^{n_2}}
		\! \bbbbl\lvert\ 
		\bar{\mcl{E}}_0\bar{\mbf{v}}+\mcl{E}_1 w=0 
		\right\},
		\end{align}
		where
		\begin{align}\label{eq:appx_BC_Operators}
		\bar{\mcl{E}}_0&=\mcl{P}[E_0]\ \Lambda_{\text{bf}},
		&
		\mcl{E}_1&=\text{M}[E_{1}],
		\end{align}
		for matrix-valued function $E_1\in\mcl{N}_{011\leftarrow 2D}^{\enn{b}\times n_w}[\Omega_{ac}^{bd}]$ and parameters $E_0\in\mcl{N}_{011}^{\enn{b}\times\enn{f}}$,	
		with $\enn{b}:=\{n_1+4n_2,n_1+2n_2\}$ and $\enn{f}:=\{4n_1+16n_2,2n_1+4n_2\}$,
		and where the operator $\Lambda_{\text{bf}}:L_2^{n_0}\!\times\! W_1^{n_1}\!\times\! W_2^{n_2}\rightarrow\text{Z}_{1}^{\enn{f}}$ is as defined in~\eqref{eq:Lambda_bf}, extracting all the possible boundary values for the state components $\bar{\mbf{v}}_1$ and $\bar{\mbf{v}}_2$, and as limited by differentiability.
		
		\begin{figure*}[!t]
			\footnotesize
			\hrulefill
			\begin{flalign}\label{eq:appx_Tmat}
			&\begin{array}{l}
			T_{11}(x,y,\theta,\nu) = K_{33}^{11}(x,y,\theta,\nu)
			+ T_{21}(x,y,\theta,\nu) + T_{12}(x,y,\theta,\nu)
			-T_{22}(x,y,\theta,\nu), \\
			T_{21}(x,y,\theta,\nu) = - K_{32}^{1}(x,y,\nu)G_{23}^{0}(\theta,\nu) + T_{22}(x,y,\theta,\nu), \\
			T_{12}(x,y,\theta,\nu) = - K_{31}^{1}(x,y,\theta)G_{13}^{0}(\theta,\nu) + T_{22}(x,y,\theta,\nu), \\
			T_{22}(x,y,\theta,\nu) = -K_{30}(x,y)G_{03}(\theta,\nu)
			-\int_{a}^{x}K_{31}^{1}(x,y,\eta)G_{13}^{1}(\eta,\nu,\theta)d\eta
			-\int_{c}^{y}K_{32}^{1}(x,y,\mu)G_{23}^{1}(\theta,\mu,\nu)d\mu, 
			\end{array}	&
			T_{00}&=\bmat{I_{n_0}&0&0\\0&0&0\\0&0&0},
			\end{flalign}
			where
			\begin{align}\label{eq:appx_K_mat}
			&K_{33}^{11}(x,y,\theta,\nu)=
			\bmat{
				0&0&0\\
				0&I_{n_1}&0\\
				0&0&(x-\theta)(y-\nu)
			},  &
			&K_{32}^{1}(x,y,\nu)=
			\bmat{
				0&0&0\\
				I_{n_1}&0&0\\
				0&(y-\nu)&(x-a)(y-\nu)
			},\nonumber\\
			&K_{30}(x,y)=
			\bmat{
				0&0&0&0&0\\
				I_{n_1}&0&0&0&0\\
				0&I_{n_2}&(x-a)&(y-c)&(y-c)(x-a)
			},   &
			&K_{31}^{1}(x,y,\theta)=
			\bmat{
				0&0&0\\
				I_{n_1}&0&0\\
				0&(x-\theta)&(y-c)(x-\theta)
			},
			\hspace{2.0cm}
			\end{align}
			and
			\begin{align}\label{eq:appx_G_mat}
			&\enspace G_{0}(x,y)=\hat{R}_{00}F_{0}(x,y) + \hat{R}_{01}(x)F_{1}^{0}(x,y) + \int_{a}^{b}\hat{R}_{01}(\theta)F_{1}^{1}(\theta,y,x)d\theta
			+ \hat{R}_{02}(y)F_{2}^{0}(x,y) + \int_{c}^{d}\hat{R}_{02}(\nu)F_{2}^{1}(x,\nu,y),	&    \\
			&
			\begin{array}{l}
			G_{1}^{0}(x,y)=\hat{R}_{11}^{0}(x)F_{1}^{0}(x,y),
			\\
			G_{1}^{1}(x,y,\theta)=\hat{R}_{10}(x)F_{0}(\theta,y) + \hat{R}_{11}^{0}(x)F_{1}^{1}(x,y,\theta) 
			\\
			\hspace*{1.5cm}+ \hat{R}_{11}^{1}(x,\theta)F_{1}^{0}(\theta,y)  
			+\int_{a}^{b}\hat{R}_{11}^{1}(x,\eta)F_{1}^{1}(\eta,y,\theta)d\eta
			\\
			\hspace*{2.0cm}+\hat{R}_{12}(x,y)F_{2}^{0}(\theta,y) + \int_{c}^{d}\hat{R}_{12}(x,\nu)F_{2}^{1}(\theta,\nu,y)d\nu,
			\end{array}
			\hspace{0.5cm}
			\begin{array}{l}
			G_{2}^{0}(x,y)=\hat{R}_{22}^{0}(y)F_{2}^{0}(x,y), 
			\\
			G_{2}^{1}(x,y,\nu)=\hat{R}_{20}(y)F_{0}(x,\nu) + \hat{R}_{22}^{0}(y)F_{2}^{1}(x,y,\nu) 
			\\
			\hspace*{1.5cm}+ \hat{R}_{22}^{1}(y,\nu)F_{2}^{0}(x,\nu)  
			+\int_{c}^{d}\hat{R}_{22}^{1}(y,\mu)F_{2}^{1}(x,\mu,\nu)d\mu
			\\
			\hspace*{2.0cm}+\hat{R}_{21}(x,y)F_{1}^{0}(x,\nu) + \int_{a}^{b}\hat{R}_{21}(\theta,y)F_{1}^{1}(\theta,\nu,x)d\theta,
			\end{array}	\nonumber
			\hspace{1.1cm}
			\end{align}
			with
			\begin{align}\label{eq:appx_F_mat}
			&F_{0}(x,y)=E_{00}H_{03}(x,y) + E_{01}(x)H_{13}^{0}(x,y) + E_{02}(y)H_{23}^{0}(x,y),    \nonumber\\
			&F_{1}^{1}(x,y,\theta)=E_{10}(x)H_{03}(\theta,y) + E_{11}^{1}(x,\theta)H_{13}^{0}(\theta,y)
			+E_{12}(x,y)H_{23}^{0}(\theta,y),   &
			&F_{1}^{0}(x,y)=E_{11}^{0}(x)H_{13}^{0}(x,y), \nonumber\\
			&F_{2}^{1}(x,y,\nu)=E_{20}(y)H_{03}(x,\nu) + E_{22}^{1}(y,\nu)H_{23}^{0}(x,\nu)
			+E_{21}(x,y)H_{13}^{0}(x,\nu),   &
			&F_{2}^{0}(x,y)=E_{22}^{0}(y)H_{23}^{0}(x,y), \hspace{3.0cm}
			\end{align}
			and $\smallbmat{\hat{R}_{00}&\hat{R}_{01}&\hat{R}_{02}\\\hat{R}_{10}&\hat{R}_{11}&\hat{R}_{12}\\\hat{R}_{20}&\hat{R}_{21}&\hat{R}_{22}}=\mcl{L}_{\text{inv}}\left(\smallbmat{R_{00}&R_{01}&R_{02}\\R_{10}&R_{11}&R_{12}\\R_{20}&R_{21}&R_{22}}\right)\in\mcl{N}_{011}$ are the parameters associated to the inverse of the 011-PI operator $\mcl{P}\smallbmat{R_{00}&R_{01}&R_{02}\\R_{10}&R_{11}&R_{12}\\R_{20}&R_{21}&R_{22}}$ (see also Appendix II-A of~\cite{jagt2021PIEArxiv}), with $R_{11}=\{R_{11}^{0},R_{11}^{1},R_{11}^{1}\}\in\mcl{N}_{1D}$ and $R_{22}=\{R_{22}^{0},R_{22}^{1},R_{22}^{1}\}\in\mcl{N}_{1D}$, and with 
			\begin{align}\label{eq:appx_R_mat}
			&R_{00}=E_{00}H_{00} + \int_{a}^{b}E_{01}(x)H_{10}(x)dx + \int_{c}^{d}E_{02}(y)H_{20}(y)dy, \nonumber\\
			&R_{01}(x)=E_{00}H_{01}(x) + E_{01}(x)H_{11}^{0}(x),  &
			&R_{02}(y)=E_{00}H_{02}(y) + E_{02}(y)H_{22}^{0}(y),  \nonumber\\
			&R_{10}(x)=E_{10}(x)H_{00}, &
			&R_{20}(y)=E_{20}(y)H_{00}, \nonumber\\
			&R_{11}^{0}(x)=E_{11}^{0}H_{11}^{0},   &
			&R_{22}^{0}(y)=E_{22}^{0}H_{22}^{0},   \nonumber\\
			&R_{11}^{1}(x,\theta)= E_{10}(x)H_{01}(\theta) + E_{11}^{1}(x,\theta)H_{11}^{0}, &
			&R_{22}^{1}(y,\nu)= E_{20}(y)H_{02}(\nu) + E_{22}^{1}(y,\nu)H_{22}^{0}, \nonumber\\
			&R_{12}(x,y)=E_{10}(x)H_{02}(y) + E_{12}(x,y)H_{22}^{0}(y), &
			&R_{21}(x,y)=E_{20}(y)H_{01}(x) + E_{21}(x,y)H_{11}^{0}(x), 
			\hspace{2.3cm}
			\end{align}
			where,
			{\tiny
				\begin{align}\label{eq:appx_H_mat}
				&H_{00} =
				\bmat{I_{n_1}&0&0&0&0\\
					I_{n_1}&0&0&0&0\\
					I_{n_1}&0&0&0&0\\
					I_{n_1}&0&0&0&0\\
					0&I_{n_2}&0&0&0\\
					0&I_{n_2}&(b-a)&0&0\\
					0&I_{n_2}&0&(d-c)&0\\
					0&I_{n_2}&(b-a)&(d-c)&(d-c)(b-a)\\
					0&0&I_{n_2}&0&0\\
					0&0&I_{n_2}&0&0\\
					0&0&I_{n_2}&0&(d-c)\\
					0&0&I_{n_2}&0&(d-c)\\
					0&0&0&I_{n_2}&0\\
					0&0&0&I_{n_2}&(b-a)\\
					0&0&0&I_{n_2}&0\\
					0&0&0&I_{n_2}&(b-a)\\
					0&0&0&0&I_{n_2}\\
					0&0&0&0&I_{n_2}\\
					0&0&0&0&I_{n_2}\\
					0&0&0&0&I_{n_2}},    &
				&H_{01}(x) =
				\bmat{
					0&0&0\\I_{n_1}&0&0\\0&0&0\\I_{n_1}&0&0\\
					0&0&0\\ 0&(b-x)&0\\ 0&0&0\\
					0&(b-x)&(d-c)(b-x)\\
					0&0&0\\ 0&I_{n_2}&0 \\0&0&0\\ 0&I_{n_2}&(d-c)\\
					0&0&0\\ 0&0&(b-x) \\0&0&0\\ 0&0&(b-x)\\
					0&0&0\\ 0&0&I_{n_2} \\0&0&0\\ 0&0&I_{n_2}
				},   &
				&H_{02}(y) =
				\bmat{
					0&0&0\\0&0&0\\I_{n_1}&0&0\\I_{n_1}&0&0\\
					0&0&0\\ 0&0&0\\ 0&(d-y)&0\\
					0&(d-y)&(b-a)(d-y)\\
					0&0&0\\ 0&0&0 \\0&0&(d-y)\\ 0&0&(d-y)\\
					0&0&0\\ 0&0&0 \\0&I_{n_2}&0\\ 0&I_{n_2}&(b-a)\\
					0&0&0\\ 0&0&0 \\0&0&I_{n_2}\\ 0&0&I_{n_2}
				},   \nonumber\\
				&\begin{array}{l}
				H_{11}^{0}=
				\bmat{
					I_{n_1}&0&0\\I_{n_1}&0&0\\
					0&I_{n_2}&0\\0&I_{n_2}&(d-c)\\0&0&I_{n_2}\\0&0&I_{n_2}
				},\\
				\\
				H_{22}^{0}=
				\bmat{
					I_{n_1}&0&0\\I_{n_1}&0&0\\
					0&I_{n_2}&0\\0&I_{n_2}&(b-a)\\0&0&I_{n_2}\\0&0&I_{n_2}
				}
				\end{array}
				&
				&\begin{array}{l}
				H_{13}^{0}(y)=
				\bmat{
					0&0&0\\0&I_{n_1}&0\\
					0&0&0\\ 0&0&(d-y)\\0&0&0\\0&0&I_{n_2}
				},\\
				\\
				H_{23}^{0}(x)=
				\bmat{
					0&0&0\\0&I_{n_1}&0\\
					0&0&0\\ 0&0&(b-x)\\0&0&0\\0&0&I_{n_2}
				},
				\end{array}
				&
				&H_{03}(x,y)=
				\bmat{
					0&0&0\\0&0&0\\0&0&0\\0&I_{n_1}&0\\
					0&0&0\\0&0&0\\0&0&0\\ 0&0&(d-y)(b-x)\\0&0&0\\0&0&0\\0&0&0\\ 0&0&(d-y)\\0&0&0\\0&0&0\\0&0&0\\ 0&0&(b-x)\\
					0&0&0\\0&0&0\\0&0&0\\0&0&I_{n_2}
				}.
				\end{align}
			}
			\hrulefill
			\caption{Parameters $T$ describing PI operator $\mcl{T}=\mcl{P}[T]$ mapping the fundamental state back to the PDE state in Theorem~\ref{thm:appx_Tmap}}
			\label{fig:appx_Tmap_matrices}
		\end{figure*}
		
		\begin{figure*}[!t]
			\footnotesize
			\hrulefill
			\begin{align}\label{eq:appx_Tinmat}
			T_{1}(x,y)&=-K_{30}(x,y)Q_{0} - \int_{a}^{x}K_{31}^{1}(x,y,\theta)Q_{1}(\theta)d\theta - \int_{c}^{y}K_{32}^{1}(x,y,\nu)Q_{2}(\nu)d\nu,	 
			\hspace{3.5cm}
			\end{align}
			where the parameters $K_{3j}$ are as defined in~\eqref{eq:appx_K_mat}, and
			\begin{flalign}\label{eq:appx_Qin_mat}
			&\enspace Q_{0}=\hat{R}_{00}E_{0} + \int_{a}^{b}\hat{R}_{01}(\theta)E_{1}(\theta)d\theta +\int_{c}^{d}\hat{R}_{02}(\nu)E_{2}(\nu)d\nu,    \nonumber\\
			&\enspace Q_{1}(x)=\hat{R}_{10}(x)E_{0} + \hat{R}_{11}^{0}(x)E_{1}(x) + \int_{a}^{x}\hat{R}_{11}^{1}(x,\theta)E_{1}(\theta)d\theta + \int_{x}^{b}\hat{R}_{11}^{2}(x,\theta)E_{1}(\theta)d\theta + \int_{c}^{d}\hat{R}_{12}(x,\nu)E_{2}(\nu)d\nu,    \nonumber\\
			&\enspace Q_{2}(y)=\hat{R}_{20}(y)E_{0} + \int_{a}^{b}\hat{R}_{21}(y,\theta)E_{1}(\theta)d\theta + \hat{R}_{22}^{0}(y)E_{2}(y) + \int_{c}^{y}\hat{R}_{22}^{1}(y,\nu)E_{2}(\nu)d\nu + \int_{y}^{d}\hat{R}_{22}^{2}(y,\nu)E_{2}(\nu)d\nu,	&
			\end{flalign}
			where $\smallbmat{\hat{R}_{00}&\hat{R}_{01}&\hat{R}_{02}\\\hat{R}_{10}&\hat{R}_{11}&\hat{R}_{12}\\\hat{R}_{20}&\hat{R}_{21}&\hat{R}_{22}}=\mcl{L}_{\text{inv}}\left(\smallbmat{R_{00}&R_{01}&R_{02}\\R_{10}&R_{11}&R_{12}\\R_{20}&R_{21}&R_{22}}\right)\in\mcl{N}_{011}$ are the parameters associated to the inverse of the 011-PI operator $\mcl{P}\smallbmat{R_{00}&R_{01}&R_{02}\\R_{10}&R_{11}&R_{12}\\R_{20}&R_{21}&R_{22}}$ defined in Eqn.~\eqref{eq:appx_R_mat}.
			
			\hrulefill
			\caption{Parameters $T_1$ describing the PI operator $\mcl{T}_1=\text{M}[T_1]\in\Pi_{0112}^{\{0,0,n_v\}\times\{n_w,0,0\}}$ mapping the fundamental state back to the PDE state in Theorem~\ref{thm:appx_Tmap}}
			\label{fig:appx_Tinmap_matrices}
		\end{figure*}

		\begin{thm}\label{thm:appx_Tmap}
			For $\enn{b}:=\{n_1+4n_2,n_1+2n_2\}$ and $\enn{f}:=\{4n_1+16n_2,2n_1+4n_2\}$, let
			$E_0=\smallbmat{E_{00}&E_{01}&E_{02}\\E_{10}&E_{11}&E_{12}\\E_{20}&E_{21}&E_{22}}\in\mcl{N}_{011}^{\enn{b}\times\enn{f}}$ and $E_1=\smallbmat{E_{1,0}\\E_{1,1}\\E_{1,2}}\in\mcl{N}_{011\leftarrow 2D}$ 
			with
			\begin{align*}
			E_{11}&:=\{E_{11}^{0},E_{11}^{1},E_{11}^{1}\}\in\mcl{N}_{1D}^{n_1+2n_2\times 2n_1+4n_2},	\\
			E_{22}&:=\{E_{22}^{0},E_{22}^{1},E_{22}^{1}\}\in\mcl{N}_{1D}^{n_1+2n_2\times 2n_1+4n_2},
			\end{align*}
			be given, and $\mcl{E}_0\mcl{H}_1$ be invertible, where $\mcl{E}_0:=\mcl{P}[E_0]\in\Pi_{011}^{\enn{b}\times\enn{f}}$ and $\mcl{H}_1\in\Pi_{011}^{\enn{f}\times\enn{b}}$ is as in Cor.~\ref{cor:vhat_to_BC}. Let $w$ be a given input signal, with associated set $X_{w}$ as defined in Eqn.~\eqref{eq:appx_Xset1}.
			Let $T_1\in\text{Z}_{2}^{\{0,0,n_v\}\times\{n_w,0,0\}}[\Omega_{ac}^{bd}]$ be as defined in~\eqref{eq:appx_Tinmat} in Figure~\ref{fig:appx_Tinmap_matrices}, and let
			\begin{align*}
			T_0&=\smallbmat{T_{00}&0&0\\0&T_{11}&T_{12}\\0&T_{21}&T_{22}}\in\mcl{N}_{2D}^{n_v\times n_v}
			\end{align*}
			where $n_v=n_0+n_1+n_2$, and where the parameters $T_{ij}$ are as defined in Eqn.~\eqref{eq:appx_Tmat} in Fig~\ref{fig:appx_Tmap_matrices}.
			Then, if $\mcl{T}_0=\mcl{P}[T_0]\in\Pi_{2D}^{n_v\times n_v}$ and $\mcl{T}_1=\text{M}[T_1]\in\Pi_{2D\leftarrow 011}^{n_v\times \{n_w,0,0\}}$, then for any $\bar{\mbf{v}}\in X_{w}$ and $\mbf{v} \in L^{n_v}_2$,
			\begin{align*}
			\bar{\mbf{v}} &= \mcl{T}_0\mscr{D}\bar{\mbf{v}} + \mcl{T}_1 w
			&	&\text{and}	&
			\mbf{v} &= \mscr{D}\bbl[\mcl{T}_0\mbf{v}+\mcl{T}_1 w\bbr],
			\end{align*}
			where $\mscr{D}=\smallbmat{I_{n_0}&&\\&\partial_x\partial_y&\\&&\partial_x^2\partial_y^2}$.\\
			
		\end{thm}
		
		\begin{proof}
			We first proof the first identity, $\bar{\mbf{v}}=\mcl{T}_0\mscr{D}\bar{\mbf{v}}+\mcl{T}_1 w$. To this end, suppose $\bar{\mbf{v}}\in X_{w}$, and define $\mbf{v}=\mscr{D}\bar{\mbf{v}}\in L_2^{n_v}[\Omega_{x,y}]$. Furthermore,
			let $K_{ij}$ and $H_{ij}$ (for appropriate $i,j\in\{0,1,2,3\}$) be as defined in Equations~\eqref{eq:appx_K_mat} and~\eqref{eq:appx_H_mat}, and let
			\begin{align*}
			H_1&=\bmat{
				H_{00}~H_{01}~H_{02}\\
				0\ ~~H_{11}\ ~~0\\
				~~0\ ~~\ 0 ~~~~H_{22}}\in\mcl{N}_{011} &
			&H_2=\bmat{H_{03}\\H_{13}\\H_{23}}\in\mcl{N}_{2D\rightarrow 011}  \\
			K_1&=\bmat{K_{30}\\K_{31}\\K_{32}}\in\mcl{N}_{011\rightarrow 2D}    &
			&\hspace{-0.5cm}K_2=\bmat{T_{00}~~0~~~0\\~0\ ~K_{22}^{11} ~0\\~0\ ~~~0 ~ ~~0}\in\mcl{N}_{2D},
			\end{align*}
			where
			\begin{align*}
			&H_{11}\! =\! \{H_{11}^{0},0,0\}\! \in\mcl{N}_{1D}	&	 &\! \! H_{22}\! =\! \{H_{22}^{0},0,0\}\! \! \in\mcl{N}_{1D}	\\
			&H_{13}\! =\! \{H_{13}^{0},0,0\}\! \in\mcl{N}_{2D\rightarrow 1D} 	&
			&\! \! H_{23}\! =\! \{H_{23}^{0},0,0\}\! \in\mcl{N}_{2D\rightarrow 1D}  \\
			&K_{31}\! =\! \{0,K_{31}^{1},0\}\! \in\mcl{N}_{1D\rightarrow 2D} 	&
			&\! \! K_{32}\! =\! \{0,K_{32}^{1},0\}\! \in\mcl{N}_{1D\rightarrow 2D} .
			\end{align*}
			Then, by Lemma~\ref{lem:vhat_to_v} and Corollary~\ref{cor:vhat_to_BC},
			\begin{align}\label{eq:vhat_to_v_appendix}
			\Lambda_{\text{bf}} \bar{\mbf{v}} &= \mcl{H}_1 \Lambda_{\text{bc}} \bar{\mbf{v}}+\mcl{H}_2 \mbf{v} \nonumber\\
			\bar{\mbf{v}} &= \mcl{K}_1 \Lambda_{\text{bc}} \bar{\mbf{v}}+\mcl{K}_2 \mbf{v},
			\end{align}
			where $\mcl{H}_1=\mcl{P}[H_1]$, $\mcl{H}_2=\mcl{P}[H_2]$, $\mcl{K}_1=\mcl{P}[K_1]$, and $\mcl{K}_2=\mcl{P}[K_2]$. Substituting this first result into the expression for the boundary conditions, $0=\mcl{E}_0\Lambda_{\text{bf}}\bar{\mbf{v}}+\mcl{E}_1 w$, we may use the composition rules of PI operators to express
			\begin{align*}
			0&=\mcl{E}_0\Lambda_{\text{bf}}\bar{\mbf{v}} + \mcl{E}_1 w	\\
			&=\mcl{E}_0\mcl{H}_1 \Lambda_{\text{bc}} \bar{\mbf{v}}+\mcl{E}_0\mcl{H}_2 \mbf{v} + \mcl{E}_1 w	\\
			&=\mcl{R} \Lambda_{\text{bc}} \bar{\mbf{v}}+\mcl{F} \mbf{v} + \mcl{E}_1 w
			\end{align*}
			where $\mcl{R}=\mcl{P}[R]$ and $\mcl{F}=\mcl{P}[F]$ with
			\begin{align*}
			R&=\bmat{R_{00}&R_{01}&R_{02}\\R_{10}&R_{11}&R_{12}\\R_{20}&R_{21}&E_{22}}\in\mcl{N}_{011},    &
			F&=\bmat{F_{0}\\F_{1}\\F_{2}}\in\mcl{N}_{2D\rightarrow 011},
			\end{align*}
			and
			\begin{align*}
			R_{11}&=\{R_{11}^{0},R_{11}^{1},R_{11}^{1}\}\in\mcl{N}_{1D}
			\\
			R_{22}&=\{R_{22}^{0},R_{22}^{1},R_{22}^{1}\}\in\mcl{N}_{1D} \\
			F_{1}&=\{F_{1}^{0},F_{1}^{1},F_{1}^{1}\}\in\mcl{N}_{2D\rightarrow 1D} \\
			F_{2}&=\{F_{2}^{0},F_{2}^{1},F_{2}^{1}\}\in\mcl{N}_{2D\rightarrow 1D}
			\end{align*}
			defined as in Equations~\eqref{eq:appx_F_mat} and~\eqref{eq:appx_R_mat}. Since, by the theorem statement, the operator $\mcl{R}:=\mcl{E}_0\mcl{H}_1$ is invertible, we may write
			\begin{align*}
			\Lambda_{\text{bc}}\bar{\mbf{v}}&=-\mcl{R}^{-1}\bbl[\mcl{F}\mbf{v}+ \mcl{E}_1w\bbr]	\\
			&=-\mcl{G}\mbf{v} - \mcl{Q}w,
			\end{align*}
			where $\mcl{G}=\mcl{P}[G]$ and $\mcl{Q}=\text{M}[Q]$ with
			\begin{align*}
			G&\!=\!\bmat{G_{0}\\\{G_{1}^{0},G_{1}^{1},G_{1}^{1}\}\\\{G_{2}^{0},G_{2}^{1},G_{2}^{1}\}}\!\in\!\mcl{N}_{2D\rightarrow 011},	&
			Q_j&\!=\!\bmat{Q_{0}\\Q_{1}\\Q_{2}}\!\in\!\bmat{\R\\L_2[x]\\L_2[y]}
			\end{align*}
			defined as in Eqns.~\eqref{eq:appx_G_mat} and~\eqref{eq:appx_Qin_mat}. Finally, substituting this expression into Equation~\eqref{eq:vhat_to_v_appendix}, and once more using the composition rules of PI operators, we obtain
			\begin{align*}
			\bar{\mbf{v}} &= \mcl{K}_1 \Lambda_{\text{bc}} \bar{\mbf{v}}+\mcl{K}_2 \mbf{v}	\\
			&\quad=-\mcl{K}_1 \bbl[\mcl{G}\mbf{v} + \mcl{Q}w\bbr] +\mcl{K}_2 \mbf{v}\\
			&\qquad=(\mcl{K}_2-\mcl{K}_1\mcl{G})\mbf{v} - \mcl{K}_1\mcl{Q}w	\\
			&\quad\qquad=\mcl{T}_0\mbf{v} + \mcl{T}_0 w
			=\mcl{T}_0\mscr{D}\bar{\mbf{v}} + \mcl{T}_1 w,
			\end{align*}
			as desired.
			
			We now prove the second identity, $\mbf{v}=\mscr{D}\bbl[\mcl{T}_0\mbf{v}+\mcl{T}_1 w\bbr]$. To this end, let $\mbf{v}\in L_2^{n_0+n_1+n_2}[x,y]$ be arbitrary, and define $\bar{\mbf{v}}=\smallbmat{\bar{\mbf{v}}_0\\\bar{\mbf{v}}_1\\\bar{\mbf{v}}_2}=\mcl{T}_0\mbf{v}+\mcl{T}_1 w$. We note that, in the proof of Thm.~32 in Appendix~II-C of~\cite{jagt2021PIEArxiv}, we already showed that the composition of the differential operator $\mscr{D}$ with the PI operator $\mcl{T}$ is an identiy operation:
			\begin{align*}
			\mscr{D}\mcl{T}_0=\text{M}[I_{n_v}].
			\end{align*}
			It suffices therefore to show that the composition of differential operator $\mscr{D}$ with PI operator $\mcl{T}$ is a null operator:
			\begin{align*}
			\mscr{D}\mcl{T}_1=\text{M}[0_{n_v}].
			\end{align*}
			To prove this, we observe that $\mcl{T}_1$ is in fact merely a multiplier operator, $\mcl{T}_1=\mcl{P}[T_1]=\text{M}[T_1]$, mapping a finite dimensional signal $w$ to an infinite-dimensional state $\bar{\mbf{v}}$. We can decompose this state $\bar{\mbf{v}}$ into its different components $\bar{\mbf{v}}_0$, $\bar{\mbf{v}}_1$ and $\bar{\mbf{v}}_2$, using the matrices
			\begin{align*}
			J_0 &= \bmat{I_{n_0}&0&0},	&	J_1&=\bmat{I_{n_1}&0},	&	J_2&=\bmat{0&I_{n_2}},
			\end{align*}
			and
			\begin{align*}
			S=\bmat{0&I_{n_1}&0\\0&0&I_{n_2}}\in\R^{n_1+n_2\times n_0+n_1+n_2},
			\end{align*}
			so that
			\begin{align*}
			\bar{\mbf{v}}_0&=J_0\mbf{v},	&
			\bar{\mbf{v}}_1&=J_1S\mbf{v},	&	\bar{\mbf{v}}_2&=J_2S\mbf{v}.
			\end{align*}
			By definition of the operator $\mscr{D}$, we then have to prove that
			\begin{align}
			0&\!=\!J_0\mcl{T}_1\mbf{v}=\text{M}[J_0 T_1]\mbf{v}, \label{eq:appx_uhat0=Du0}\\
			0&\!=\!J_1 S\ \partial_x\partial_y\bbl(\mcl{T}_1\mbf{v}\bbr)=\text{M}\bbl[\partial_x\partial_y \bbl(J_1 S T_1\bbr)\bbr]\mbf{v},	\label{eq:appx_uhat1=Du1}\\
			0&\!=\!J_2 S\ \partial_x^2\partial_y^2\bbl(\mcl{T}_1\mbf{v}\bbr)=\text{M}\bbl[\partial_x^2\partial_y^2 \bbl(J_2 S T_1\bbr)\bbr]\mbf{v}. \label{eq:appx_uhat2=Du2}	
			\end{align} 
			Here we recall the definition of the function $T_1\in L_2^{n_v\times n_w}[\Omega_{ac}^{bd}]$ from~\eqref{eq:appx_Tmat} in Fig.~\ref{fig:appx_Tinmap_matrices}:
			\begin{align*}
			T_{1}(x,y)&=-K_{30}(x,y)Q_{0} - \int_{a}^{x}K_{31}^{1}(x,y,\theta)Q_{1}(\theta)d\theta\\
			&- \int_{c}^{y}K_{32}^{1}(x,y,\nu)Q_{2}(\nu)d\nu,
			\end{align*}
			where the functions $K_{3j}$ are as defined in~\eqref{eq:appx_K_mat} in Fig.~\ref{fig:appx_Tmap_matrices} and the functions $Q_{j}$ are as defined in~\eqref{eq:appx_Qin_mat} in Fig.~\ref{fig:appx_Tinmap_matrices}. Studying the definitions of the parameters $K$, we observe that the first $n_0$ rows of each are zero, and therefore $J_0 K_{30}=0$ and $J_0K_{3j}^{1}=0$ for each $j\in\{1,2\}$. It follows that also $J_0 T_1\equiv 0$, proving relation~\eqref{eq:appx_uhat0=Du0}.
			
			To prove Relation~\eqref{eq:appx_uhat1=Du1}, we use the Leibniz integral rule to compute the derivative $\partial_x [ST_1(x,y)]$:		
			\begin{align*}
			&\partial_x [ST_1(x,y)]=-\partial_x\bl[SK_{30}\br](x,y)G_0 	\\
			&\quad- SK_{31}^{1}(x,y,x)G_1(x) - \int_{a}^{x}\partial_x \bl[SK_{31}^{1}\br](x,y,\theta)G_1(\theta)d\theta	\\
			&\qquad - \int_{c}^{y}\partial_x\bl[SK_{32}^{1}\br](x,y,\nu)G_2(\nu)d\nu
			\end{align*}
			Differentiating with respect to $y$ is well, we find
			\begin{align}\label{eq:appx_dST1}
			&\partial_x\partial_y [ST_1(x,y)]=-\partial_x\partial_y\bl[SK_{30}\br](x,y)G_0 	\nonumber\\
			&\quad- \partial_y\bl[SK_{31}^{1}\br](x,y,x)G_1(x) 
			-\partial_x\bl[SK_{32}^{1}\br](x,y,y)G_2(y)	\nonumber\\
			&\qquad- \int_{a}^{x}\partial_x\partial_y \bl[SK_{31}^{1}\br](x,y,\theta)G_1(\theta)d\theta	\nonumber\\ &\quad\qquad-\int_{c}^{y}\partial_x\partial_y\bl[SK_{32}^{1}\br](x,y,\nu)G_2(\nu).
			\end{align}
			Taking the derivatives of the different function $K$, we find
			\begin{align}\label{eq:appx_dST1_params}
			\partial_x\partial_y [SK_{30}](x,y)\!&=\!\bmat{0&\!\!0&\!\!0&\!\!0&\!\!0\\0&\!\!0&\!\!0&\!\!0&\!\! I_{n_2}}\!\!\in\! \R^{n_1+n_2\times n_1+4n_2},	\nonumber\\
			\partial_y [SK_{31}^{1}](x,y,x)&=0\in \R^{n_1+n_2\times n_1+2n_2}	\nonumber\\
			\partial_x [SK_{32}^{1}](x,y,y)&=0\in \R^{n_1+n_2\times n_1+2n_2},	\nonumber\\
			\partial_x\partial_y \bl[SK_{31}^{1}\br](x,y,\theta)		 &=\bmat{0&0&0\\0&0&I_{n_2}}\in \R^{n_1+n_2\times n_1+2n_2},	\nonumber\\
			\partial_x\partial_y \bl[SK_{32}^{1}\br](x,y,\nu)	 &=\bmat{0&0&0\\0&0&I_{n_2}}\in \R^{n_1+n_2\times n_1+2n_2}.
			\end{align}
			The first $n_1$ rows of each of these matrices is identically equal to zero, from which it follows that $\partial_x\partial_y [J_1 ST_1(x,y)]\equiv0$, satisfying Relation~\eqref{eq:appx_uhat1=Du1}.
			
			Finally, to prove Relation~\eqref{eq:appx_uhat2=Du2}, we substitute the values in~\eqref{eq:appx_dST1_params} into the Equation~\eqref{eq:appx_dST1}, and premultiply with $J_2$ to find
			\begin{align}\label{eq:appx_JdST1}
			&\partial_x\partial_y[J_2ST_1](x,y)=-\partial_x\partial_y\bl[J_2 SK_{30}\br](x,y)G_0 	\nonumber\\
			&\qquad- \int_{a}^{x}\partial_x\partial_y \bl[J_2 SK_{31}^{1}\br](x,y,\theta)G_1(\theta)d\theta	\nonumber\\ &\quad\qquad-\int_{c}^{y}\partial_x\partial_y\bl[J_2 SK_{32}^{1}\br](x,y,\nu)G_2(\nu),
			\end{align}
			where		
			\begin{align}\label{eq:appx_JdST1_params}
			\partial_x\partial_y [J_2SK_{30}](x,y)\!&=\!\bmat{0&\!\!0&\!\!0&\!\!0&\!\! I_{n_2}}\!\!\in\! \R^{n_2\times n_1+4n_2},	\nonumber\\
			\partial_x\partial_y \bl[J_2 SK_{31}^{1}\br](x,y,\theta)		 &\!=\!\bmat{0&\!\!0&\!\!I_{n_2}}\!\in\! \R^{n_2\times n_1+2n_2},	\nonumber\\
			\partial_x\partial_y \bl[J_2 SK_{32}^{1}\br](x,y,\nu)		 &\!=\!\bmat{0&\!\!0&\!\!I_{n_2}}\!\in\! \R^{n_2\times n_1+2n_2}.
			\end{align}
			Using the Leibniz integral rule, we can differentiate Expression~\eqref{eq:appx_JdST1} with respect to $x$ and $y$ again, obtaining
			\begin{align}\label{eq:appx_dST1}
			&\partial_x^2\partial_y^2 [J_2ST_1](x,y)=-\partial_x^2\partial_y^2\bl[J_2 SK_{30}\br](x,y)G_0 	\nonumber\\
			&\qquad- \partial_x\partial_y^2\bl[J_2SK_{31}^{1}\br](x,y,x)G_1(x) \nonumber\\
			&\quad\qquad-\partial_x^2\partial_y\bl[J_2SK_{32}^{1}\br](x,y,y)G_2(y)	\nonumber\\
			&\qquad\qquad- \int_{a}^{x}\partial_x^2\partial_y^2 \bl[J_2 SK_{31}^{1}\br](x,y,\theta)G_1(\theta)d\theta	\nonumber\\ &\quad\qquad\qquad-\int_{c}^{y}\partial_x^2\partial_y^2\bl[J_2SK_{32}^{1}\br](x,y,\nu)G_2(\nu).
			\end{align}
			Each of the terms in this expression is a derivative of one of the functions in~\eqref{eq:appx_JdST1_params}. However, since each of these functions is homogeneous, the derivative of these functions will be identically equal to zero. It follows that $\partial_x^2\partial_y^2 [J_2ST_1](x,y)\equiv0$, satisfying the final Relation~\eqref{eq:appx_uhat2=Du2}, and thereby concluding the proof.
			
		\end{proof}

		\clearpage
		
		\section{A Parameterization of Positive PI Operators}\label{sec:appx_pos_PI}
		
		\begin{prop}\label{prop:appx_pos_PI}
			For any $Z\in L_2^{q\times n_2}[\Omega_{ac}^{bd}\times\Omega_{ac}^{bd}]$ and scalar function $g\in L_2[\Omega_{ac}^{bd}]$ satisfying $g(x,y)\geq 0$ for any $(x,y)\in\Omega_{ac}^{bd}$, let $\mcl{L}_{\text{PI}}:\R^{(9q+n_0)\times (9q+n_0)}\rightarrow\mcl{N}_{0112}^{\enn{u}\times \enn{u}}$ be defined as 
			\begin{align}\label{eq:appx_posmat_to_posPI_appendix}
			&\mcl{L}_{\text{PI}}\left(\smallbmat{P_{00}&\hdots&P_{09}\\\vdots&\ddots&\vdots\\P_{90}&\hdots&P_{99}}\right)    
			=B:=\bmat{P_{00}&B_{02}\\B_{20}&N}\in\mcl{N}_{0112}\smallbmat{n_0&n_0\\0&0\\n_2&n_2},
			\end{align}
			where 
			{\small
				\begin{align*}
				&B_{02}(x,y)=g(x,y)P_{01} Z_1(x,y) dy dx \\ 
				&\ +\! \int_{x}^{b}\! \sqrt{g}(\theta,y) P_{02}Z_2(\theta,y,x) d\theta +\! \int_{a}^{x}\! \sqrt{g}(\theta,y) P_{03}Z_3(\theta,y,x) d\theta \\ 
				&\ +\! \int_{y}^{d}\! \sqrt{g}(x,\nu)P_{04}Z_4(x,\nu,y)d\nu  +\! \int_{c}^{y}\! \sqrt{g}(x,\nu)P_{05}Z_5(x,\nu,y)d\nu  \\
				&\ +\int_{x}^{b}\int_{y}^{d}\sqrt{g}(\theta,\nu)P_{06}Z_6(\theta,\nu,x,y) d\nu d\theta \\
				&\qquad+\int_{a}^{x}\int_{y}^{d}\sqrt{g}(\theta,\nu)P_{07}Z_7(\theta,\nu,x,y) d\nu d\theta \\
				&\qquad\qquad+\int_{x}^{b}\int_{c}^{y}\sqrt{g}(\theta,\nu)P_{08}Z_8(\theta,\nu,x,y) d\nu d\theta \\
				&\qquad\qquad\qquad+\int_{a}^{x}\int_{c}^{y}\sqrt{g}(\theta,\nu)P_{09}Z_9(\theta,\nu,x,y) d\nu d\theta,	\\
				&B_{20}(x,y)=B_{02}^T(x,y),
				\end{align*}
			}
			and where the parameters $N:=\smallbmat{N_{00}&N_{01}&N_{02}\\N_{10}&N_{11}&N_{12}\\N_{20}&N_{21}&N_{22}}\in\mcl{N}_{2D}^{n_2\times n_2}$ are as defined in Equations~\eqref{eq:appx_pos_Nmats} in Figure~\ref{fig:appx_positive_parameters}.
			in Eqn.~() in~\cite{}, where $\enn{u}:=\{n_0,0,n_2\}$. Then, for any $P\geq 0$, if $B=\mcl{L}_{\text{PI}}(P)$, then 
			$\mcl{P}:=\mcl{P}[B]\in\Pi_{0112}^{\enn{u}\times\enn{u}}$ satisfies $\mcl{P}^*=\mcl{P}$ and $\ip{\mbf{u}}{\mcl{P}\mbf{u}}_{\text{Z}_{2}^{\enn{u}}}\geq 0$ for any $\mbf{u}\in \text{Z}_{2}^{\enn{u}}$.
		\end{prop}
		
		\begin{proof}
			Let $P\geq 0$ be an arbitrary matrix of appropriate size, and $B=\mcl{L}_{\text{PI}}(P)\in\mcl{N}_{0112}\smallbmat{n_0&n_0\\0&0\\n_2&n_2}$.
			It is easy to see that, by definition of the function $B_{20}$ and the functions $N_{ij}\in L_2$, the PI operator $\mcl{P}[B]$ defined by $B$ is self-adjoint. Furthermore, defining $\mcl{Z}:=\bmat{\mcl{Z}_0&\\&\mcl{Z}_2}\in\Pi_{0112}^{\{0,0,n_0+n_2\}\times\{n_0,0,n_2\}}$, where
			\begin{align}\label{eq_Zop_appendix}
			(\mcl{Z}_0u_0)(x,y)&=\sqrt{g(x,y)}u_0	\\
			(\mcl{Z}_2\mbf{u}_2)(x,y)
			&\!=\!\bmat{
				\sqrt{g(x,y)}Z_1(x,y)\mbf{u}_2(x,y)\\
				\int_{a}^{x}\sqrt{g(x,y)}Z_2(x,y,\theta)\mbf{u}_2(\theta,y)d\theta    \\
				\int_{x}^{b}\sqrt{g(x,y)}Z_3(x,y,\theta)\mbf{u}_2(\theta,y)d\theta    \\
				\int_{c}^{y}\sqrt{g(x,y)}Z_4(x,y,\nu)\mbf{u}_2(x,\nu)d\nu \\
				\int_{y}^{d}\sqrt{g(x,y)}Z_5(x,y,\nu)\mbf{u}_2(x,\nu)d\nu \\
				\int_{a}^{x}\int_{c}^{y}\sqrt{g(x,y)}Z_6(x,y,\theta,\nu)\mbf{u}_2(\theta,\nu)d\nu d\theta \\
				\int_{x}^{b}\int_{c}^{y}\sqrt{g(x,y)}Z_7(x,y,\theta,\nu)\mbf{u}_2(\theta,\nu)d\nu d\theta \\
				\int_{a}^{x}\int_{y}^{d}\sqrt{g(x,y)}Z_8(x,y,\theta,\nu)\mbf{u}_2(\theta,\nu)d\nu d\theta \\
				\int_{x}^{b}\int_{y}^{d}\sqrt{g(x,y)}Z_9(x,y,\theta,\nu)\mbf{u}_2(\theta,\nu)d\nu d\theta },	\nonumber
			\end{align}
			by the composition rules of PI operators (see also~\cite{jagt2021PIEArxiv}), it follows that $\mcl{P}[B]=\mcl{Z}^* P\mcl{Z}$. Since $P\geq 0$, we may split $P=\left[P^{\f{1}{2}}\right]^T P^{\f{1}{2}}$ for some $P^{\f{1}{2}}\in\R^{(9q+n_0)\times (9q+n_0)}$, and thus, for any $\mbf{u}\in Z^{\enn{u}}$,
			\begin{align*}
			\ip{\mbf{u}}{\mcl{P}[B]\mbf{u}}_{Z^{\enn{u}}}
			&=\ip{\mcl{Z}\mbf{u}}{P\mcl{Z}\mbf{u}}_{L_2}   \\
			&=\ip{P^{\frac{1}{2}}\mcl{Z}\mbf{u}}{P^{\frac{1}{2}}\mcl{Z}\mbf{u}}_{L_2}\geq 0,
			\end{align*}
			concluding the proof.
			
		\end{proof}

		\begin{figure*}[!t]
			
			\hrulefill
			{\footnotesize
				\begin{align}\label{eq:appx_pos_Nmats}
				&N_{00}(x,y)=g(x,y)[Z_{1}(x,y)]^T P_{11}Z_{1}(x,y)
				\nonumber\\
				& \nonumber\\
				&N_{10}(x,y,\theta)=g(x,y)[Z_{1}(x,y)]^T P_{12}Z_{2}(x,y,\theta) +g(\theta,y)[Z_{3}(\theta,y,x)]^T P_{31}Z_{1}(\theta,y)
				\nonumber\\
				&\qquad+\int_{x}^{b}g(\eta,y)[Z_{2}(\eta,y,x)]^T P_{22}Z_{2}(\eta,y,\theta)d\eta
				+\int_{\theta}^{x}g(\eta,y)[Z_{3}(\eta,y,x)]^T P_{32}Z_{2}(\eta,y,\theta)d\eta
				+\int_{a}^{\theta}g(\eta,y)[Z_{3}(\eta,y,x)]^T P_{33}Z_{3}(\eta,y,\theta)d\eta
				\nonumber\\
				&N_{20}(x,y,\theta)=[N_{10}(\theta,y,x)]^T
				\nonumber\\
				& \nonumber\\
				&N_{01}(x,y,\nu)=g(x,y)[Z_{1}(x,y)]^T P_{14}Z_{4}(x,y,\nu) +g(x,\nu)[Z_{5}(x,\nu,y)]^T P_{51}Z_{1}(x,\nu)
				\nonumber\\
				&\qquad+\int_{y}^{d}g(x,\mu)[Z_{4}(x,\mu,y)]^T P_{44}Z_{4}(x,\mu,\nu)d\mu +\int_{\nu}^{y}g(x,\mu)[Z_{5}(x,\mu,y)]^T P_{54}Z_{4}(x,\mu,\nu)d\mu
				+\int_{c}^{\nu}g(x,\mu)[Z_{5}(x,\mu,y)]^T P_{55}Z_{5}(x,\mu,\nu)d\mu
				\nonumber\\
				&N_{02}(x,y,\nu)=[N_{01}(x,\nu,y)]^T
				\nonumber\\
				& \nonumber\\
				&N_{11}(x,y,\theta,\nu)
				=g(x,y)[Z_{1}(x,y)]^T P_{16}Z_{6}(x,y,\theta,\nu) +g(\theta,\nu)[Z_{9}(\theta,x,\nu,y)]^T P_{91}Z_{1}(\theta,\nu)
				\nonumber\\
				&\qquad+g(x,\nu)[Z_{5}^{02}(x,\nu,y)]^T P_{52}Z_{2}(x,\theta,\nu)
				+g(\theta,y)[Z_{3}^{20}(\theta,y,x)]^T P_{34}Z_{4}(\theta,y,\nu)
				\nonumber\\
				&\qquad+\int_{x}^{b}g(\eta,y)[Z_{2}(\eta,y,x)]^T  P_{26}Z_{6}(\eta,y,\theta,\nu)]d\eta
				+\int_{\theta}^{x}g(\eta,y)[Z_{3}(\eta,y,x)]^T P_{36}Z_{6}(\eta,y,\theta,\nu)d\eta
				+\int_{a}^{\theta}g(\eta,y)[Z_{3}(\eta,y,x)]^T P_{37}Z_{7}(\eta,y,\theta,\nu)d\eta
				\nonumber\\
				&\qquad+\int_{x}^{b}g(\eta,\nu)[Z_{7}(\eta,\nu,x,y)]^T  P_{72}Z_{2}(\eta,\theta,\nu)]d\eta
				+\int_{\theta}^{x}g(\eta,\nu)[Z_{9}(\eta,\nu,x,y)]^T P_{92}Z_{2}(\eta,\theta,\nu)d\eta
				+\int_{a}^{\theta}g(\eta,\nu)[Z_{9}(\eta,\nu,x,y)]^T P_{93}Z_{3}(\eta,\theta,\nu)d\eta
				\nonumber\\
				&\qquad+\int_{y}^{d}g(x,\mu)[Z_{4}(x,\mu,y)]^T P_{46}Z_{6}(x,\mu,\theta,\nu)d\mu
				+\int_{\nu}^{y}g(x,\mu)[Z_{5}(x,\mu,y)]^T P_{56}Z_{6}(x,\mu,\theta,\nu)d\mu
				+\int_{c}^{\nu}g(x,\mu)[Z_{5}(x,\mu,y)]^T P_{58}Z_{8}(x,\mu,\theta,\nu)d\mu
				\nonumber\\
				&\qquad+\int_{y}^{d}g(\theta,\mu)[Z_{7}(\theta,\mu,x,y)]^T P_{74}Z_{4}(\theta,\mu,\nu)d\mu
				+\int_{\nu}^{y}g(\theta,\mu)[Z_{9}(\theta,\mu,x,y)]^T P_{94}Z_{4}(\theta,\mu,\nu)d\mu
				+\int_{c}^{\nu}g(\theta,\mu)[Z_{9}(\nu,x,\mu,y)]^T P_{95}Z_{5}(\theta,\mu,\nu)d\mu
				\nonumber\\
				&\qquad+\int_{x}^{b}\int_{y}^{d}g(\eta,\mu)[Z_{6}(\eta,\mu,x,y)]^T  P_{66}Z_{6}(\eta,\mu,\theta,\nu)]d\mu d\eta
				+\int_{\theta}^{x}\int_{y}^{d}g(\eta,\mu)[Z_{7}(\eta,\mu,x,y)]^T  P_{76}Z_{6}(\eta,\mu,\theta,\nu)]d\mu d\eta    \nonumber\\
				&\qquad\qquad+\int_{a}^{\theta}\int_{y}^{d}g(\eta,\mu)[Z_{7}(\eta,\mu,x,y)]^T  P_{77}Z_{7}(\eta,\mu,\theta,\nu)]d\mu d\eta
				+\int_{x}^{b}\int_{\nu}^{y}g(\eta,\mu)[Z_{8}(\eta,\mu,x,y)]^T  P_{86}Z_{6}(\eta,\mu,\theta,\nu)]d\mu d\eta \nonumber\\
				&\qquad\qquad\qquad+\int_{\theta}^{x}\int_{\nu}^{y}g(\eta,\mu)[Z_{9}(\eta,\mu,x,y)]^T  P_{96}Z_{6}(\eta,\mu,\theta,\nu)]d\mu d\eta
				+\int_{a}^{\theta}\int_{\nu}^{y}g(\eta,\mu)[Z_{9}(\eta,\mu,x,y)]^T  P_{97}Z_{7}(\eta,\mu,\theta,\nu)]d\mu d\eta    \nonumber\\
				&\qquad\qquad\qquad\qquad+\int_{x}^{b}\int_{c}^{\nu}g(\eta,\mu)[Z_{8}(\eta,\mu,x,y)]^T  P_{88}Z_{8}(\eta,\mu,\theta,\nu)]d\mu d\eta
				+\int_{\theta}^{x}\int_{c}^{\nu}g(\eta,\mu)[Z_{9}(\eta,\mu,x,y)]^T  P_{98}Z_{8}(\eta,\mu,\theta,\nu)]d\mu d\eta    \nonumber\\
				&\qquad\qquad\qquad\qquad\qquad+\int_{a}^{\theta}\int_{c}^{\nu}g(\eta,\mu)[Z_{9}(\eta,\mu,x,y)]^T  P_{99}Z_{9}(\eta,\mu,\theta,\nu)]d\mu d\eta
				\nonumber\\
				&N_{22}(x,y,\theta,\nu)=[N_{11}(\theta,x,\nu,y)]^T
				\nonumber\\
				& \nonumber\\
				&N_{21}(x,y,\theta,\nu)
				=g(x,y)[Z_{1}(x,y)]^T P_{17}Z_{7}(x,y,\theta,\nu) +g(\theta,\nu)[Z_{8}(\theta,x,\nu,y)]^T P_{81}Z_{1}(\theta,\nu)
				\nonumber\\
				&\qquad+g(x,\nu)[Z_{5}(x,\nu,y)]^T P_{53}Z_{3}(x,\theta,\nu)
				+g(\theta,y)[Z_{2}(\theta,y,x)]^T P_{24}Z_{4}(\theta,y,\nu)
				\nonumber\\
				&\qquad+\int_{\theta}^{b}g(\eta,y)[Z_{2}(\eta,y,x)]^T  P_{26}Z_{6}(\eta,y,\theta,\nu)]d\eta
				+\int_{x}^{\theta}g(\eta,y)[Z_{2}(\eta,y,x)]^T P_{28}Z_{8}(\eta,y,\theta,\nu)d\eta
				+\int_{a}^{x}g(\eta,y)[Z_{3}(\eta,y,x)]^T P_{37}Z_{7}(\eta,y,\theta,\nu)d\eta
				\nonumber\\
				&\qquad+\int_{\theta}^{b}g(\eta,\nu)[Z_{8}(\eta,\nu,x,y)]^T  P_{82}Z_{2}(\eta,\theta,\nu)]d\eta
				+\int_{x}^{\theta}g(\eta,\nu)[Z_{8}(\eta,\nu,x,y)]^T P_{83}Z_{3}(\eta,\theta,\nu)d\eta
				+\int_{a}^{x}g(\eta,\nu)[Z_{9}(\eta,\nu,x,y)]^T P_{93}Z_{3}(\eta,\theta,\nu)d\eta
				\nonumber\\
				&\qquad+\int_{y}^{d}g(\theta,\mu)[Z_{6}(\theta,\mu,x,y)]^T P_{64}Z_{4}(\theta,\mu,\nu)d\mu
				+\int_{\nu}^{y}g(\theta,\mu)[Z_{8}(\theta,\mu,x,y)]^T P_{84}Z_{4}(\theta,\mu,\nu)d\mu
				+\int_{c}^{\nu}g(\theta,\mu)[Z_{8}(\theta,\mu,x,y)]^T P_{85}Z_{5}(\theta,\mu,\nu)d\mu
				\nonumber\\
				&\qquad+\int_{y}^{d}g(x,\mu)[Z_{4}(x,\mu,y)]^T P_{47}Z_{7}(x,\mu,\theta,\nu)d\mu
				+\int_{\nu}^{y}g(x,\mu)[Z_{5}(x,\mu,y)]^T P_{57}Z_{7}(x,\mu,\theta,\nu)d\mu
				+\int_{c}^{\nu}g(x,\mu)[Z_{5}(x,\mu,y)]^T P_{59}Z_{9}(x,\mu,\theta,\nu)d\mu
				\nonumber\\
				&\qquad+\int_{\theta}^{b}\int_{y}^{d}g(\eta,\mu)[Z_{6}(\eta,\mu,x,y)]^T  P_{66}Z_{6}(\eta,\mu,\theta,\nu)]d\mu d\eta
				+\int_{x}^{\theta}\int_{y}^{d}g(\eta,\mu)[Z_{6}(\eta,\mu,x,y)]^T  P_{67}Z_{7}(\eta,\mu,\theta,\nu)]d\mu d\eta    \nonumber\\
				&\qquad\qquad+\int_{a}^{x}\int_{y}^{d}g(\eta,\mu)[Z_{7}(\eta,\mu,x,y)]^T  P_{77}Z_{7}(\eta,\mu,\theta,\nu)]d\mu d\eta
				+\int_{\theta}^{b}\int_{\nu}^{y}g(\eta,\mu)[Z_{8}(\eta,\mu,x,y)]^T  P_{86}Z_{6}(\eta,\mu,\theta,\nu)]d\mu d\eta    \nonumber\\
				&\qquad\qquad\qquad+\int_{x}^{\theta}\int_{\nu}^{y}g(\eta,\mu)[Z_{8}(\eta,\mu,x,y)]^T  P_{87}Z_{7}(\eta,\mu,\theta,\nu)]d\mu d\eta
				+\int_{a}^{x}\int_{\nu}^{y}g(\eta,\mu)[Z_{9}(\eta,\mu,x,y)]^T  P_{97}Z_{7}(\eta,\mu,\theta,\nu)]d\mu d\eta    \nonumber\\
				&\qquad\qquad\qquad\qquad+\int_{\theta}^{b}\int_{c}^{\nu}g(\eta,\mu)[Z_{8}(\eta,\mu,x,y)]^T  P_{88}Z_{8}(\eta,\mu,\theta,\nu)]d\mu d\eta
				+\int_{x}^{\theta}\int_{c}^{\nu}g(\eta,\mu)[Z_{9}(\eta,\mu,x,y)]^T  P_{98}Z_{8}(\eta,\mu,\theta,\nu)]d\mu d\eta    \nonumber\\
				&\qquad\qquad\qquad\qquad\qquad+\int_{a}^{x}\int_{c}^{\nu}g(\eta,\mu)[Z_{9}(\eta,\mu,x,y)]^T  P_{99}Z_{9}(\eta,\mu,\theta,\nu)]d\mu d\eta    \nonumber\\
				&N_{12}(x,y,\theta,\nu)=[N_{21}(\theta,x,\nu,y)]^T
				\end{align}
			}
			\hrulefill
			\caption{Parameters $N$ describing the positive PI operator $\mcl{P}[N]=\mcl{Z}^* P\mcl{Z}$ in Proposition~\ref{prop:appx_pos_PI}}
			\label{fig:appx_positive_parameters}
		\end{figure*}
		
	\end{appendices}

\end{document}